\documentclass[11pt,letterpaper]{amsart}

\usepackage{amsmath,amsthm,amssymb,graphics,color}
\usepackage{hyperref} 
\usepackage{xcolor}
\usepackage[english]{babel}
\usepackage{tikz}
\usepackage{bm}
\usepackage{yhmath}
\usepackage{longtable}

\usepackage{etex}
\usetikzlibrary{shapes,calc,arrows,through,intersections}
\usetikzlibrary{quotes,angles}
\usetikzlibrary{calc}
\usetikzlibrary{shapes}
\usetikzlibrary{shapes.geometric}
\tikzset{>=latex} 
\usepackage{pgfplots} 

\usepackage{subcaption}
\usepackage{paralist}
\usepackage{graphicx}
\usepackage{esvect}
\usepackage{float}
\usepackage{array}
\usepackage{esvect}
\usepackage{float}
\usepackage{array}
\usepackage{tensor}
\usepackage{amsthm}
\usepackage{amsfonts}
\usepackage{bbm}

\newcommand{\regphi}{\phi}
\let\regphi\phi

\newcommand{\C}{\mathbb{C}}
\newcommand{\Q}{\mathbb{Q}}
\newcommand{\Z}{\mathbb{Z}}
\newcommand{\N}{\mathbb{N}}

\newcommand{\R}{\mathbb{R}}

\newcommand{\T}{\mathbb{T}}

\newcommand{\Tr}{\operatorname{Tr}}

\newcommand{\Khalf}{\frac{K}{2}}
\newcommand{\Monehalf}{\frac{M_1}{2}}

\renewcommand{\phi}{\varphi}
\renewcommand{\theta}{\vartheta}
\renewcommand{\epsilon}{\varepsilon}

\newtheorem{theo}{Theorem}[section]
\newtheorem{prop}[theo]{Proposition}
\newtheorem{coro}[theo]{Corollary}
\newtheorem{lemm}[theo]{Lemma}
\newtheorem{conj}[theo]{Conjecture}

\theoremstyle{definition}
\newtheorem{ques}[theo]{Question}

\theoremstyle{definition}
\newtheorem{def1}[theo]{Definition}
\theoremstyle{remark}
\newtheorem{rema}[theo]{Remark}
\theoremstyle{definition}
\newtheorem{exam}[theo]{Example}

\newcommand{\nwc}{\newcommand}
\nwc{\Oph}{\operatorname{Op}_\hbar}
\nwc{\la}{\langle}
\nwc{\ra}{\rangle}

\nwc{\mf}{\mathbf} 
\nwc{\blds}{\boldsymbol} 
\nwc{\ml}{\mathcal} 

\newcommand{\eps}{\varepsilon}
\newcommand{\tr}{\operatorname{Tr}}

\newcommand{\wt}{\widetilde}
\newcommand{\wh}{\widehat}

\renewcommand{\d}{\partial}

\newcommand{\half}{{\frac{1}{2}}}

\renewcommand{\phi}{\varphi}

\newcommand{\lsp}{\operatorname{LSP}}
\newcommand{\mls}{\operatorname{MLS}}

\newcommand{\Y}{\mathcal{Y}}

\newcommand{\red}[1]{{\color{black}{#1}}}

\newcommand{\black}[1]{\color{black}}

\title[Marked Length Isospectral Compactness]{Compactness of Marked Length Isospectral Sets of Birkhoff Billiard Tables}
\date{}

\author{Amir Vig}
\address{Department of Mathematics, University of Toronto, 40 St George St., Toronto, ON, Canada M5S 2E4} \email{amir.vig@utoronto.ca}

\allowdisplaybreaks
\newcommand{\I}{\mathcal{I}}
\newcommand{\M}{\mathcal{M}}
\newcommand{\dtpm}{\delta_\pm}
\newcommand{\dtp}{\delta_+}

\newcommand{\dt}{\delta}

\newcommand{\hot}{\text{h.o.t.}}
\newcommand{\lot}{\text{l.o.t.}}
\newcommand{\ka}{\kappa}
\newcommand{\J}{\mathcal{J}}
\newcommand{\len}{\operatorname{length}}

\newcommand{\z}{\mathcal{Z}}
\newcommand{\q}{\mathcal{Q}}
\newcommand{\RR}{\mathcal{R}}
\newcommand{\WW}{\mathcal{W}}

\newcommand{\Hm}{\mathcal{H}_{2m - 2, 3}}
\newcommand{\zr}{\zeta_\text{Riem}}
\newcommand{\msfz}{Z}

\allowdisplaybreaks

\begin{document}
	\maketitle
	\begin{abstract}
		We prove that equivalence classes of marked length isospectral Birkhoff billiard tables are compact in the $C^\infty$-topology, analogous to the Laplace spectral results in \cite{MIC}, \cite{POS2} and \cite{POS1}. To do so, we derive a hierarchical structure for the integral invariants of Marvizi and Melrose \cite{MM}, or equivalently the coefficients of a caustic length-Lazutkin parameter expansion, which are in turn algebraically equivalent to the Taylor coefficients of Mather's $\beta$-function (also called the mean minimal action). Under a generically satisfied noncoincidence condition, these are also Laplace spectral invariants. As a byproduct, we obtain an independent proof of the compactness of Laplace isospectral sets for \red{generic} strictly convex planar billiard tables. The proof of the structure theorem uses an interpolating Hamiltonian for nearly glancing billiard orbits and some analytic number theory to compute its Taylor coefficients.
	\end{abstract}

	\section{Main Results}
	\label{sec: Main Results}
		Let $\Omega \subset \R^2$ be a smooth, bounded and strictly \red{(strongly)} convex domain. Such a domain is uniquely characterized \red{up to Euclidean isometry} by the curvature of its boundary, which is a strictly positive function. Billiard orbits are concatenations of oriented straight line segments in $\Omega$ which make equal angles when reflected at the boundary. To each periodic billiard orbit, we can associate a rotation number $p/q$, where $p$ is the winding number and $q$ is the bounce number. The marked length spectrum of $\Omega$ is a map $\mls_\Omega : \Q \cap (0,\frac{1}{2}] \to \R$ which associates to each rational $p/q$ in reduced form, the maximal length of a periodic billiard orbit with rotation number $p/q$. We have the following dynamical analogue of the Laplace spectral results in \cite{MIC}, \cite{POS2}, and \cite{POS1}:
		
	\begin{theo}\label{cpt}
		For all $\Omega \subset \R^2$ with $\d \Omega$ smooth, bounded and strictly convex, the marked length isospectral set containing $\Omega$ is compact in the $C^\infty$-topology on boundary curvatures.
	\end{theo}
	
	In particular, there are strictly positive lower bounds \red{and finite upper bounds on} the curvature within the isospectral set (see Lemma \ref{kappa0}), which prevents both asymptotic flattening \red{and the formation of corners in limiting domains}. Under the generically satisfied noncoincidence condition in \cite{MM}, this also applies to the Laplace isospectral set and provides an independent proof of the results of Melrose, Osgood, Phillips and Sarnak (see Section \ref{Connection with Laplace spectrum} below). The marked length spectrum is encoded by Mather's $\beta$-function, also called the mean minimal action, which for each rational $0 < \frac{p}{q} \leq \half$, returns $- \frac{1}{q}$ times the maximum length of periodic orbits having rotation number $\frac{p}{q}$. \red{It has a continuous extension to all $\omega \in [0,\half]$ and turns out to be regular enough that there exists a formal Taylor expansion near zero (as well as certain Diophantine rotation numbers). Its Taylor coefficients are marked length spectral invariants and the compactness in Theorem \ref{cpt} is proved using only the $\beta$-function Taylor coefficients at $\omega = 0$; therefore, compactness extends to the class of tables which are not necessarily marked length isospectral but instead have the same $\beta$-function expansion at $0$. In fact, we show in Example \ref{ex: nonisometric MM} below that there exist one-parameter families of non-isometric domains for which the corresponding $\beta$-functions have identical Taylor expansions at $0$.}
	\\
	\\
	It is well known, as a consequence of the KAM theorem, that there exist sequences of caustics with Diophantine rotation numbers in any neighborhood of the boundary (see \cite{Lazutkin}). The lengths of such caustics \red{have} an asymptotic expansion as the Lazutkin parameter $Q \to 0$ (see Definition \ref{laz}), with coefficients given by boundary integrals of algebraic functions in the curvature jet. It was shown in \cite{MM} and \cite{Amiran} that these coefficients are marked length spectral invariants and are in one-to-one algebraic correspondence with the Taylor coefficients of Mather's $\beta$-function (see also \cite{Si}). To describe their structure, we introduce the following terminology.
	
	\begin{def1}\label{diffdeg}
		\red{Let $u: \d \Omega \to \R$ be any smooth, positive function of one variable and denote its $i^{\text{th}}$ derivative in arclength coordinates by $u_i = u^{(i)}$. The \textbf{differential degree} of a polynomial in the variables $\{u^{\pm \frac{1}{3}}, u_1, \cdots, u_m\}$ is defined to be the maximum over all constituent monomials $c u^{\frac{p_0}{3}} u_1^{p_1} \cdots u_m^{p_m}$, $c \in \R$ of the quantities
		\begin{align*}
			\text{deg}_\d \left(c u^{\frac{p_0}{3}} u_1^{p_1} \cdots u_m^{p_m}\right) = \sum_{i = 0}^m i p_i.
		\end{align*}}
	\end{def1}
	
	\begin{theo}\label{main}

		\red{Let $\Gamma_Q \subset \Omega$ denote a family of convex caustics which are close to the boundary of $\Omega$, indexed by their Lazutkin parameters $Q$,} and denote by $\kappa$ the curvature of $\d \Omega$ and by $\kappa_k$ its $k^{\text{th}}$ derivative \red{with respect to the} arclength \red{parameter $s$}. \red{Then, the lengths of such caustics admit an asymptotic expansion of the form}
		\begin{align*}
			|\Gamma_{\red{Q}}| \sim |\d \Omega| + \sum_{m = 1}^{\infty} \frac{1}{m!} \I_{m} \left(\frac{3}{2}Q\right)^{2m/3},
		\end{align*}
		\red{where} $\I_m$ are \red{marked length spectral invariants, given by} integrals of curvature polynomials:
		\begin{align*}
			\red{\I_1 = - \frac{1}{2} \int_{\d \Omega} \ka^{\frac{2}{3}}(s) ds,} \qquad
			\mathcal{I}_{m} =  \int_{\d \Omega}  \mathcal{P}_m (\ka^{\pm \frac{1}{3}}, \ka_1, \cdots, \ka_{m-1}) ds, \qquad \red{(m \geq 2)}.
		\end{align*}
		Here, $\mathcal{P}_m \in \R\left[\ka^{\pm \frac{1}{3}}, \ka_1, \cdots, \ka_{m-1}\right]$ has differential degree $2m - 2$ and the highest derivatives in $\mathcal{P}_m$ appear quadratically in the form
		\begin{align*}
			\mathcal{P}_m = c_m \ka^{-4m /3} \ka_{m-1}^2 + \mathcal{R}_m,
		\end{align*}
		\red{where
		\begin{align*}
			c_m = \frac{(2m + 3)}{\pi^{2m + 2}}  m! \zr(2m+2),
		\end{align*}
		with $\zr$ being the Riemann $\zeta$-function.} $\mathcal{R}_m \in \R\left[\ka^{\pm \frac{1}{3}}, \cdots, \ka_{m-2}\right]$ \red{is a remainder of} differential degree \red{at most} $2m - 2$ which \red{depends on no more than $m-2$ derivatives of $\ka$}.
	\end{theo}
	
	The \red{coefficients} $\I_m$ are \red{called} \textbf{Marvizi-Melrose invariants} (see Definition \ref{action}), the first two of which were computed in \cite{MM} and the subsequent two in \cite{Sorr15}. \red{As a simple consequence of Theorem \ref{main}, we have the following important example of one-parameter families of nonisometric domains with the same Marvizi-Melrose invariants:}
	
	\begin{exam}[One-parameter families of non-isometric domains with the same Marvizi-Melrose invariants]
		\label{ex: nonisometric MM}
		\red{Consider the unit circle in $\R^2$. Now add two ``bumps'' with disjoint support, sufficiently $C^\infty$-small so that the resulting domain is still strictly convex. Sliding the bumps around the circular portion of the boundary in $\R^2$ so that their supports remain disjoint, it is clear that as the distance between them varies, the resulting domains will in general \textit{not} be isometric. However, Theorem \ref{main} tells us that the Marvizi-Melrose invariants $\I_k$ are integrals of local densities (algebraic functions in the jet of the curvature). Hence, the resulting non-isometric domains have \textit{identical} Marvizi-Melrose invariants.}
	\end{exam}
	
	\red{The example above shows that for an infinite dimensional family of bumps, the linearization of what one might call the ``Marvizi-Melrose operator'' $\Omega \mapsto (\I_k)_{k \in \N}$ has a nontrivial kernel. The existence of such families of domains which are indistinguishable from their Marvizi-Melrose invariants alone precludes their direct use in rigidity problems and suggests that outside of the compactness in Theorem \ref{cpt} and the extremization of invariants in \cite{MM}, there are limitations in the use of these invariants for inverse spectral problems. The same applies to the Taylor coefficients of Mather's $\beta$-function (see Definition \ref{def: MB and MM invariants} below). {The precise form in Theorem \ref{main} is not needed for this simple example. That the invariants $\I_k$ are integrals of polynomials in the jet of $\kappa$ (with coefficients which are constant multiples of powers of $\ka^{\frac{1}{3}}$) was already known from the work of Marvizi and Melrose \cite{MM} (Proposition 4.9). Nonuniqueness was also already known from the recent work of Buhovsky and Kaloshin \cite{Buhovsky}, who give a pair (but not a one-parameter family) of non-isometric domains with the same Marvizi-Melrose invariants.}}
	\\
	\\
	The main difficulty in proving Theorem \ref{main} {is} the inversion of a highly nonlinear map sending polynomials in the curvature jet to rational functions of the Taylor coefficients of an interpolating Hamiltonian. Precise calculation of the constants, which we need to be nonzero, requires tools from analytic number theory. A complete description of the invariants would involve sorting through even more terms, in particular those arising from large powers of a certain vector field when considered as a differential operator. \red{We establish a method for computing these coefficients in Section \ref{Small lambda asymptotics}, after which} \red{we transform the integrands of each Marvizi-Melrose invariant into a ``reduced form'' (see Definition \ref{def:excess differential gap reduced form}), which leads to} the structure in Theorem \ref{main}. The approach below uses a different algorithm than the one presented in \cite{Sorr15} and works for a {more general} class of dynamical systems which admit an interpolating Hamiltonian.

	\subsection{Outline}
	We begin with a literature review in Section \ref{background} and describe the connection between the length and Laplace spectra. In Section \ref{Billiards}, we review symplectic aspects of the billiard map and introduce an interpolating Hamiltonian for nearly glancing orbits. In particular, we give a formula for the integral invariants $\I_k$ in terms of it. In Section \ref{MIC}, we show how Theorem \ref{main} implies Theorem \ref{cpt}. Section \ref{Geometric and combinatorial preliminaries} deals with algebraic aspects of integration by parts and provides an algorithm for reducing the number of derivatives appearing in a polynomial in the curvature jet. We show that any such polynomial of differential degree $D$ in the jet of $\ka$ is, when multiplied by the arclength one-form $ds$, cohomologous to another one-form whose $ds$-coefficient has at most the same differential degree and whose $\ka$-derivatives {have order less than or equal to} $\lfloor D/2 \rfloor$. In Section \ref{Small lambda asymptotics}, we compute the leading order asymptotics of the billiard map near glancing directions in two different ways. One is geometric and uses curvature coordinates in Section \ref{sec: Computing AM Geometrically}. The other, in Section \ref{sec: computing AM Algebraically}, is algebraic in nature and deals with the combinatorics of large powers of a Hamiltonian vector field. Equivalently, this can be rephrased in terms of Lie series or iterated Poisson brackets. In Section \ref{Integral invariants}, we compute explicitly the Taylor coefficients of an interpolating Hamiltonian in terms of the curvature jet. This is done by finding an infinite order recursion relation for the highest derivatives, putting them into a generating function and solving an ordinary differential equation. This yields a surprising relationship with Bernoulli numbers and the Riemann zeta function. We then integrate by parts and keep track of all constants \red{to derive an explicit formula for $c_m$, again using generating functions}, which completes the proof of Theorem \ref{main} \red{and hence Theorem \ref{cpt}.} \red{After introducing more technical concepts in Sections \ref{Billiards} and \ref{Geometric and combinatorial preliminaries}, we provide a more detailed outline of the proof of Theorem \ref{main} in Section \ref{subsec: Outline of the algorithm}.}

	\tableofcontents

	\section{Background}\label{background}
	
	\subsection{Marked length spectrum}
	The marked length spectrum is a natural object to study in the context of both closed manifolds and domains with boundary. In the boundaryless case, the marked length spectrum is a function which returns for each free homotopy class, the {minimal} length of a closed geodesic belonging to that class. The unmarked length spectrum (without marking by homotopy classes or rotation number) is a much harder object to study. In either case, the natural inverse problem which arises is to determine the shape of a domain (metric, boundary curve, etc...) from knowledge of its marked or unmarked length spectrum. For planar billiard tables, both spectra are intimately related to the so-called Birkhoff-{Poritsky} Conjecture, which postulates that only ellipses have completely integrable billiard dynamics {(see Section \ref{subsec:Length spectra})}.
	
	\subsubsection{Marked length spectrum for billiards}
	In the case of {strictly convex} billiard tables, or more generally monotone twist maps, one can study the marked length spectrum through Mather's $\beta$-function (see Definition \ref{mbf}), which for rational $\omega = p/q$ gives the mean minimal action of orbits having rotation number $\omega$. It is a complete marked length spectral invariant. The first four coefficients were derived using symbolic computer algebra in \cite{Sorr15}, where one can also find a discussion of local integrability, the Birkhoff-{Poritsky} Conjecture \red{(Conjectures \ref{conj:BP1} and \ref{conj:BP2} below)} and its relationship to the regularity of Mather's $\beta$-function. It was also shown there that disks are uniquely determined by their marked length spectrum (in fact, {by} only the first two Taylor coefficients of $\beta$). {In fact,} the first four coefficients, or rather their algebraically equivalent counterparts in terms of the caustic length-Lazutkin expansion, are {shown in Section \ref{MIC} below to be} all that is needed to derive $C^2$-compactness of isospectral sets. The structure {in} Theorem \ref{main} allows us to upgrade this to $C^\infty$-compactness, which {is also} proved in Section \ref{MIC} below. The algebraic equivalence of Mather's $\beta$-function coefficients and the caustic length-Lazutkin parameter coefficients was proved in \cite{MM} {and} \cite{Amiran}; an {explicit} formula for one in terms of the other is conjectured in \cite{KK}. In \cite{GuMeCohomological}, it was shown that the unmarked length spectrum is also a symplectic invariant. There has been much recent progress on the Birkhoff-{Poritsky} Conjecture and the {inverse} marked length spectral problem for convex planar domains; see \cite{koval2025local}, \cite{KovalGevrey}, \cite{KaSo16}, \cite{KaAvDS16}, \cite{KaDSWe17}, \cite{KaHuSo18}, \cite{KaloshinZhangRationalCaustics}, \cite{KaSoHuNearlyCircular}, and \cite{Popov1994}.  \red{More recently, there have been several parallel developments for projective, symplectic, and outer (dual) billiards \cite{Tabachnikov}, \cite{TabachnikovDual}, \cite{Tabachnikovprojectivebilliards}, \cite{BaraccoBernardiNardi}, \cite{BBN}, \cite{FSV25}.} We refer the readers to the manuscripts \cite{Tabachnikov}, \cite{Treschev}, \cite{Katok}, and \cite{Si} \red{for more comprehensive surveys of the literature on convex billiards.} \red{For related results on} the marked length spectrum for chaotic billiards, see \cite{DeSimoiKaloshinLeguil1} and \cite{BalintDeSimoiKaloshinLeguil}.

	\subsubsection{Marked length spectrum for closed manifolds} \red{In the context of geodesic flows on closed Riemannian manifolds, there are various notions of the marked length spectrum. A folklore conjecture asserts that the only metrics with integrable geodesic flow on the two-torus are Liouville, i.e., those which can be written as $(f(x) + g(y))(dx^2 + dy^2)$ for some smooth functions $f$ and $g$. This is formally analogous to the Birkhoff-Poritsky Conjecture (Conjectures \ref{conj:BP1} and \ref{conj:BP2}). There has been much recent work on this conjecture \cite{ABM}, \cite{CorsiKaloshin2018}, \cite{Henheik}, \cite{HKLV}.} 	\red{Even more recently, Abbondandolo and Mazzucchelli have defined the notion of a marked length spectrum for metrics of revolution on the two-sphere and proved that isospectral metrics have conjugate geodesic flows \cite{AbbondandoloMazzucchelli}. Furthermore, they show that each marked length isospectral class has a unique $\Z_2$-symmetric representative.}
	\\
	\\
	In the context of closed manifolds with Anosov geodesic flow, it was conjectured in \cite{KatokBurns85} that the marked length spectrum uniquely determines a Riemannian metric. There have been several recent advances in this direction; see for example \cite{GuillarmouLefeuvre}, \cite{GLP25} and \cite{KarenButt}. There is also a nice survey by Amie Wilkinson on the subject \cite{wilkinson2012lectures}. It was shown in \cite{Vigneras} that \red{\textit{unmarked}} length isospectral surfaces need not be isometric.

	\subsection{Laplace spectrum}
	Dual to the length spectrum is the Laplace spectrum, which consists of eigenvalues of the Laplace-Beltrami operator:
	\begin{align}\label{PSF}
		\begin{cases}
			- \Delta u = \lambda^2 u,\\
			Bu = 0.
		\end{cases}
	\end{align}
	Here, $B$ is a boundary operator encoding Dirichlet, Neumann, Robin or mixed boundary conditions. If the manifold is closed, there is of course no boundary operator needed. The connection with the length spectrum is given by the \textit{Poisson relation}:
	\begin{align}\label{PoissonRelation}
		\operatorname{SingSupp} \,\,\tr \left(\cos t \sqrt{- \Delta}\right) \subset \overline{\lsp(\Omega)} \cup \Z |\d \Omega|,
	\end{align}
	where the left-hand side is the singular support of the even wave trace and the right-hand side is the closure of the (unmarked) length spectrum (Definition \ref{def:LSP}). The wave trace is to be interpreted in the sense of distributions. This beautiful formula was first derived by Poisson for flat tori, where it reduces to basic Fourier analysis. It was later studied in the context of closed hyperbolic surfaces, in which case one has the Selberg trace formula. This was further generalized by Duistermaat and Guillemin in their celebrated work \cite{DuGu75}, extending the trace formula to arbitrary smooth, closed manifolds. For domains with boundary, the Poisson {relation} was first introduced by Anderson and Melrose \cite{AndersonMelrose}, {with a trace formula derived} later by Guillemin and Melrose \cite{GuMe79b}. Whether or not the inclusion can be made strict has been the subject of much recent speculation. In \cite{KKV} and \cite{KVSilentOrbits25}, together with Vadim Kaloshin and Illya Koval, we show that within a finite degree of regularity, the inclusion can be made strict for Birkhoff billiard tables. For a residual set of boundaries, the length spectrum is simple and all periodic orbits are nondegenerate {(\cite{PeSt17}, Theorems 6.2.3 and 6.4.1)}, which prevents these types of cancellations and implies an \textit{equality} in \eqref{PoissonRelation}.
	\\
	\\
	\begin{figure}
		\begin{center}
			\includegraphics[scale = 0.4]{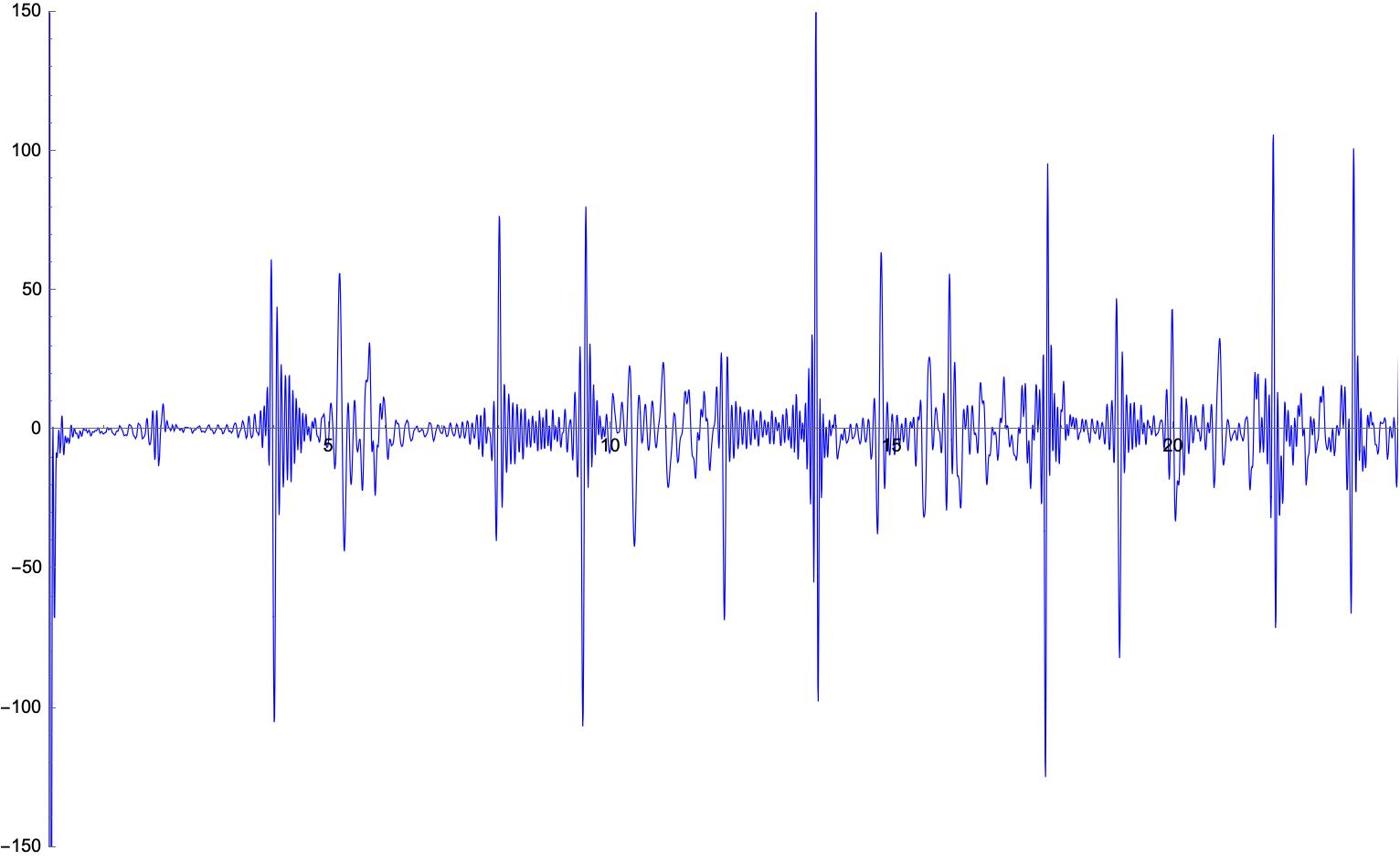}
		\end{center}
		\caption{\red{A plot of the truncated wave trace $\tr \cos t \sqrt{-\Delta}$ for the Laplacian on the unit disk in $\R^2$ with Dirichlet boundary conditions using the first $650$ eigenvalues, computed from zeroes of Bessel functions. The variable $t$ along the horizontal axis is time, and the spikes correspond to the lengths of periodic billiard orbits. Notice that the first spike appears to be centered at $t = 4$, which corresponds exactly to the length of a minimal bouncing ball (period $2$) orbit along the diameters.}}
		\label{fig: Poisson relation 650}
	\end{figure}
	Compactness of the Laplace isospectral set was first studied by Melrose for smooth planar domains in \cite{MIC}, using the algebraic structure of heat invariants. One advantage of that paper is that no convexity was assumed, but consequently the precompactness derived there did not exclude the possibility of domains degenerating to a pinched, nonsmooth ``dumbbell'' within the isospectral set. This was later addressed by Osgood, Phillips and Sarnak in \cite{POS2} and \cite{POS1}, where compactness was proven for both closed surfaces and domains with boundary. The approach in those papers was via an analysis of the spectral zeta function, similar to the Selberg zeta function. One important feature of Theorem  \ref{cpt} in this paper (see also Lemma \ref{kappa0}) is that it also excludes the possibility of degeneration even within the class of strictly convex domains; i.e. there are uniform positive lower bounds \red{and finite upper bounds} on the curvature within any marked length isospectral set. Extensions of this compactness to isospectral (and isoresonant) potentials can be found in \cite{Donnelly05} and \cite{Hislop18}.
	\\
	\\
	One curiosity is that there does not seem to be an existing holomorphic analogue of the dynamical zeta function in the context of Birkhoff billiards, nor a $\beta$-type function (or mean minimal action, see Definition \ref{mbf}) of a single variable in the context of closed manifolds. If one could prove compactness in the Anosov case, then by the results of \cite{GLP25}, \cite{GuillarmouLefeuvre}, one would obtain finiteness of the marked length isospectral set. Similarly, if rigidity could be proved in the planar billiards case via some kind of zeta function, one would obtain finiteness of the marked length isospectral set. However, Example \ref{ex: nonisometric MM} above and the results in \cite{Buhovsky} show that the {Taylor} coefficients of the mean minimal action {at zero} do not uniquely determine a billiard table. 
	\\
	\\
	Again in the context of strictly convex, smooth billiard tables, Marvizi and Melrose showed, under the noncoincidence condition that $|\d \Omega|$ is not a limit point from below of the lengths of periodic orbits which have winding number at least two, that the coefficients of Mather's $\beta$-function are also Laplace spectral invariants. They introduced a new family of integral invariants via a so-called interpolating Hamiltonian, which are equal (up to nonzero universal multiplicative constants) to the caustic length-Lazutkin parameter expansion coefficients in Theorem \ref{main} and are amenable to direct computation. The noncoincidence condition is known to hold for a residual, in particular, dense, set of domains in the $C^\infty$-topology on boundary curvatures, including $C^1$-open neighborhoods of disks, ellipses and analytic domains. Using the first two invariants, they constructed a two-parameter family of spectrally determined domains within this class. One family of spectrally determined domains has curvature function given by an elliptic integral, which is tantalizingly close to being that of an ellipse. For more on the subject of determining a convex billiard table from its Laplace spectrum, we refer the reader to the surveys \cite{ZelditchSurvey2} and \cite{ZelditchSurvey2014}. Recent results on hearing the shape of a drum can be found in \cite{HeZe19}, \cite{Vig21}, and \cite{Zelditch4}.

	\section{Billiards}\label{Billiards}
	
	Denote by $\Omega$ a bounded and strictly \footnote{\red{Strictly speaking, a domain can have flat boundary at some points and still be strictly convex \textit{as a set} in $\R^2$. Some authors instead refer to domains with strictly positive curvature as being \textit{strongly} convex, reserving \textit{strictly} convex for domains with nonnegative curvature. We follow the technically incorrect convention in the field by referring to such domains as strictly convex.}} \red{(strongly)} convex region in $\R^2$ with smooth boundary. This means that the curvature of $\d \Omega$ is a strictly positive function. \red{Such domains are called \textbf{Birkhoff billiard tables}; we will denote the set of all Birkhoff billiard tables by $\mathcal{B}$.} The billiard map is defined on the coball bundle of the boundary $B^* \d \Omega = \{(x, \xi) \in T^*\d \Omega : |\xi| < 1\}$, which can be identified with the inward or outward parts of the circle (cosphere) bundle $S_{\d \Omega}^* \R^2$, via the natural orthogonal projection maps. \red{By means of an arclength parametrization $x: (\R/ \ell \Z)_s \to \R^2$ of the boundary,} we can also identify $B^* \partial \Omega$ with $({\R}\slash {\ell \Z})_s \times (0, \pi)_\phi$ where $\ell=|\partial \Omega|$ is the length of the boundary and $\phi$ is the angle made by an inward $(+)$ or outward $(-)$ pointing unit covector $\wh \xi_\pm \in S_{x(s)}^* \R^2$ and the positively oriented tangent line at the point $x(s)$. Instead of $(s,\phi)$, we will use symplectic coordinates $(s, \sigma)$ with $\sigma = \cos \phi$ so that $\xi d x =  \sigma ds$. \red{Equivalently, $\wh \xi_\pm \cdot x'(s) = \cos \phi = \sigma$ is the symplectic dual variable to $s$ (see Figure \ref{Billiard Table}).}
	
	\begin{def1} \label{def:dtpm billiard maps}
		If $(x,\xi) \in B^* \d \Omega$ is mapped to the inward $(+)$ (resp. outward $(-)$) pointing covector $(x, \wh \xi_\pm) \in S_{\d \Omega}^* \R^2$ (the unit circle bundle over the boundary) under the inverse projection map, we define the \textbf{billiard maps} to be
		$$
		\delta_{\pm}(x,\xi) = (x^\pm,\xi^\pm),
		$$
		where ${(x^\pm,\xi^\pm)}$ is the projection onto the coball bundle of the parallel transported unit covector $\wh \xi_\pm$ along the line containing $x$ in the direction of $\pm \wh \xi_\pm$ at the subsequent intersection point with $\d \Omega$. The maps $\delta_{\pm}^n$ are defined via iteration and it is clear that \red{$\delta_\pm^{n} = (\delta_\mp^{n})^{-1}$} for each $n \in \Z$. See Figure \ref{Billiard Table}.
	\end{def1}
	
	\red{We will use the natural symplectic coordinates $(s,\sigma)$ on $B^* \d \Omega$ and by an abuse of notation, continue to write $\delta_\pm(s, \sigma)$ for the billiard maps $(s,\sigma) \mapsto (s^{\pm}, \sigma^{\pm})$ in these coordinates.} A point $\red{P = (s, \sigma) \in \R / \ell \Z \times (-1,1)}$ is called $q$-periodic ($q \geq 2$) if $\dt_\pm^q(P)=P$. We define the \textbf{rotation number} of a $q$-periodic orbit $\gamma$ to be $\omega(\gamma)= \frac{p}{q}$, where $p$ is the winding number of $\gamma$ which we now define. There exists a unique lift ${\wt \delta_{+} }$ of the map ${\dt_+}$ to the closure of the universal cover $\R_{\wt s} \times [-1,1]_{\wt \sigma}$ which is continuous and satisfies
		\begin{itemize}
			\item ${\wt \dt_+}(\wt s + \ell, \wt \sigma) = {\wt \dt_+}(\wt s, \wt \sigma) + (\ell,0),$
			\\
			\item $ {\wt \dt_+} (\wt s, 1) = (\wt s, 1)$,
			\\
			\item $\wt \dt_+ (\wt s, -1) = (\wt s + \ell, -1)$,
			\\
			\item ${\wt \dt_-} =  {\wt \dt_+^{-1}}$.
		\end{itemize}
	Given this normalization, for any point $(s,\sigma) \in \R/\ell \Z \times [-1,1]$ belonging to a $q$-periodic orbit of $\dt_+$, we see that ${\wt \dt_+ }^q(\wt s, \wt \sigma) = (\wt s + p \ell, \wt \sigma)$ for some $p \in \Z$. This $p$ is defined to be the winding number of the orbit $\gamma$ generated by $(s,\sigma) \in \overline{B^* \d \Omega}$. \red{For orbits $(s_q)_{q \in \Z}$ which are not periodic, we can sometimes still define a rotation number by
		\begin{align}
			\label{eq: general rotation number}
			\omega = \lim_{q \to \infty} \frac{\wt s_q}{\ell q},
		\end{align}
		for any lift $(\wt s_q)_{q \in \Z}$ of the orbit to $\R$, assuming that this limit exists. Orbits of a fixed winding number $p$ and period $q$ are said to be of \textbf{type-$(p,q)$}.}

	\begin{def1}
		\red{The set $\{\sigma = \pm 1\} \subset \overline{B^* \d \Omega}$ is called the \textbf{glancing set.} Billiard orbits} which are nearly tangent to the boundary, having qualitatively small rotation number (depending on the context) are called \textbf{nearly glancing}.
	\end{def1}
	
	\begin{rema}
		\red{The terminology ``nearly glancing'' comes from the theory of glancing, gliding, and grazing rays, which are orbits for the associated billiard \textit{flow} in a domain with boundary, whether it is convex or not (see \cite{MelroseTaylor}). Inside of a convex billiard table, there are no glancing rays, but orbits of rotation number $1/q$ for large $q$ are \textit{approximations} of {glancing} rays in the ambient space $\R^2$.}
	\end{rema}
	
		\begin{figure}
			\centering
			\includegraphics[scale = 1]{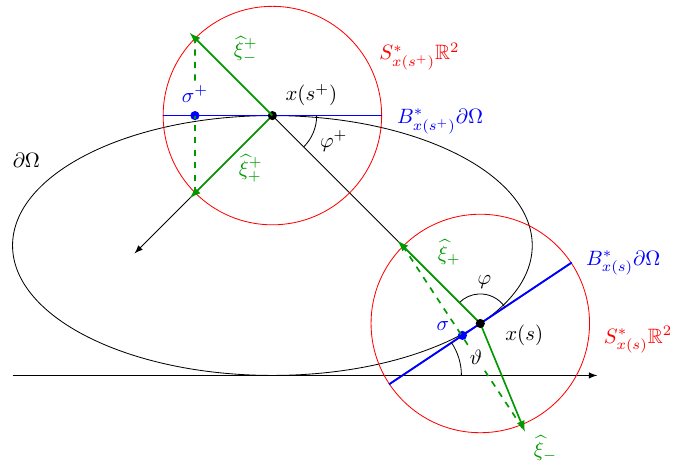}
			\caption{\red{With respect to the symplectic coordinates $(s, \sigma)$,} the billiard map $\delta_{+}$ sends $(s, \sigma) \mapsto (s^+, \sigma^+)$ and preserves the symplectic form $d\sigma \wedge ds$ on \red{the coball bundle} $B^*\d \Omega$, which is the (co)tangential projection of $S_{\d  \Omega}^*\R^2$ (red) \red{onto the cotangent bundle $T^* \d \Omega$.} \red{The fibers of $B^*\d \Omega$ at $x(s)$ and $x(s^+)$ are colored in blue}. \red{The inward and outward pointing covectors $\wh \xi_+$ and $\wh \xi_-$, both in green, project to $(x(s), \xi)$, where $\xi$ is the boundary covector corresponding to $\sigma = \cos \phi$. The pair $(x,\xi)$ (equivalently $(s,\sigma)$) is then mapped to $(x^+, \xi^+)$ (or $(s^+, \sigma^+)$), which corresponds via the inverse projection to the inward (resp. outward) pointing covector $\wh \xi_+^+$ (resp. $\wh \xi_-^+$), also in green. The angle $\theta$ made between the horizontal axis and the tangent line to $\d \Omega$ at $x(s)$ is called a \textit{curvature coordinate} and is described in more detail in Section \ref{curvecoord}.}}
			\label{Billiard Table}
		\end{figure}

	\subsection{Properties of the billiard map:}

	Again \red{denote by} $x: \R/\ell \Z \to \R^2$ an arclength parametrization of the boundary, then there exists a generating function $h$ for $\delta_+$: 
	\begin{align}\label{eq: generating functions}
		h(s, s^+) = - |x(s) - x(s^+)|.
	\end{align}
	If $x(s), x(s^+) \in \d \Omega$ are connected by a straight line making angles $\phi, \phi^+$ with the tangent lines at $x(s)$ and $x(s^+)$ respectively, then
	\begin{align*}
		\begin{cases}
			\d_s h =  \cos \phi = \sigma,\\
			\d_{s^+} h = - \cos \phi^+ = - \sigma^+.
		\end{cases}
	\end{align*}
	Here we tabulate some important properties of the billiard map. {We focus on $\dt_+$, but analogous statements apply to $\dt_-$.}
	\begin{itemize}
		\item $\dt_\pm$ is \textbf{exact symplectic}, meaning that it preserves the $2$-form ${d(\sigma ds) =} -\sin \phi d \phi \wedge ds$ and

		\begin{align}
			\label{eq: exact symplectic}
			\sigma^+ ds^+ - \sigma ds = - dh(s, s^+).
		\end{align}

		\item $\dtpm$ is {smooth} on $B^* \d \Omega$ and extends continuously up to the boundary, with square-root-type singularity there.
		
		\item $\dt_+$ satisfies the \textbf{monotone twist condition}:

		\red{$$
		\frac{\d s^+}{\d \sigma} < 0.
		$$
		In other words, for a fixed $s$, decreasing $\sigma$ (equivalently, increasing $\phi$), causes the second impact point to wind monotonically clockwise around $\d \Omega$, resulting in a \textit{twist} of the phase space $B^* \d \Omega$. In view of \eqref{eq: exact symplectic}, this is equivalent to
		$$
		\frac{\d^2 h}{\d s \d s^+} < 0.
		$$
		}
		The \textbf{twist interval} of $\dt_+$ is $[0,1]$, coming from the formulas
		\red{
		\begin{align*}
			0 =& \frac{\pi_1 (\wt \dt_+(\wt s, 1)) - \wt s}{\ell},\\
			1 =& \frac{\pi_1 (\wt \dt_+(\wt s, -1)) -\wt s}{\ell}.
		\end{align*}
	}
		\item Billiard orbits \red{$(\cdots, \dt_+^{-1}(s, \sigma), (s,\sigma), \delta_{+}(s, \sigma), \dt_+^2(s, \sigma), \cdots)$} correspond to critical points of the action functional
		\begin{align*}
			\sum_{i \in \Z} h(s_i, s_{i+1}),  \qquad (s_j, \sigma_j) = \dt_+^j(s, \sigma),
		\end{align*}
		in the sense that the points $s_i$ are critical on each finite segment with fixed endpoints, starting at $s_{N}$ and terminating at $s_{M}$ for any $N, M \in \Z$.
	\end{itemize}

	\subsection{Elements of Aubry-Mather theory}
	\red{The existence of periodic orbits for any given rational rotation number was postulated by Poincar\'e and proved by} Birkhoff \cite{PoincareSuruntheoremedegeometrie}, \cite{BirkhoffPoincare}. \red{Refinements and generalizations of this led to the development of Aubry-Mather theory, which we now describe; it will be referenced freely throughout the paper. For more details, we refer the reader to \cite{Aubry}, \cite{Mather}, \cite{MatherForni} and \cite{BirkhoffSurfaceTransformations}.}
	\begin{theo}[\cite{Si}, Theorem 1.3.4]
		\label{thm:Siburg Lipschitz graph}
		A monotone twist map\footnote{\red{Monotone twist maps are generalizations of the billiard map to other $C^1$-diffeomorphisms of a cylinder $\T \times [a,b]$, where $[a,b]$ is the twist interval defined analogously to the formula above; see Definition 1.1.1 on pg. 2 of \cite{Si} for more details.}} possesses minimal orbits for every rotation number in its twist interval; for rational rotation numbers, there are always at least two periodic orbits. Every minimal orbit lies on a Lipschitz graph over the $s$-axis. Moreover, if there exists an invariant circle, then every orbit on that circle is minimal.
	\end{theo}
	See also \cite{BialyPolterovich} for generalizations to higher dimensions.
	
	\begin{def1}\label{mbf}
		\textbf{Mather's $\beta$-function}, also called the \textbf{mean minimal action}, is the function
		\begin{align*}
			\beta(\omega) = \lim_{N \to \infty} \frac{1}{2N} \sum_{i = -N}^{N-1} h(s_i, s_{i+1}),
		\end{align*}
		for any minimal orbit $(s_i)_{i \in \Z}$ \red{of rotation number $\omega$}.
	\end{def1}
	
	\begin{rema}
		Note that $\beta$ is well-defined, since any minimal orbit \red{of a given rotation number} has the same action by definition.
	\end{rema}

	\begin{theo}[Theorems 1.3.7 and 3.2.5 in \cite{Si}]\label{aubmath}
		Let $f$ be a monotone twist map and $\beta$ its mean minimal action. The following hold true:
		\begin{enumerate}
			\item $\beta$ is strictly convex; in particular it is continuous.
			
			\item $\beta$ is differentiable at all irrational numbers.
			
			\item If $\omega = p/q$ is rational, $\beta$ is differentiable at $\omega$ if and only if there is an $f$-invariant circle of rotation number $p/q$ consisting entirely of periodic minimal orbits.
			
			\item If $\Gamma_\omega$ is an $f$-invariant circle of rotation number $\omega$, then $\beta$ is differentiable at $\omega$ with\footnote{\red{See theorem \ref{thm: caustic length lazutkin expansion} below for a geometric interpretation of both the invariant circle and the integral over $\Gamma_\omega$ of the canonical one-form.}}
			$$
			\beta'(\omega) = - \int_{\Gamma_\omega} \sigma ds.
			$$
		\end{enumerate}
		If $f = \dt_+$ is the billiard map on a Birkhoff billiard table, then the following also hold:
		\begin{enumerate}
			\setcounter{enumi}{4}
			\item $\beta$ is symmetric about the point $\omega = 1/2$.
			
			\item $\beta$ is three times differentiable at the boundary points with $\beta'(0) = - \ell = - |\d \Omega|$.
		\end{enumerate}
	\end{theo}
	
		\red{We will also need the following property of Mather's $\beta$-function in the proof of Theorem \ref{cpt} (see Proposition \ref{prop: MLIS is closed} below):
		
		\begin{prop} \label{prop: Lipschitz dependence of beta}
			Let $f_1, f_2$ be $C^1$ monotone twist maps on the circle $\T = \R/\Z$ with twist interval $[\omega_-, \omega_+]$ and generating functions $h_1, h_2$. Denote by $\beta_{h_i}$ the corresponding minimal actions as in Definition \ref{mbf}. We then have
			\begin{align*}
				\|\beta_{h_1} - \beta_{h_2}\|_{C^0([\omega_-, \omega_+])} \leq \|h_1 - h_2\|_{C^0(\T^2)},
			\end{align*}
			i.e., the dependence of Mather's $\beta$-function on the underlying generating function is \textit{Lipschitz} and in particular, continuous.
		\end{prop}
	}
		
		\begin{proof}
			\red{Fix a rational $p/q \in [\omega_-, \omega_+]$ and consider the \textbf{$q$-cyclic action functionals:}
			\begin{align*}
				H_q^{(j)}(s_1, \cdots, s_q) = \sum_{i = 1}^q h_j(s_i, s_{i+1}), \qquad j = 1,2,
			\end{align*}
			with the convention that $s_{q+1} = s_1$. Now consider any lift $s \to \wt s$ of $\T$ to $\R$ as in the discussion preceding \eqref{eq: general rotation number} and define the $q$-cyclic \textit{lifted} action functionals
			\begin{align*}
				\wt H_q^{(j)}(\wt s_1, \cdots, \wt s_q) = \sum_{i = 1}^q \wt h_j(\wt s_i, \wt s_{i+1}), \qquad j = 1,2,
			\end{align*}
			where $\wt h_j(\wt s_i, \wt s_{i+1}) = h_j(\pi(\wt s_i), \pi(\wt s_{i+1}))$ (with $\pi: {\R \to \T}$ being the quotient map) are $\Z^2$-periodic lifts of the generating functions. Denote by  $C_{p,q}$ the configuration space for orbits of type-$(p,q)$:
			\begin{align*}
				C_{p,q} = \{ (\wt s_1, \cdots, \wt s_q) \in \R^q : \wt s_1 < \wt s_2 < \cdots < \wt s_q, \quad \wt s_{i+q} = p + \wt s_i \}.
			\end{align*}
		Birkhoff's proof of Poincar\'e's geometric theorem showed that minima of $\wt H_q^{(j)}$ are attained away from the boundary (where $s_m = s_n$ for some $m \neq n$) and their projections to $\T^q$ correspond to type-$(p,q)$ minimal orbits of $f_j$  \cite{PoincareSuruntheoremedegeometrie}, \cite{BirkhoffPoincare}. Observe that
		\begin{align*}
			\beta_{h_j}(p/q) = \frac{1}{q} \min_{C_{p,q}} \wt H_q^{(j)}(\wt s_1, \cdots, \wt s_q).
		\end{align*}
		It then follows that
		\begin{align}\label{eq: Lipschitz on rationals}
			\left|\beta_{h_1}(p/q) - \beta_{h_2}(p/q) \right| \leq \|h_1 - h_2\|_{C^0(\T^2)}.
		\end{align}
		Since each $\beta_{h_j}$ is strictly convex and hence continuous, it follows that \eqref{eq: Lipschitz on rationals} extends from rational rotation numbers to all $\omega \in [\omega_-, \omega_+]$, which proves the proposition.}
		\end{proof}

	\subsection{Length spectra}
	\label{subsec:Length spectra}
	
	\begin{def1} \label{def:LSP}
		The \textbf{length spectrum} of $\Omega$ is
		$$
		\lsp(\Omega) = {\cup_{\gamma \,\, \text{periodic}} \left\{ \text{length}\,(\gamma)\right\}} \cup \N |\d \Omega|.
		$$
		The \textbf{marked length spectrum} is defined by
		$$
		\mls_\Omega \left(\frac{p}{q}\right) = \max \left\{ \len(\gamma) : \omega(\gamma) = p/q \right\},
		$$
		where $p,q$ are relatively prime and $p/q \in (0,1/2]$.
	\end{def1}
	
	It follows immediately that
	\begin{align*}
		- \beta\left(\frac{p}{q}\right) = \frac{1}{q} \mls_\Omega\left(\frac{p}{q}\right).
	\end{align*}
	
	\begin{rema}
		Notice that maximal lengths correspond to orbits of minimal action when considering the generating function $h(s, s^+)$ as above. The marked length spectrum ``marks'' lengths of minimal orbits by their rotation number, which plays the role of a homotopy or homology class in the closed manifold setting.
	\end{rema}

	\red{The famous Birkhoff-\red{Poritsky} Conjecture can be reformulated in terms of Mather's $\beta$-function. To explain this further, we introduce two different notions of ``integrability.''}
	
	\begin{def1}
		\red{Let $\Omega$ be a strictly convex billiard table. Its corresponding billiard map is said to be
			\begin{itemize}
				\item (locally) \textbf{Liouville-Arnold integrable} if there exists some open neighborhood in phase space $B^* \d \Omega$ which is foliated by caustics which are the level sets of some nonconstant function (a first integral). If this neighborhood can be taken to be the interior of the entire phase space, we say that the billiard table is \textit{globally} Liouville-Arnold integrable.
				
				\item (locally) \textbf{rationally integrable} if there exists an open set $U \subset (0,1/2)$ such that for each rational $p/q \in U$, there exists a rational caustic of rotation number $p/q$. If the set $U$ is of the form $(0,\eps)$, we say that the billiard map is rationally integrable near the boundary. If $U = (0,1/2)$, we say that the billiard table is globally rationally integrable.
		\end{itemize}}
	\end{def1}
	
	\begin{conj}[\textbf{Birkhoff-Poritsky}, Liouville-Arnold integrability \cite{Birkhoff}, \cite{Poritsky}]
		\label{conj:BP1}
		\red{If a billiard table is locally Liouville-Arnold integrable, then it must be an ellipse.}
	\end{conj}
	
	\red{Poritsky's formulation of the conjecture characterizes integrability locally in the Liouville-Arnold sense. As for global integrability, Bialy has shown that the only globally Liouville-Arnold integrable billiard table is a disk \cite{Bialy}. In view of Theorem \ref{aubmath}, an alternate formulation of the Birkhoff-Poritsky Conjecture using rational integrability is:}
	
	\begin{conj}[\textbf{Birkhoff-Poritsky}, rational integrability \cite{Birkhoff}, \cite{Poritsky}]
		\label{conj:BP2}
		\red{Mather's $\beta$-function is smooth on some open set if and only if the underlying billiard table is an ellipse.}
	\end{conj}
	\red{Let us also mention that continuous deformations of a domain which are \textit{unmarked} length isospectral are also marked length isospectral (see e.g., Corollaries 3.2.3 and 3.2.7 in \cite{Si}). Further recent progress towards the Birkhoff-Poritsky Conjecture can be found in \cite{KaSo16}, \cite{KaDSWe17}, \cite{KaAvDS16}, \cite{KaSoHuNearlyCircular}, \cite{BialyMironov}, and \cite{KKZ}, \cite{koval2025local}, \cite{KovalGevrey}.}
	\\
	\\
	As above, we will continue to denote by $\mathcal{B}$ the space of all Birkhoff billiard tables. Since we are working with a fixed domain $\Omega$, we will denote by $\M = \M(\Omega) \subset \mathcal{B}$ the marked length isospectral set containing $\Omega$. More generally, one could consider the entire moduli space $\wt{\M}$ of all strictly convex billiard tables quotiented by the equivalence relation of marked length isospectrality.
	\\
	\\
	\red{We now introduce a family of spectral invariants.
	\begin{def1}
		Given a set $\mathcal A$, an $\mathcal A$-valued \textbf{marked length spectral invariant} is a function $c: \mathcal{B} \to \mathcal A$ which is constant on $\M(\Omega)$ for each $\Omega \in \mathcal{B}$; equivalently, $c$ descends to a well-defined function on $\wt \M$. 
	\end{def1}

	\begin{exam}\label{exampleinvariant}
		In view of Theorem \ref{aubmath}, Mather's $\beta$-function is a marked length spectral invariant.
	\end{exam}

	In \cite{MM}, Marvizi and Melrose introduce another family of algebraically equivalent marked length spectral invariants; for each winding number $p$, the length spectrum can be decomposed into a union over $q \geq 2$ of clusters of lengths corresponding to orbits of rotation number $p/q$ (see Figure \ref{cartoon}). We denote these clusters by $[t_{p,q}, T_{p,q}]$, where
	\begin{align}\label{eq:tpqTpq}
		\begin{split}
			t_{p,q} = &\red{\min}_{\gamma \text{ of type-} (p, q)} \len(\gamma),\\
			T_{p,q} =& \red{\max}_{\gamma \text{ of type-} (p, q)} \len(\gamma).
		\end{split}
	\end{align}
	Using a so-called ``interpolating Hamiltonian'' (see Theorem \ref{thm: interpolating hamiltonian}), whose existence is a consequence of Melrose's equivalence of glancing hypersurfaces \cite{MelroseGlancingHypersurfaces1}, \cite{MelroseGlancingHypersurfaces2}, \cite{MelroseTaylor}, they show that there exist constants $c_{p,k}$ such that
	\begin{align}
		T_{p,q} \sim p \ell + \sum_{k = 1}^\infty c_{p,k} q^{-2k},
		\label{MMcpk}
	\end{align}
	and {$T_{p,q} - t_{p,q} =O(q^{-\infty})$, i.e., $T_{p,q} - t_{p,q} =O_k(q^{-k})$ for each $k \in \N$}. It follows immediately that Mather's $\beta$-function has a formal Taylor expansion at $\omega = 0$ (see also \cite{Poschel} and \cite{CarminatiMarmiSauzinSorrentino}).
	
	\begin{rema}
			For an ellipse, $t_{p,q} = T_{p,q}$ for all $0 < p/q < 1/2$, corresponding to the one-parameter families of orbits tangent to confocal conic sections. In fact, this is true whenever there is a rational caustic of rotation number $p/q$ \red{(see Section \ref{subsec: Caustics})}.
	\end{rema}
	
	\begin{rema}
		\red{Despite their appearance in the expansion of Mather's $\beta$-function, existence of such asymptotics in the length spectrum was already implicit in Lazutkin's construction of invariant curves and normal form coordinates \cite{Lazutkin}. However, it was Marvizi and Melrose who first wrote down such an explicit asymptotic expansion and computed their first two coefficients \cite{MM}. Their use of an interpolating Hamiltonian provided a new framework for studying convex billiards. A geometric interpretation of these invariants was provided in \cite{Amiran} and the interpolating Hamiltonian approach was later expanded upon by Kovachev and Popov in \cite{kovachev1990invariant}.}
	\end{rema}
	}
	
\begin{figure}
	\centering
	\includegraphics{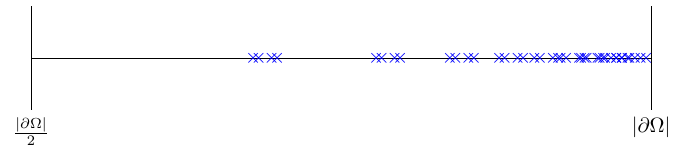}
	\caption{An illustration of what bands $[t_{1,q}, T_{1,q}]$ and clusters in the length spectrum corresponding to $1/q$ orbits might look like.}
	\label{cartoon}
\end{figure}

	\red{	
	\begin{def1}\label{def: MB and MM invariants}
		The coefficients in the formal Taylor expansion
			\begin{align*}
				\beta(\omega) \sim \sum_{k = 1}^\infty \beta_k \omega^k
			\end{align*}
		of {Mather's $\beta$-function} are real valued marked length spectral invariants which we call \textbf{$\beta$-coefficients}. They are related to the invariants $c_{p,k}$ of Marvizi and Melrose by the formula
	\begin{align}\label{bettaylor}
		\begin{split}
			\beta_{\red{2k +1}} = & - \frac{1}{p^{2k+1}} c_{p,k}, \qquad k \geq 1
			\\
			\beta_1 =& - \ell, \qquad \beta_{2k} = 0.
		\end{split}
	\end{align}
	\red{Regarding the nomenclature, some authors have called $\beta_k$ (equivalently, $c_{p,k}$) \textit{Marvizi-Melrose invariants}, but for clarity, we propose to reserve this term for the invariants $\I_k$ in Theorem \ref{main} (see Definition \ref{action} and Theorem \ref{thm: caustic length lazutkin expansion} below).} In Definition \ref{def: MB and MM invariants} above, we have used the notation $\sim$ to denote a formal asymptotic expansion which does not necessarily converge but nonetheless identifies uniquely defined coefficients. We will continue to use this notation throughout the paper.
	\end{def1}
}
	
	\subsection{Connection to the Laplace spectrum}\label{Connection with Laplace spectrum}
	Under generic conditions (the noncoincidence condition), the asymptotics of these lengths were shown to also be Laplace spectral invariants.
	
	\begin{def1}
		A domain $\Omega$ is said to satisfy the noncoincidence condition if there exists an $\eps > 0$ such that
		\begin{align}\label{NCC}
			(|\d \Omega| - \eps, |\d \Omega| ] \cap  \lsp_{p,q}(\Omega) = \emptyset,
		\end{align}
		for all $p \geq 2$ and $q < \infty$, where $\lsp_{p,q}(\Omega)$ is the portion of the length spectrum arising from periodic orbits of type-$(p,q)$.
	\end{def1}
	
	\noindent In this case, it is shown in \cite{MM}, that the endpoints
	\begin{align*}
		t_{1, q} =& \min_{\gamma \text{ of type-} (1, q)} \text{length}(\gamma),\\
		T_{1, q} =& \max_{\gamma \text{ of type-} (1, q)} \text{length}(\gamma)
	\end{align*}
	belong to the singular support of the wave trace $\Tr \cos t \sqrt{- \Delta}$, i.e. an \textit{equality} in the Poisson relation \eqref{PoissonRelation} \red{at the endpoints of the clusters in $[t_{1,q},T_{1,q}]$ near the length of the boundary. In fact, this was improved in \cite{HeZe19} (Theorem 1.4 and Remark 6.2), where it is shown that under the noncoincidence condition,} 
	\begin{align*}
		\red{
		\operatorname{SingSupp} \left(\tr \cos t \sqrt{-\Delta}\right) \cap (|\d \Omega| - \eps, |\d \Omega| ] = \lsp(\Omega) \cap (|\d \Omega| - \eps, |\d \Omega| ],
		}
	\end{align*}
	\red{where $\eps$ is as in \eqref{NCC}; in other words, the wave trace is singular at \textit{each} length in the interval $[t_q, T_q]$, not just the endpoints.} This follows from a version of stationary phase due to Soga which applies to oscillatory integrals with degenerate phases \cite{Soga}. Hence, invariants of the distribution of these lengths are also Laplace spectral invariants amongst domains satisfying the noncoincidence condition. As mentioned in the introduction, this class is \red{residual (in particular, dense)} amongst $C^\infty$-Birkhoff billiard tables and includes ellipses, analytic domains, and a $C^1$-open neighborhood of disks. An explicit formula for the wave trace near such orbits is given in \cite{Vig22}.
	\\
	\\
	\red{The importance of the noncoincidence condition \eqref{NCC} is two-fold. First, it prevents cancellations in the wave trace which could render it smooth, i.e., a strict inclusion in the Poisson relation \eqref{PoissonRelation}. This would make such lengths inaudible from the Laplace spectrum; examples of such pathological domains (exhibiting cancellation of arbitrarily many terms in the wave trace) were constructed recently in joint work with Kaloshin and Koval \cite{KKV}, \cite{KVSilentOrbits25}. Secondly, the noncoincidence condition allows us to examine the distribution of only \textit{polygonal} lengths near the perimeter $|\d \Omega|$, i.e. those which come from orbits of rotation number $1/q$. If $p \geq 2$ and there were an accumulation of type-$(p,q)$ orbits at $|\d \Omega|$, it would be impossible to decouple the unmarked length spectrum in a way which singles out only polygonal orbits, preventing us from reading off the coefficients $c_{1,k}$.}
	\\
	\\
	By extremizing the first two invariants in \eqref{MMcpk} (or rather, their algebraically equivalent counterparts $\I_1$ and $\I_2$ in Theorem \ref{main}; see Definition \ref{action} and Proposition \ref{prop: Kovachev Popov coordinates and caustic length lazutkin expansion}), Marvizi and Melrose found a two-parameter family of spectrally determined domains amongst those satisfying the noncoincidence condition. Given the recent symbolic computations in \cite{Sorr15}, it would be interesting to find critical points of higher order invariants, or more generally, any function of a finite number of the $\I_k$, by solving the resultant Euler-Lagrange equations.
	\\
	\\
	The clustering of these lengths within narrow bands resembles the eigenvalue clusters seen in perturbation theory \red{(e.g., Weinstein's band invariants \cite{WeinsteinBandInvariants}). Although there is no known operator whose spectrum coincides with the length spectrum, this formal analogy suggests that the internal distribution of lengths within each cluster may contain additional geometric information beyond that which is encoded in the extremal (maximal or minimal) length spectrum}.
	\begin{ques}
		\red{What information about the geometry of $\Omega$ is encoded in the internal distribution of lengths within the clusters $[t_{p,q}, T_{p,q}]$ as $q \to \infty$?}
	\end{ques}
	
	 For an ellipse, the intervals $[t_{p,q}, T_{p,q}]$ collapse into single points with \red{$t_{p,q} = T_{p,q}$ whenever $q \neq 2p$; except for bouncing ball (period two) orbits, all periodic billiard orbits of a given rotation number have the same length and come in one-parameter families corresponding to confocal conic sections (ellipses, hyperbolas, or the limiting segment which connects the two focal points).} In the analytic category, each band has finitely many lengths and one can study their distribution on a logarithmic scale, as was done in \cite{MartinRamirezRos1} and \cite{MartinRamirezRos2}. \red{If one views a general strictly convex domain as the perturbations of an ellipse, then it is natural to ask the extent to which this perturbation is reflected in the corresponding length spectrum.} \red{As an example of the complexity of any answer to this question, we mention the recently announced work of De Simoi, Koudjinan and Zhang, which asserts that there exist Birkhoff billiard tables with \textit{uncountable} length spectrum \cite{deSimoi}.}

	\subsection{Caustics}\label{subsec: Caustics}
	
	\begin{def1}\label{def:caustics}
		A smooth, convex, closed curve $\Gamma$ lying in $\Omega$ is called a {caustic}\footnote{\red{Caustics can be defined much more generally and need not be closed nor convex. For example, confocal hyperbolas are caustics in an elliptical billiard table. In this paper, however, we will focus exclusively on closed convex caustics.}} if any link drawn tangent to $\Gamma$ remains tangent to $\Gamma$ after an elastic reflection at the boundary of $\Omega$. By elastic reflection, we mean that the angle of incidence equals the angle of reflection at an impact point on the boundary. We can map $\Gamma$ to the total phase space $B^* \partial \Omega$ to obtain a smooth closed curve which is invariant under the billiard ball maps $\dtpm$. By an abuse of notation, we will denote by $\Gamma$ both the caustic itself and its corresponding invariant curve in $B^* \d \Omega$.
	\end{def1}
	
	\begin{rema}
		If the dynamics are integrable (for example, in the sense of Liouville), these invariant curves are precisely the Lagrangian tori which foliate \red{the} phase space.
	\end{rema}
	
	We will denote the length of a caustic $\Gamma$ by $|\Gamma|$. \red{We define the rotation number of a \textit{caustic} $\Gamma$ (as opposed to an individual orbit) to be the rotation number of any orbit which is tangent to it; indeed, the billiard map restricted to an invariant circle (corresponding to the caustic) in $B^*\d \Omega$ is an orientation preserving homeomorphism and it was proved by Poincar\'e that every such map from a circle to itself has an invariantly defined rotation number. This rotation number coincides with \eqref{eq: general rotation number} and in particular, gives a rational $p/q$ for billiard orbits with winding number $p$ and period $q$. We will continue to denote the rotation number of a caustic by $\omega$.} Besides the rotation number, we may introduce another invariant associated to a caustic:
	
	\begin{def1}\label{laz}
		Let $x$ and $y$ be two points on a caustic $\Gamma$ and $z \in \d \Omega$ such that the links $\overline{xz}, \overline{zy}$ correspond to a billiard orbit. Denote by $\wideparen{xy}$ the minimal arc connecting $x$ and $y$ and by $|\wideparen{xy}|$ its length. Then the quantity
		\begin{align}
			\label{eq: laz}
			Q = |\overline{xz}| + |\overline{zy}| - \red{|\wideparen{xy}|}
		\end{align}
		is called the \textbf{Lazutkin parameter} of $\Gamma$. See Figure \ref{Lazutkin}.
	\end{def1}

		\begin{figure}
			\centering
			\includegraphics{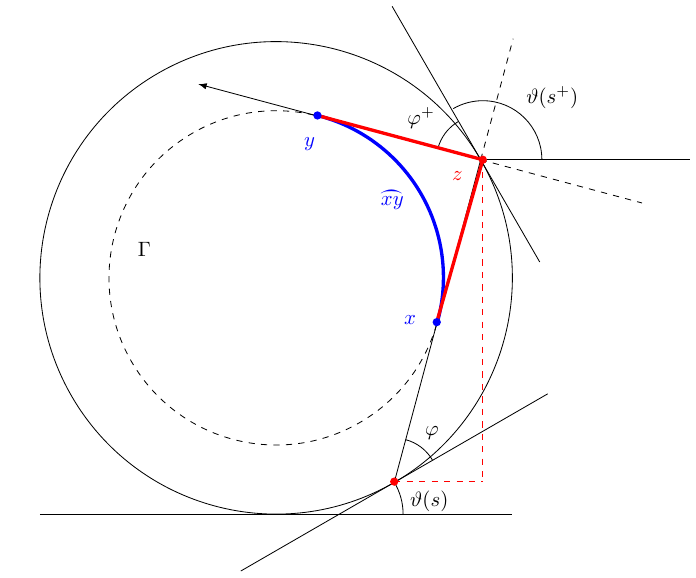}
			\caption{The Lazutkin parameter $Q$ of the caustic $\Gamma$ (dashed) is the sum of lengths of the two red segments minus the length of the blue arc between them.}
			\label{Lazutkin}
		\end{figure}

	Since $\Gamma$ is a caustic, $\d \Omega$ is an \textbf{involute} of $\Gamma$; i.e. if a circular string of length $Q + |\Gamma|$ is wrapped taut around $\Gamma$, then the locus of points traced out by the vertex of the cinched string will coincide with $\d \Omega$. While the quantity \red{\eqref{eq: laz}} can be defined for any convex, closed curve, it is only constant on a caustic. In order to study the Taylor coefficients of Mather's $\beta$-function near $0$, we need to know that sufficiently many caustics exist in order to apply Theorem \ref{aubmath}.
	\begin{theo}[\cite{Lazutkin}]
		In any neighborhood of the boundary, there exists a family of caustics having a Cantor set of Diophantine rotation numbers which have positive Lebesgue measure in any neighborhood of zero.
	\end{theo}
	The Taylor coefficients of $\beta$ can then be extracted by taking $\omega \to 0$ along a family of Diophantine rotation numbers. $\beta$ is in fact $C^\infty$ in the sense of Whitney on the corresponding Cantor set \cite{Poschel}; see also \cite{CarminatiMarmiSauzinSorrentino}.

	\begin{theo}[Theorems 3.2.10 and 3.2.23 in \cite{Si}]
		\label{thm: caustic length lazutkin expansion}
		Let $\Gamma_\omega$ be a convex caustic of rotation number $\omega$. Then, $\left| \Gamma_\omega \right|$ and $Q\left(\Gamma_\omega\right)$ are marked length spectral invariants satisfying
		\begin{align*}
			\left|\Gamma_\omega \right| = - \beta'(\omega),\\
			Q \left(\Gamma_\omega\right) = \alpha\left(\beta'(\omega)\right),
		\end{align*}
		where $\alpha = \beta^*$ is the convex conjugate ({Legendre-Fenchel transform}) of $\beta$. Furthermore, there exists a formal asymptotic expansion
		\begin{align*}
			\left|\Gamma_\omega\right| \sim \ell + \sum_{k \geq 1} b_{k} Q^{2k/3},
		\end{align*}
		with the coefficients $b_{k}$ being marked length spectral invariants of $\Omega$.
	\end{theo}

	\begin{rema}
		{Comparing Theorem \ref{thm: caustic length lazutkin expansion} with Theorem \ref{main}, we see that
			\begin{align*}
				b_k = \frac{1}{k!} \left(\frac{3}{2}\right)^{2k/3} \I_k,
			\end{align*}
			where $\I_k$ are the Marvizi-Melrose invariants (see Definition \ref{action}), as} will be shown in Proposition \ref{prop: Kovachev Popov coordinates and caustic length lazutkin expansion}. Comparing the formula for $|\Gamma|$ in Theorem \ref{thm: caustic length lazutkin expansion} with the formula for $\beta'(\omega)$ in Proposition \ref{aubmath}, we see that
		\begin{align*}
			|\Gamma| = \int_{\Gamma} \sigma ds,
		\end{align*}
		$\sigma ds$ being the canonical one form. This will also be important in the proof of Proposition \ref{prop: Kovachev Popov coordinates and caustic length lazutkin expansion} below.
	\end{rema}

	\red{Recall from Section \ref{sec: Main Results} that $\kappa_i$ denotes the $i^{\text{th}}$ derivative of the curvature function $\kappa$ in arclength coordinates.} In \cite{Sorr15} (page 22), it was shown that
	\begin{align}\label{SIK}
		\begin{split} 
			\I_0 =& \iota_0 \int_0^{\ell} ds = \ell,
			\\
			\I_1 =& \iota_1 \int_0^\ell \ka^{2/3} ds,
			\\
			\I_2 =&\iota_2 \int_0^\ell \left(9 \ka^{4/3} + \frac{8 \ka_1^2}{\ka^{8/3}}\right)ds,
			\\
			\I_3 =&\iota_3  \int_0^\ell \left(9 \ka^2 + \frac{24 \ka_1^2}{\ka^2} + \frac{24 \ka_2^2}{\ka^4} - \frac{144 \ka_1^2 \ka_2}{\ka^5} + \frac{176 \ka_1^4}{\ka^6}\right)ds,
			\\
			\I_4 =& \iota_4 \int_0^\ell \bigg( \frac{281 }{44800}  \ka^{8/3} + \frac{281 \ka_1^2}{8400 \ka^{4/3}} + \frac{167 \ka_2^2}{4200 \ka^{10/3}} - \frac{167 \ka_1^2 \ka_2}{700 \ka^{13/3}}\\
			&+ \frac{\ka_3^2}{42 \ka^{16/3}} + \frac{559 \ka_1^4}{2100 \ka^{16/3}}
			- \frac{473 \ka_2^3}{4725 \ka^{19/3}} - \frac{10 \ka_3 \ka_1 \ka_2}{21 \ka^{19/3}}\\
			& +\frac{5 \ka_3 \ka_1^3}{7 \ka^{22/3}} + \frac{10777 \ka_1^4 \ka_2}{1575 \ka^{25/3}} + \frac{521897 \ka_1^6}{127575 \ka^{28/3}} \bigg)ds,
		\end{split}
	\end{align}	
	\red{for some universal nonzero constants $\iota_k, 0 \leq k \leq 4$ (see \eqref{eq:iotavalues} for their precise values)}. Theorem \ref{main} establishes a hierarchical structure for all $\I_k$, including those in \eqref{SIK}. {To verify that the first four coefficients fit into the structure described in Theorem \ref{main}, let us note the following important example, which shows that the \textit{integrands} themselves in \eqref{SIK} are not unique.}

	\begin{exam}\label{rmk: nonuniqueness}
		\red{Since $\kappa$ satisfies periodic boundary conditions, one can freely integrate by parts, which changes the expressions of the integrands without altering the integral itself. Let us integrate by parts to show how the integrals \eqref{SIK} of $\I_k$ can be simplified further to be put into the form described in Theorem \ref{main}}. \red{The integrands of $\I_1$ and $\I_2$ are clearly already compatible with the structure in Theorem \ref{main}. In $\I_3$, the differential degree of each term is at most four and the term $24 \ka_2^2 \ka^{-4}$ has maximal derivatives appearing quadratically. There is another term $-144 \ka_2 \ka_1^2 \ka^{-5}$ in which maximal derivative $\ka_2$ appears without a square. It can be rewritten as}
		\red{
		\begin{align*}
			-144 \ka_2 \ka_1^2 \ka^{-5} ds =& - 48 d\left(\ka_1^3\right) \ka^{-5} = - 240 \ka_1^4 \ka^{-6} ds - 48 d \left(\ka_1^3 \ka^{-5} \right).
		\end{align*}
		The term $- 48 d(\ka_1^3 \ka^{-5})$ is an \textit{exact} differential and hence integrates to zero. With this reduction the integrand of $\I_3$ becomes
		\begin{align*}
			\iota_3 \left(9 \ka^2 + \frac{24 \ka_1^2}{\ka^2} + \frac{24 \ka_2^2}{\ka^4} + (176 - 240) \frac{\ka_1^4}{\ka^6}\right)ds,
		\end{align*}
		which matches the structure described in Theorem \ref{main}; the top order derivative has order $3-1 =  2$ and appears quadratically with a coefficient of $\ka^{-4 \times 3 / 3}$ times the nonzero constant $24 \iota_3$. Similarly, the integrand
		\begin{align*}
			\iota_4 \bigg( &\frac{281 }{44800}  \ka^{8/3}
			+ \frac{281 \ka_1^2}{8400 \ka^{4/3}}
			+ \frac{167 \ka_2^2}{4200 \ka^{10/3}}
			- \frac{167 \ka_1^2 \ka_2}{700 \ka^{13/3}}\\
			&+ \frac{\ka_3^2}{42 \ka^{16/3}}
			+ \frac{559 \ka_1^4}{2100 \ka^{16/3}}
			- \frac{473 \ka_2^3}{4725 \ka^{19/3}}
			- \frac{10 \ka_3 \ka_1 \ka_2}{21 \ka^{19/3}}\\
			&+\frac{5 \ka_3 \ka_1^3}{7 \ka^{22/3}}
			+ \frac{10777 \ka_1^4 \ka_2}{1575 \ka^{25/3}}
			+ \frac{521897 \ka_1^6}{127575 \ka^{28/3}} \bigg)
		\end{align*}
		of $\I_4$ has a quadratic term $\ka_3^2 / (42 \ka^{{16}/{3}})$ and two other terms in which $\ka_3$ appears without a square. They can each be transformed as follows:
		\begin{align*}
			- \frac{10 \ka_3 \ka_1 \ka_2}{21 \ka^{19/3}} ds =& - \frac{10 \ka_1 }{42 \ka^{19/3}} d \left(\ka_2^2\right)\\
			=&
			\frac{10 \ka_2^3 }{42 \ka^{19/3}} ds - \frac{95 \ka_1^2 \ka_2^2}{63 \ka^{22/3}} ds - d \left(\frac{10 \ka_1 \ka_2^2}{42 \ka^{19/3}} \right)
		\end{align*}
		and
		\begin{align*}
			\frac{5 \ka_3 \ka_1^3}{7 \ka^{22/3}} ds = - \frac{15 \ka_2^2 \ka_1^2}{7 \ka^{22/3}} ds + \frac{110 \ka_2 \ka_1^4}{21 \ka^{25/3}} ds + d \left(\frac{5 \ka_2 \ka_1^3}{7 \ka^{22/3}}\right),
		\end{align*}
		both of which can be absorbed into the remainder term $\mathcal{R}_4$ in Theorem \ref{main}.
	}
	\end{exam}

	One goal of this paper is to describe an algorithm for strategically integrating by parts in a way that \textit{minimizes} the maximal order of derivatives appearing in the integrands of higher order $\I_k$. This algorithm is described in detail in Section \ref{subsec: cohomological considerations and curvature polynomials} (see Proposition \ref{IBPalgorithm}) and used extensively in Section \ref{Integral invariants}, in particular, Section \ref{subsec: integration by parts}.

	\section{Proof of Theorem \ref{main} $\implies$ Theorem \ref{cpt}}\label{MIC}
	
	In this section, we show how the algebraic structure formula for $\I_k$ in Theorem \ref{main} yields compactness in Theorem \ref{cpt}. We begin with some $L^\infty$ estimates on $\ka, \ka_1$ which will be needed later \red{to estimate} the $L^2$-norms of higher order derivatives. {Since $|\d \Omega| = \ell$ is constant on marked length isospectral sets, we will fix $\ell$ throughout the remainder of the paper.}

	\begin{lemm}\label{kappa0}
		For each Birkhoff billiard table $\Omega_0$, there exists a $c(\Omega_0) > 0$ such that for all $\Omega$ marked length isospectral to $\Omega_0$,
		$$
		0  <c(\Omega_0) \leq \ka_{\Omega} \leq \frac{1}{c(\Omega_0)}.
		$$
	\end{lemm}
	
	\begin{proof}
		Let us examine the formula
		$$
		\I_2 = \red{\iota_2} \int_0^\ell \left(9 \ka^{4/3} + \frac{8 \ka_1^2}{\ka^{8/3}}\right)ds.
		$$
		It follows that for any $\Omega$ which is marked length isospectral to $\Omega_0$, we have
		\begin{align*}
			8 \red{| \iota_2|} \int_0^\ell \ka^{-2/3} \left(\frac{d}{d s}\log \ka(s)\right)^2 ds = \red{|\iota_2|} \int_0^\ell \frac{8 \ka_1^2}{\ka^{8/3}} ds \leq |\I_2|.
		\end{align*}
		\red{Since
		\begin{align*}
			\int_{\d \Omega} \ka(s) ds = 2\pi
		\end{align*}
		by the Gauss-Bonnet formula, it follows that there exists some point $s_0 \in \R/\ell \Z$ at which $\ka(s) = 2\pi / \ell$. Using the fundamental theorem of calculus and Cauchy-Schwarz inequality, we see that
		\begin{align*}
			\|\log \ka\|_{L^\infty} \leq& \left|\log \left(\frac{2\pi}{\ell}\right)\right| + \int_0^\ell \left| \frac{d}{d s} \log \ka \right| ds
			\\
			\leq &  \log 2\pi + \left|\log \ell \right| +\left(\int_0^\ell \left| \frac{d}{d s} \log \ka \right|^2 \ka^{-2/3} ds\right)^{1/2} \left(\int_0^\ell \ka^{2/3} ds\right)^{1/2}
			\\
			\leq &  \log 2\pi + |\log \I_0| + \sqrt{\frac{| \I_1 \I_2 |}{8 |\iota_1 \iota_2|}}.
		\end{align*}
	}

		The lemma then follows immediately from the boundedness of $\log \ka$.
	\end{proof}

	\begin{lemm}\label{k1}
		For any $\Omega_0$, $\ka_2$ is bounded in $L^2$ and $\ka_1$ is bounded in $L^\infty$ on the marked length isospectral set $\mathcal{M}(\Omega_0)$.
	\end{lemm}
	\begin{proof}
		\red{Since $\ka_1$ has mean zero, there exists $s_0 \in \R/\ell \Z$ such that $\ka_1(s_0) = 0$.} It is \red{then} clear that
		\begin{align*}
			\left|\ka_1 (s) \right| = \left| \int_{s_0}^s \ka_2(t)\, dt\right| \leq \| \ka_2\|_{L^1}.
		\end{align*}
		\red{We will demonstrate the Lemma by showing that via integration by parts, the integrand of $\iota_3^{-1} \I_3$ can be converted into a sum of all \textit{nonnegative} terms, one of which will be $(24 - A^2) \ka_2^2/\ka^4$ for some positive constant $A < \sqrt{24}$ (see equation \eqref{CSq} below). Combining this with the upper bound $\ka \leq C$ in Lemma \ref{kappa0}, we see that
		\begin{align*}
			\|\ka_2\|_{L^2}^2 \leq & \int_0^\ell \frac{C^4}{\ka^4} \ka_2^2ds \leq \frac{C^4}{(24-A^2)} \int_0^\ell (24 - A^2)  \frac{\ka_2^2}{\ka^4} + \text{(nonnegative terms) } ds\\
			=&
			\left| \frac{C^4}{\iota_3 (24- A^2)} \I_3 \right|,
		\end{align*}
		and hence, $\ka_2$ is bounded in $L^2$. The claim that $\ka_1 \in L^\infty$ then follows directly from the fundamental theorem of calculus $\|\ka_1\|_{L^\infty} \leq \|\ka_2\|_{L^1}$ together with H\"older's inequality $\|\ka_2\|_{L^1} \leq \sqrt{\ell} \|\ka_2\|_{L^2}$.}
		\\
		\\
		\red{Let us now show that $\iota_3^{-1} \I_3$ can be written as the sum of nonnegative terms as described above. To begin,} observe that all terms in $\iota_3^{-1} \I_3$ are \red{nonnegative} except for \red{$-144 \ka_1^2 \ka_2 \ka^{-5}$,} which can be combined with $176 \ka_1^4 \ka^{-6}$ via integration by parts:
		\begin{align*}
			\int_0^\ell \frac{\ka_1^4}{\ka^6} \,ds = \int_0^\ell \frac{- 3 \ka \ka_1^2 \ka_2}{\ka^6}\,ds + 6 \int_0^\ell \frac{\ka \ka_1^4}{\ka^7},\\
			\implies \int_0^\ell \frac{\ka_1^4}{\ka^6} \,ds = \frac{3}{5} \int_0^\ell \frac{\ka_1^2 \ka_2}{\ka^5}\,ds.
		\end{align*}
		\red{However, applying this to the entire $-144 \frac{\ka_1^2 \ka_2}{\ka^5}$ term produces a negative coefficient when combined with $176 \frac{\ka_1^4}{\ka^6}$. Instead, we will include $24 \frac{\ka_2^2}{\ka^4}$ in our simplification as follows.} Let us write
		\begin{align*}
			176 = \alpha + \beta, \qquad 0 \leq \alpha, \beta \leq 176,
		\end{align*}
		so that
		\begin{align*}
			\begin{split}
				\I_3 &= \iota_3 \int_0^\ell 9 \ka^2 + \frac{24 \ka_1^2}{\ka^2} + \frac{24 \ka_2^2}{\ka^4} - 144 \frac{\ka_1^2 \ka_2}{\ka^5} +  (\alpha + \beta)\frac{ \ka_1^4}{\ka^6}\,ds\\
				&= \iota_3 \int_0^\ell 9 \ka^2 + \frac{24 \ka_1^2}{\ka^2} + \frac{24 \ka_2^2}{\ka^4} + \left(\frac{3 \alpha }{5} - 144 \right)\frac{\ka_1^2 \ka_2}{\ka^5} +  \beta \frac{ \ka_1^4}{\ka^6}\,ds.
			\end{split}
		\end{align*}
		Isolating the differential degree four terms, we multiply all terms by the appropriate power of $\ka$ and complete the square to obtain
		\begin{align}\label{CSq}
			\I_3 = \iota_3 \int_0^\ell 9 \ka^2 + \frac{24 \ka_1^2}{\ka^2} + \frac{\left(A \ka_2 \ka - B\ka_1^2 \right)^2}{\ka^6} + \frac{(24- A^2) \ka_2^2 \ka^2 }{\ka^6} \,ds,
		\end{align}
		where $A, B$ are chosen so that 
		\begin{align*}
			\begin{cases}
				0 < B^2 = \beta,\\
				-2 A B = \frac{3 \alpha}{5} - 144,\\
				\alpha + \beta = 176.
			\end{cases}
		\end{align*}
		Clearly, $B = \sqrt{\beta}$ and solving the other two equations for $\beta$ yields
		\begin{align}\label{quadratic}
			\red{B^2} - \frac{10}{3} A \red{B} + 64 = 0. 
		\end{align}
		Solving for positive roots, we see that the discriminant is positive if and only if
		\begin{align*}
			A^2 > \frac{256 }{100} \times 9 = 23.04,
		\end{align*}
		which is conveniently less than $24$, \red{so that the rightmost integrand in \eqref{CSq} is \textit{positive}.} Putting in $A^2 = 257 \times 9 / 100$, we find that
		\begin{align*}
			23.04 < A^2 = 23.13 < 24,
		\end{align*}
		and the zeros of \eqref{quadratic} are given by
		\begin{align*}
			0 < \red{B_\pm} = \frac{\sqrt{257}}{2}  \pm \frac{1}{2},
		\end{align*}
		\red{from which it is clear that either choice of $\beta_\pm = B_\pm^2$ satisfies $0 < \beta_\pm < 176$.} \red{For $\beta = \beta_+$ and $\alpha = 176 - \beta_+$}, we see that $(24 - A^2) > 0$ and each of the terms in the integrand of \eqref{CSq} \red{becomes} nonnegative.
	\end{proof}

	We now show how to use the algebraic structure in Theorem \ref{main} to estimate the Sobolev norms of the curvature on $\M(\Omega)$.

	\begin{prop}\label{unifbdd}
		For each $\Omega_0$ and $k \in \N$, $\ka_{k-1}$ is bounded in $L^2$ on the marked length isospectral set $\M(\Omega_0)$.
	\end{prop}
	\begin{proof}
		The proposition is clearly true for $k = 1,2, 3$, coming from Lemmas \ref{kappa0} and \ref{k1}, which use formulas \eqref{SIK} for the first few $\I_k$ in \cite{Sorr15}. \red{Assuming Theorem \ref{main}, we see that
		\begin{align*}
			\left| \mathcal{I}_k \right| = \left|\int_{\d \Omega} c_k \ka^{-4k/3}(s) \ka_{k-1}^2(s) + \mathcal{R}_k(s) ds \right|
		\end{align*}
		is bounded on $\mathcal{M}(\Omega_0)$ and hence,
		\begin{align*}
			\int_0^\ell \ka_{k-1}^2 ds  \leq \frac{1}{|c_k| c(\Omega_0)^{4k/3}} \left(\left| \int_0^\ell \mathcal{R}_k ds \right|  +|\I_k|\right),
		\end{align*}
		where $c(\Omega_0)$ is the constant in Lemma \ref{kappa0} and $c_k$ is the constant in Theorem \ref{main}. We will prove the proposition by showing that for each $k \in \N$, the remainder term $\mathcal{R}_k$ is bounded in $L^1$ on $\mathcal{M}(\Omega_0)$. $\mathcal{R}_k$ depends on derivatives of $\ka$ of order at most $k-2$ and has differential degree less than or equal to $2k - 2$. For any monomial $u = \ka_0^{p_0} \cdots \ka_{k-2}^{p_{k-2}}$ in $\mathcal{R}_k$, denote by $p_i$ the power of each $\ka_i$, $1 \leq i \leq k-2$. Since the differential degree of $\mathcal{R}_k$ is at most $2k - 2$, we see that
		\begin{align}
			\label{eq:pk bound}
			p_{k-2} (k -2) \leq 2k - 2 \implies p_{k-2} \leq 2 + \frac{2}{k-2}.
		\end{align}
		We have two cases:}
		\\
		\\
		\red{\textbf{Case 1.} If $k = 4$, we can write
		\begin{align*}
			\mathcal{R}_4 = \sum_{\substack{p_0, p_1, p_2\\ p_1 + 2 p_2 \leq 6} } c_{p_0,p_1,p_2} \kappa^{p_0} \kappa_1^{p_1} \kappa_2^{p_2},
		\end{align*}
		for some coefficients $c_{p_0, p_1, p_2}$. Now, \eqref{eq:pk bound} tells us that $p_2 \leq 3$. If $p_2 = 0$, Lemmas \ref{kappa0} and \ref{k1} show us that $u$ is bounded in $L^\infty$ and in particular, in $L^1$. If $p_2 = 1$, H\"older's inequality gives us
		\begin{align*}
			\|u\|_{L^1} \leq \|\ka_0^{p_0} \|_{L^\infty} \|\ka_1\|_{L^\infty}^{p_1} \ell^{\half} \|\ka_2 \|_{L^2},
		\end{align*} 
		the right-hand side of which is bounded on $\mathcal{M}(\Omega_0)$. If $p_2 = 2$, we automatically have the same bound without needing to invoke H\"older's inequality, so we can bound all terms with $p_2 < 3$ by some constant which is uniform on $\M(\Omega_0)$:
		\begin{align*}
		\sum_{\substack{p_0,p_1, p_2\\ p_1 + 2 p_2 \leq 6\\p_2 \leq 2}}  c_{p_0,p_1,p_2} \| \kappa^{p_0} \kappa_1^{p_1} \kappa_2^{p_2} \|_{L^1} \lesssim 1.
		\end{align*}
		If $p_2 = 3$, then $p_1 = 0$ and we use the interpolated Gagliardo-Nirenberg-Sobolev inequality:
		\begin{align*}
			\| \d^s \ka_2\|_{L^p} &\leq C \| \d^t \ka_2 \|_{L^r}^\theta \|\ka_2 \|_{L^q}^{1- \theta},\\
			\frac{1}{p} &= \frac{s}{n} + \theta \left( \frac{1}{r} - \frac{t}{n}\right) + \frac{1- \theta}{q},\\
			&1 \leq p < \infty, \quad 1 \leq q, r \leq \infty, \quad s < t, \quad \frac{s}{t} \leq \theta \leq 1
		\end{align*}
		to see that
		\begin{align*}
			\|\ka_2\|_{L^3} \lesssim \|\ka_3\|_{L^2}^{\frac{1}{6}} \|\ka_2\|_{L^2}^{\frac{5}{6}}.
		\end{align*}
		Going back to the $L^1$-bound on $\RR_4$, we see that
		\begin{align*}
			\|\ka_3\|_{L^2}^2 \lesssim & \|\RR_4\|_{L^1} + |\I_4|
			\\
			\lesssim & \|\ka_2\|_{L^3}^3 + 1 + |\I_4|
			\\
			\lesssim & \|\ka_3\|_{L^2}^{\frac{1}{2}} \|\ka_2\|_{L^2}^{\frac{5}{2}} + \|\ka_2\|_{L^2}^3 + 1 + |\I_4|\\
			\lesssim & \eps \|\ka_3\|_{L^2}^2  + C_\eps \|\ka_2\|_{L^2}^{\frac{10}{3}} + \|\ka_2\|_{L^2}^3 + 1 + |\I_4|,
		\end{align*}
		where the last line follows from the Peter-Paul inequality. The implied constants in $\lesssim$ are independent of $\eps$, so by choosing $\eps$ sufficiently small, we have
		\begin{align*}
			(1 - \eps) \|\ka_3\|_{L^2}^2 \lesssim C_\eps \|\ka_2\|_{L^2}^{\frac{10}{3}} + \|\ka_2\|_{L^2}^3  + 1  + |\I_4|,
		\end{align*}
		the right-hand side of which is bounded from Lemmas \ref{kappa0} and \ref{k1}. We conclude that $\|\ka_3\|_{L^2}$ is bounded on $\M(\Omega_0)$.}
		\\
		\\
		\red{\textbf{Case 2.} $k \geq 5$. 
		We argue by induction on $k$. \red{Fix $k \in \N$ and assume that there exists a constant $C_{k-2}(\Omega_0) < \infty$ such that $\|\ka_i\|_{L^2} \leq C_{k-2}(\Omega_0)$ for $i = 1, \cdots, k-2$ on $\mathcal{M}(\Omega_0)$.} Since the exponent $p_{k-2}$ is integral, we see that the condition $k \geq 5$ forces $p_{k-2} \leq 2$. For $3 \leq j < k$, we have
		\begin{align*}
			\|\kappa_{k - j}\|_{L^\infty} \leq  \|\ka_{k - j +1 } \|_{L^1} \leq \ell^{\frac{1}{2}} \| \ka_{k - j + 1} \|_{L^2} \leq \ell^\half C_{k - 2}(\Omega_0),
		\end{align*}
		and for $j = k \geq 3$, $\ka_0$ is bounded in $L^\infty$ by Lemma \ref{kappa0}. We see that all derivatives are bounded in $L^\infty$ and hence, also in $L^1$, except for potentially $\ka_{k-2}$. Applying the Cauchy-Schwarz inequality $\|\ka_{k-2}\|_{L^1} \leq \ell^{\frac{1}{2}} \|\ka_{k-2}\|_{L^2}$, the right hand side is bounded by our induction hypothesis. It follows that the $L^1$-norm of any monomial in $\mathcal{R}_k$ can be bounded in terms of $\| \ka_{k-2} \|_{L^2}$ and the $L^\infty$-norms of lower order derivatives, all of which are bounded, in view of our induction hypothesis, on $\mathcal{M}(\Omega_0)$.}
	\end{proof}

	\red{
	In view of the well known Sobolev embedding $H^m \hookrightarrow C^k$ for $m - n/2 > k$ and the corresponding inequality
	\begin{align*}
		\|u\|_{C^k} \leq C(n, m, k) \sum_{j \leq m} \|\d^j u\|_{L^2} \leq C'(\M(\Omega), n, m, k).
	\end{align*}
	We conclude that the curvature is bounded in every $C^k$-norm on any marked length isospectral set. The Arzel\`a-Ascoli theorem then implies that the $C^\infty$-closure of $\mathcal{M}(\Omega_0)$ is \textit{compact} in $C^\infty$. We now show that $\mathcal{M}(\Omega_0)$ is in fact closed itself, from which Theorem \ref{cpt} follows.
	}

	\red{
	\begin{prop}
		\label{prop: MLIS is closed}
		For any $\Omega_0 \in \mathcal{B}$, the marked length isospectral set $\mathcal{M}(\Omega_0)$ is closed with respect to the $C^\infty$-topology on boundary curvatures.
	\end{prop}
	
	\begin{proof}
		Observe that the lower bound in Lemma \ref{kappa0} together with the upper bounds on Sobolev (and hence $C^k$) norms in Proposition \ref{unifbdd} show that any limit point of the marked length isospectral set is both smooth and strictly convex (with positive curvature), i.e. a Birkhoff billiard table. In other words, the $C^\infty$-closure of $\mathcal{M}(\Omega_0)$ is completely contained within $\mathcal{B}$. To show that $\mathcal{M}(\Omega_0)$ is closed itself, we need to show that the marked length spectrum of a limit point coincides with that of any approximating sequence. It is clear that $C^\infty$-convergence of curvatures implies $C^0$ (in fact, $C^\infty$)-convergence of generating functions, and continuity of the marked length spectrum as a map on $C^\infty$ domains follows directly from Lipschitz continuity of Mather's $\beta$-function with respect to the underlying generating functions in Proposition \ref{prop: Lipschitz dependence of beta}.
	\end{proof}
	}

	\section{Geometric and combinatorial preliminaries}\label{Geometric and combinatorial preliminaries}

	\subsection{Curvature coordinates}\label{curvecoord}
	
	Let us now choose a convenient coordinate system in which the curvature does not involve derivatives of a parametrization. Following \cite{MM}, we may rotate and translate our domain $\Omega$ so that it is tangent to the horizontal axis at the origin.  Denoting by $ds$ the arclength measure along $\d \Omega$ with $0 \leq s \leq |\d \Omega| := \ell$, we will call $\ka(s)$ the curvature and $\rho(s) = 1/\ka(s)$ the radius of curvature. {Denote by} $\theta$ the angle \red{between} the positively oriented tangent line {and} the horizontal axis. \red{It is clear that $s \in \R/\ell \Z$ and $\theta \in \R/2\pi \Z$ are in one-to-one correspondence, so we may freely change coordinates by writing $s(\theta)$ or $\theta(s)$. We can then parametrize $\d \Omega$ with respect to arclength by
	\begin{align}\label{eq: arclength parametrization}
		x(s) = \int_0^{s} \cos \theta(\tilde s) d \tilde s, \quad
		y(s) = \int_0^{s} \sin \theta(\tilde s) d \tilde s, \quad s \in \R/\ell \Z.
	\end{align}
	Differentiating \eqref{eq: arclength parametrization} with respect to $s$ shows that the tangent vector $(x'(s), y'(s))$ makes an angle of $\theta(s)$ with respect to the $x$ axis; the boundary $\d \Omega$ is then the envelope of these tangent lines. Note that $\theta'(s) = \ka(s)$, so we may equivalently write the above parametrization with respect to $\theta$ as
	\begin{align}
		\label{eq:radiusofcurvature coordinate}
		x(s(\theta)) = \int_0^{\theta} \rho(s(\tilde\theta)) \cos \tilde \theta d \tilde \theta, \quad
		y(s(\theta)) = \int_0^{\theta} \rho(s(\tilde{\theta})) \sin \tilde{\theta} d \tilde \theta,
	\end{align}
	for $\theta \in \R/2\pi\Z.$}
	The coordinates $(x(s), y(s))$ are called \textit{curvature coordinates} for $\d \Omega$. See Figure \ref{Billiard Table}.

	\subsection{Mather's-$\alpha$, $\beta$-functions and Marvizi-Melrose invariants}
	
	\red{In this section, we describe Marvizi and Melrose's perspective on convex billiards by means of an interpolating Hamiltonian and relate this approach to Mather's mean minimal action. Let us first note that the phase space for the billiard map is the coball bundle of the boundary $B^*\d  \Omega$, which is topologically a cylinder (see Figure \ref{Cylinder}).}

	\begin{theo}[\cite{MM}]
		If $\delta_\pm$ are boundary \red{(billiard)} maps of a strictly convex $C^\infty$ planar domain, there exists a $C^\infty$ function $\zeta \in C^\infty(T^* \d \Omega)$ which is a defining function for the positive half $S_+^*\d \Omega$ of the cosphere bundle \red{(i.e., $\zeta \geq 0$ in $B^* \d \Omega$ and $\{\zeta = 0\} = S_+^*\d \Omega$; see Figure \ref{Cylinder})} and satisfies
		\begin{align}\label{33}
			\delta_\pm \circ \exp\left( \pm \zeta^{1/2} X_\zeta\right) = \rho_\pm,
		\end{align}
		\red{where} $X_\zeta$ is the Hamiltonian vector field of $\zeta$ \red{and $\rho_\pm$ are} $C^\infty$ maps near $S_+^*\d \Omega$ which fix $S_+^* \d \Omega$ to infinite order. The Taylor series of $\zeta$ at $S_+^* \d \Omega$ is uniquely determined by the requirement \eqref{33}.
		\label{thm: interpolating hamiltonian}
	\end{theo}

	\begin{def1}\label{def:interpolating Hamiltonian}
		Such a function $\zeta$ is called an \textbf{interpolating Hamiltonian}.
	\end{def1}
	The idea is that \red{for $\tau \in \R$,} $\exp\left(\pm \tau \zeta^{1/2} X_\zeta \right)$ provides a continuous time flow which interpolates the discrete time billiard \textit{maps}, locally near glancing directions.
	
	\begin{rema}
		\red{If $\rho_\pm$ were to coincide with the identity in some open strip containing the boundary, then $\zeta$ would be a first integral for the billiard map and its level sets would constitute a foliation of $B^*\d\Omega$ by invariant curves, i.e., the billiard map would be locally integrable in the sense of Liouville-Arnold. As mentioned in Section \ref{Billiards}, this is false in general and conjecturally holds only for ellipses. However, this theorem does show that the billiard map is always \textit{nearly} integrable close to the boundary in phase space. Popov and Kovachev have applied the KAM theorem in this setting to both rederive and generalize the results of Lazutkin on the existence of invariant curves for the billiard map \cite{kovachev1990invariant}.}
	\end{rema}
	
	\begin{rema}
		\red{Writing out the Hamiltonian vector field $X_\zeta$ explicitly in Darboux coordinates, it is clear that
		\begin{align*}
			\zeta^{1/2} X_\zeta = X_{\frac{2}{3} \zeta^{3/2}},
		\end{align*}
		so to infinite order at the boundary, $\dt_\pm$ can be written as the \textit{time-one} maps of the Hamiltonian flow of a singular Hamiltonian $-\frac{2}{3} \zeta^{3/2}$. The square-root-type singularity of $\zeta^{3/2}$ at the boundary reflects an intrinsic geometric aspect of the billiard map; the reflection maps $\xi_\pm \mapsto \xi_\mp$ (in the notation of Section \ref{Billiards}) have a \textbf{Whitney-fold} over the glancing set $\{\sigma = \pm 1\}$ and convexity of the boundary implies that the glancing region is ``simple'' \cite{MelroseTaylor}. In this setting, Melrose derived a symplectic normal form for the hypersurfaces $S^* \R^2$ (which is the characteristic variety of the classical Hamiltonian $p(x,\xi) = |\xi|^2 - 1$) and $T_{\d \Omega}^*(\R^2)$ near the glancing region, where the Hamiltonian vector field of $p$ is tangent to $\d \Omega$ \cite{MelroseGlancingHypersurfaces1}, \cite{MelroseGlancingHypersurfaces2}. This leads to the existence of an interpolating Hamiltonian as in Theorem \ref{thm: interpolating hamiltonian}.}
	\end{rema}
	
	On {each level set $\{\zeta = t\} \subset \overline{B^*\d \Omega}$}, we define a one-form
	\begin{align*}
		dz = \frac{d  \zeta}{|d \zeta|^2} \circ J,
	\end{align*}
	where $J: T^*\d \Omega \to T^*\d \Omega$ is a clockwise rotation by $\pi/2$. It is dual to $X_\zeta$ {in the sense} that $dz(X_\zeta) = 1$.
	\begin{def1}\label{action}
		The \textbf{action integral} of nearly glancing orbits is {defined to be}
		\begin{align*}
			\I(t) = \int_{\zeta = t} dz
		\end{align*}
		\red{and the coefficients of its Taylor expansion at $t = 0$ are called \textbf{Marvizi-Melrose invariants:}
		\begin{align*}
			\I(t) \sim \sum_{k = 0}^\infty \frac{\I_{k+1}}{k!} t^k.
		\end{align*}
		}
	\end{def1}

	{To explain} the connection with Mather's $\beta$-function, {we first recall a result of Kovachev and Popov, adapted to our normalization conventions:}
	
	\begin{theo}[Theorem 2 in \cite{kovachev1990invariant}] \label{thm:Kovachev Popov}
		\red{Let $\Omega$ be a smooth strictly convex domain of boundary length $\ell$. Then there exists $\eps_0 > 0$ and an exact symplectic change of coordinates
			\begin{align*}
				\Phi: \d \Omega \times [1-\eps_0,1] \subset B^*\d \Omega \to& \R/\Z \times [\ell - \ell \eps_0, \ell]\\
				(s, \sigma) \mapsto& (\regphi, I),
			\end{align*}
			with $\regphi = s / \ell$ such that
			\begin{align*}
				\sigma ds = I d \regphi + dF,
			\end{align*}
			for some $F \in C^\infty (B^*\d \Omega)$ and in these coordinates, the billiard map
			$$
			\Phi \circ \dt_+ \circ \Phi^{-1} : (\regphi_0, I_0) \mapsto (\regphi_1, I_1)
			$$
			is generated by
			\begin{equation*}
				S(\regphi_0, I_1) = \regphi_0 I_1 - K(I_1)^{3/2} + R(\regphi_0, I_1),
			\end{equation*}
			i.e.,
			\begin{equation}
				\label{eq: KP generating function}
				\begin{cases}
					I_0 = \partial_1 S = I_1 + \partial_1 R, \\
					\regphi_1 = \partial_2 S = \regphi_0 - \frac{3}{2} K(I_1)^{1/2} K'(I_1) + \partial_2 R.
				\end{cases}
			\end{equation}
			Here, $K \in C^\infty(\R)$ with $K(\ell) = 0$, $K'(\ell) < 0$, and $R \in C^\infty(\R / \Z \times \mathbb{R})$. Moreover, there exists a Cantor set $\mathcal{C}^* \subset [\ell - \ell \eps_0, \ell]$ with $\ell \in \mathcal{C}^*$ such that $R \equiv 0$ on $\R / \Z \times \mathcal{C}^*$. The set $\mathcal{C}^*$ can be chosen so that for each $N \in \N$, there exists a positive constant $c_N$ such that
			\begin{align*}
				\ell \eps - c_N \eps^N \leq \text{Leb} \left (\mathcal{C}^* \cap [\ell - \ell \eps, \ell] \right), \qquad 0 \leq \eps \leq \eps_0.
		\end{align*}}
	\end{theo}
	\red{See also Theorem 3.2.20 in \cite{Si}. We now have the following} connection to Mather's $\beta$-function
	
	\begin{prop}\label{prop: Kovachev Popov coordinates and caustic length lazutkin expansion}
		{With respect to the} symplectic coordinates $(\regphi, I)$ {of Popov and Kovachev,} the function 
		$$
		\red{\zeta(s,\lambda)} = \left(\frac{3}{2} \alpha(- I(s, \lambda)) \right)^{\frac{2}{3}}
		$$
		is an interpolating Hamiltonian for the billiard map, where $\alpha$ is the convex conjugate (Legendre-Fenchel transform) of Mather's $\beta$-function. Furthermore,
		\begin{align*}
			 \red{\left|\Gamma_{\omega(Q)}\right| = |\d \Omega| + \int_0^{\zeta} \I (\eta) d \eta \Big|_{\zeta = \left(\frac{3}{2} Q \right)^{\frac{2}{3}}}},
		\end{align*}
		whenever ${\{(\regphi, I) : \alpha(-I) = Q\}}$ corresponds to a caustic $\Gamma_\omega$ of rotation {number} $\omega$ and corresponding Lazutkin parameter $Q$. In particular, $|\Gamma_{\omega(Q)}|$ has an asymptotic expansion in $t = \left(\frac{3}{2} Q\right)^{2/3}$ as $t$ (equivalently $Q$) tends to zero, with coefficients equal to \red{nonzero universal multiples of} the Marvizi-Melrose invariants $\I_k$:
		\begin{align*}
			\left|\Gamma_{\omega(Q)} \right| \sim \ell + \sum_{k=1}^\infty \frac{1}{k!}  \left(\frac{3}{2}\right)^{\frac{2k}{3}} \I_k Q^{\frac{2k}{3}}.
		\end{align*}
	\end{prop}

\begin{figure}
	\centering
	\includegraphics{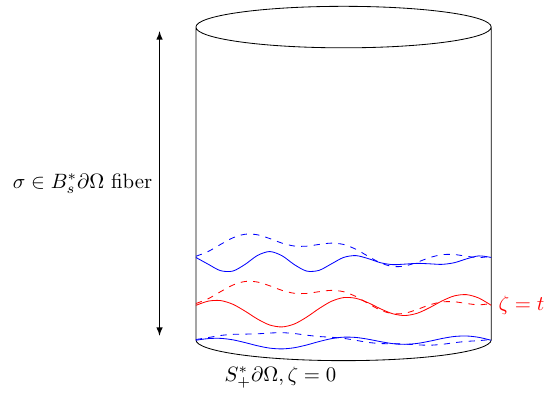}
	\caption{The phase space $\overline{B^* \d \Omega}$ of the billiard maps $\dt_\pm$. The red and blue curves are ``invariant tori,'' corresponding to caustics in the interior of $\Omega$. By Theorem \ref{thm:Siburg Lipschitz graph}, they are topologically circles and are given by graphs of Lipschitz functions of $s$. By Proposition \ref{prop: Kovachev Popov coordinates and caustic length lazutkin expansion}, they are also level sets of the interpolating Hamiltonian $\zeta$. The lower boundary $S_+^* \d \Omega$ belongs to the glancing region.}
	\label{Cylinder}
\end{figure}

	\begin{rema}
		The jet of $\I(t)$ at $t = 0$ (which corresponds to the {glancing set}), consists of marked length spectral invariants {which are equal} to the $\I_{k}$ in {Definition \ref{action}} and {formula} \eqref{SIK}. These are uniquely determined and do not depend on the choice of interpolating Hamiltonian. {From Proposition \ref{prop: Kovachev Popov coordinates and caustic length lazutkin expansion}, we also see that the invariants $\I_k$} are algebraically equivalent to the Taylor coefficients of Mather's $\beta$-function. 
	\end{rema}
	
	\red{Since a proof of Proposition \ref{prop: Kovachev Popov coordinates and caustic length lazutkin expansion} does not occur elsewhere in the literature, we use Theorem \ref{thm:Kovachev Popov} to provide one here.}
	
	\begin{proof}[Proof of Proposition \ref{prop: Kovachev Popov coordinates and caustic length lazutkin expansion}]
	
	\red{Our proof will use four sets of symplectic coordinates on $B^* \d \Omega$: $(s, \sigma)$, $(\regphi, I)$, $(z, \zeta)$ and $(\psi, J)$, which will be defined shortly. We begin with the symplectic coordinates $(\regphi,I)$ of Kovachev and Popov in Theorem \ref{thm:Kovachev Popov} above. Note that on an invariant curve $\Gamma$ of the billiard map, we have
	\begin{align*}
		|\Gamma| = \int_{\Gamma} \sigma ds = \int_\Gamma I d \regphi + dF = I \int_\Gamma d \regphi = I.
	\end{align*}
	Recall the identity $\alpha'(-|\Gamma|) = \omega$ from Theorem \ref{thm: caustic length lazutkin expansion}, where $\alpha$ is Mather's $\alpha$-function and $\omega$ is the rotation number of an orbit along $\Gamma$. Indeed, since $\alpha$ and $\beta$ are convex conjugates of one another,
	\begin{align*}
		\omega = \alpha'(\beta'(\omega)) = \alpha'(- |\Gamma|) = \alpha'(-I).
	\end{align*}
	This implies that $\regphi_0 I_1 - \alpha(-I_1)$ is another generating function for the billiard map $\dt_+$ on the Cantor set of invariant circles accumulating at $S_+^* \d \Omega$:
	\begin{align*}
		\begin{cases}
			I_0 = \d_1 \left(\regphi_0 I_1 - \alpha(- I_1)\right) = I_1,\\
			\regphi_1 = \d_2 \left(\regphi_0 I_1 - \alpha(- I_1) \right) = \regphi_0 + \omega.
		\end{cases}
	\end{align*}
	(See the discussion on pages 52-54 of \cite{Si}). Since $\alpha(- \ell) = K(\ell) = 0$, where $K$ is the smooth function appearing in Kovachev and Popov's generating function $S$ (Theorem \ref{thm:Kovachev Popov}), it follows that $K(I) = \alpha(-I)^{2/3}$ whenever $I \in \mathcal{C}^*.$ Recall also from Theorem \ref{thm: caustic length lazutkin expansion} that $\alpha(\beta'(\omega)) = \alpha(-|\Gamma|) = Q$, the Lazutkin parameter from Definition \ref{laz}, whenever there exists a caustic of rotation number $\omega$ and Lazutkin parameter $Q$.}
	\\
	\\
	\red{We now show that
	\begin{align*}
		\zeta(s,\lambda) = \left( \frac{3}{2} \alpha (-I(s, \lambda)) \right)^{\frac{2}{3}}
	\end{align*}
	is an interpolating Hamiltonian for the billiard map $\dt_+$. Indeed, the time $\tau$ flow of this Hamiltonian maps
	\begin{align*}
		\regphi \mapsto \regphi +\tau \frac{\d }{\d I} \left( \frac{3}{2} \alpha (- I(s, \lambda)) \right)^{\frac{2}{3}} = \regphi - \tau \left(\frac{2}{3}\right)^{\frac{1}{3}} \alpha^{- \frac{1}{3}}(-I) \frac{\d \alpha}{\d I}(-I).
	\end{align*}
	Setting $\tau = - \zeta^{1/2}$, we see that the billiard map sends $\regphi \mapsto \regphi + \omega$ in these coordinates, which verifies the claim on the Cantor set $\mathcal{C}^*$ from Theorem \ref{thm:Kovachev Popov}. Since the remainder $R$ in Theorem \ref{thm:Kovachev Popov} vanishes on $\mathcal{C}^*$, which accumulates at $I = \ell$ with super-polynomial density, its Taylor series at $I=\ell$ is zero. Hence, the difference between the time $-\zeta^{1/2}$ flow map and the billiard map $\dt_+$ is flat at the glancing set, which shows that $\left(\frac{3}{2} \alpha(-I)\right)^{2/3}$ is indeed an interpolating Hamiltonian as in Definition \ref{def:interpolating Hamiltonian}.}
	\\
	\\
	\red{Let us now find a further symplectic change of coordinates which allows us to write the length of a caustic in terms of our interpolating Hamiltonian.	 Following Section 6 in \cite{kovachev1990invariant}, we let $M_0 = \{s = s_0\}$ for some $s_0 \in \R/\ell \Z$. $M_0$ is a transversal curve to the Hamiltonian flow of $\zeta$. Now consider the flowout
	\begin{align*}
		M_0 \ni \rho_0 \mapsto \rho = \exp \tau X_\zeta \rho_0, \qquad \tau \in \R.
	\end{align*}
	Since the level sets $\{\zeta = t\}$ are immersed circles, the orbit $\exp \tau X_\zeta \rho$ is periodic. To find its period $T$, let $\gamma(\tau) = \exp \tau X_\zeta \rho_0$ parametrize the curve $\{\zeta = t\}$ with $0 \leq \tau \leq T$. Using $\dot \gamma = X_\zeta$ and the normalization $dz (X_\zeta) = 1$, we have
	\begin{align}\label{eq:period of exp}
		T = \int_0^T dz (X_\zeta) d \tau = \int_\gamma dz = - \int_{\zeta = t} dz = - \I(t),
	\end{align}
	($\gamma$ is parametrized with the opposite orientation compared to the interpolating Hamiltonian flow). In other words, each circle corresponds to a Hamiltonian orbit of period
	\begin{align*}
		- \I(\zeta(\rho_0)) = - \int_{\zeta = \zeta(\rho_0)} dz,
	\end{align*}
	which is minus the action integral from Definition \ref{action}.}
	\\
	\\
	\red{The inclusion map $\{\zeta = t\} \hookrightarrow B^*\d \Omega$ does not automatically extend $dz$ to a one-form on $B^*\d \Omega$ since its action on transversal vectors is a priori undefined. Nevertheless, $(z, \zeta)$ still defines a local coordinate system near $M_0$ on $B^*\d \Omega$; to each $\rho \in B^*\d \Omega$, we associate the coordinates
	\begin{align*}
		\zeta = \zeta(\rho), \qquad z = \int_{\rho_0}^{\rho} dz,
	\end{align*}
	where $\rho_0 = M_0 \cap \{\zeta = \zeta(\rho)\}$ and the $dz$ integral is taken along the segment of the curve $\{\zeta = t\}$ connecting $\rho_0$ to $\rho$. However, this coordinate system does \textit{not} extend globally. For different choices of $\rho_0 \in M_0$, the closed orbits $\{\exp \tau X_\zeta \rho_0\}$ have different periods. We instead introduce a further local symplectic change of variables which can be extended globally:
	\begin{align}
		\label{eq:Jpsi}
		(z, \zeta) \mapsto (\psi, \J) := \left(\frac{z}{\I(\zeta)}, \J(z, \zeta) \right) \in \R/\Z \times \R.
	\end{align}
	The new action variable $\J$ is defined so as to make the change of variables symplectic. We compute
	\begin{align}\label{dtheta}
		d \psi = - \frac{z \I'(\zeta)}{\I^2(\zeta)} d\zeta + \frac{1}{\I(\zeta)} dz, 
	\end{align}
	and set
	\begin{align}\label{new action}
		\J(\zeta) = \int_0^\zeta \I(\eta) d \eta,
	\end{align}
	to ensure that $d \J \wedge d \psi = d \zeta \wedge d z$ locally. Note that $\dim H_{\text{dR}}^1(B^* \d \Omega) = 1$. Since $\J d\psi$ and $\sigma ds$ are both primitives of the symplectic form, their cohomology classes differ by a multiple of the angular one-form $ds$. To compute this multiple, we integrate along $S_+^*\d \Omega$ to see that
	\begin{align*}
		 \int_{S_+^* \d \Omega} \sigma ds - \J d \psi = \int_{S_+^* \d \Omega} ds = |\d \Omega|,
	\end{align*}
	(using the fact that $\J$ vanishes on $S_+^*\d \Omega = \{\zeta = 0\}$), from which it follows that
	\begin{align*}
		[\sigma ds]_{\text{dR}} = [\J d\psi + ds]_{\text{dR}},
	\end{align*}
	where $[\cdot]_{\text{dR}}$ denotes the cohomology class. A related formula appears in Equation 6.2 of \cite{Amiran}. Note that $\J$ is also first integral of the Hamiltonian flow of $\zeta$, so we have
	\begin{align*}
		|\Gamma| = \int_{\Gamma} \sigma ds = \int_{\Gamma} \J d \psi + ds = |\d \Omega| + \J \int_{\Gamma} d \psi = |\d \Omega| + \int_0^{\zeta} \I(\eta) d \eta.
	\end{align*}
	The first part of the proposition then follows from writing $\zeta$ in terms of Mather's $\alpha$-function and recalling that $\alpha$ coincides with $Q$ on the set of lengths of caustics for the billiard map.} \red{For the second part of the proposition, we plug in $\zeta = (\frac{3}{2} Q)^{\frac{2}{3}}$, expand $\I(\eta)$ at $\eta = 0$, and integrate term by term:
		\begin{align*}
			|\Gamma| \sim |\d \Omega| + \sum_{k = 1}^\infty \int_0^\zeta \I_k \frac{\eta^{k-1}}{(k - 1)!} d \eta = |\d \Omega| + \sum_{k = 1}^{\infty} \frac{\I_k}{k!} \left(\frac{3}{2} Q\right)^{\frac{2k}{3}}.
		\end{align*}	
	}
		
	\end{proof}

	\begin{exam}
		\red{Let us compute the asymptotics described in Proposition \ref{prop: Kovachev Popov coordinates and caustic length lazutkin expansion} in the case of the unit disk $\Omega = B(0,1) \subset \R^2$. Caustics are concentric circles $\d B(0,r)$ for $0 < r < 1$ and elementary geometry shows that Mather's $\beta$-function is given by
		\begin{align}\label{eq:betadisk}
			\beta_{B(0,1)}(\omega) = - 2 \sin \pi \omega.
		\end{align}
		A short computation (\cite{Si}, pg. 48) shows that the convex conjugate of $\beta$ is given by
		\begin{align}\label{eq:alphadisk}
			\alpha(-I) = \frac{1}{\pi}\left( -I \arccos \left(\frac{I}{2\pi}\right) + \sqrt{ 4 \pi^2 - I^2}\right).
		\end{align}
		Recall that for a caustic $\Gamma$ of rotation number $\omega$, we have $\beta'(\omega) = -|\Gamma|$ and $\alpha(-|\Gamma|) = Q$, the Lazutkin parameter. Integrability of the billiard map in a disk is reflected in the smoothness of both \eqref{eq:betadisk} and \eqref{eq:alphadisk}. As a function of $(2\pi - I)$, the right hand side of \eqref{eq:alphadisk} has an asymptotic expansion in half integer powers, starting with a constant multiple of $(2\pi - I)^{3/2}$. Inverting this gives an expansion of $I = |\Gamma|$ in powers of $\alpha^{2/3}$ (equivalently, $Q^{2/3}$), as is described in Theorem \ref{thm: caustic length lazutkin expansion} and Proposition \ref{prop: Kovachev Popov coordinates and caustic length lazutkin expansion} above. Since the curvature of the boundary is constant for a disk, all of its derivatives vanish and one cannot find the coefficient in Theorem \ref{main} from the asymptotics of \eqref{eq:alphadisk} alone. However, one can directly check that the zeroth-order derivative terms do match those in Sorrentino's formulas \eqref{SIK}; see \eqref{eq:iotavalues} below for the precise values of $\iota_k$.}
	\end{exam}

	\subsection{Hamiltonian formulation}\label{subsec: Hamiltonian formulation}
	To compute the invariants $\I_k$, we use Darboux coordinates $(s,\lambda)$ with $0 \leq s \leq \ell$ being arclength along $\d \Omega$ and $\lambda = 1 - \sigma = 1 - \cos(\phi)$. This is simpler than the {non-symplectic} coordinates {$(s,\phi)$} used in \cite{MM}, {which would generate additional terms when differentiated in computations below}. The symplectic form is then given by $\omega = d s \wedge d \lambda$ and the Hamiltonian vector field is
	\begin{align}\label{eq:Hamiltonian vector field}
		X_\zeta = \frac{\d \zeta}{\d s} \frac{\d }{\d \lambda} - \frac{\d \zeta}{\d \lambda} \frac{\d }{\d s}.
	\end{align}
	The billiard map satisfies $\dtp \sim \exp(- \zeta^{1/2} X_\zeta)$ and maps $(s, \lambda)$ to $(s^+, \lambda^+)$. {Taylor expanding as $\lambda \to 0$ (i.e., near the glancing region $\{\sigma = 1\} = \{\zeta = 0\}$) gives}
	\begin{align}\label{sp}
		\begin{split}
			s^+ \sim \sum_{k = 0}^\infty \frac{(-1)^k}{k!} (\zeta^{1/2} X_\zeta)^k s + O(\lambda^{\infty}),
		\end{split}
	\end{align}
	with $s$ being the first coordinate function. {The notation $\sim$ signifies that
	\begin{align*}
		s^+ - \sum_{k = 0}^N \frac{(-1)^k}{k!} (\zeta^{1/2} X_\zeta)^k s = O_N(\lambda^\frac{N+1}{2}),
	\end{align*}
	as $\lambda \to 0$.} \red{The first few terms are given by
	\begin{align*}
		s^+  \sim s + \zeta^{1/2} \frac{\d \zeta}{\d \lambda} + \frac{1}{2} \zeta \left(\frac{\d \zeta}{\d \lambda} \frac{\d^2 \zeta}{\d s \d \lambda} - \frac{\d \zeta}{\d s} \frac{\d^2 \zeta}{\d \lambda^2}\right) + \cdots
	\end{align*}}
	The expansion \eqref{sp} is valid as a consequence of the spectral theorem for self-adjoint operators; $A = i \zeta^{1/2} X_\zeta$ is a self-adjoint differential operator with respect to the measure $dsd\lambda$, which allows us to define the unitary group $\exp(it A)$ via the functional calculus. Stone's theorem then guarantees that this operator is the unique operator with infinitesimal generator $A$, and hence coincides with pullback by the time $t$ flow map: for $f \in C^\infty(T^* \d \Omega)$, $\exp(- t \zeta^{1/2} X_\zeta) f = f(\Phi_{-t}(s, \lambda))$ where $\Phi_t$ is the Hamiltonian flow of $\frac{2}{3} \zeta^{3/2}$. Computation of the $k^{\text{th}}$-order operator $X_\zeta^k$ will be carried out combinatorially in Section \ref{Small lambda asymptotics}. This amounts to computing the coefficients of iterated Poisson brackets, or equivalently a Lie series.

\subsection{Cohomological considerations and curvature polynomials}
\label{subsec: cohomological considerations and curvature polynomials}

We {now provide an algorithm for reducing the structure of Marvizi-Melrose invariants to the form appearing in Theorem \ref{main}. Observe that for $1 \leq k \leq 4$,} the integral invariants $\I_k$ in \eqref{SIK} depend only on {derivatives of $\kappa$ up to order $k-1$. Each of the highest order derivatives} $\kappa_{k-1}$ {appears} quadratically (see Definition \ref{def:linearquadratic}) {and in Example \ref{rmk: nonuniqueness}, we showed how integration by parts allowed us to relegate all other terms to the remainder $\mathcal{R}_k$ in Theorem \ref{main}}. {We will show in Proposition \ref{intinv} below that from the action integral} $\I(t)$ in Definition \ref{action}, {we can read off the coefficient} $\I_{k+1}$ in terms of the Taylor coefficients of $\zeta$. These {in turn} will be shown in Sections \ref{Small lambda asymptotics} and \ref{Integral invariants} to depend on derivatives of $\kappa$ up to order $2k$. Moreover, the top order term $\kappa_{2k}$ will appear linearly, {meaning that it is multiplied by an undifferentiated power of $\kappa$ (Definition \ref{def:linearquadratic})}. \red{Our strategy will be to combine the formulas for $\I_k$ in terms of $\zeta$ and those relating $\zeta$ to the curvature $\kappa$, integrating by parts to reduce the order of the maximal derivatives until we arrive at the structure in Theorem \ref{main}.}
\\
\\
\red{Let us now develop an integration by parts algorithm for systematically reducing the order of the maximal derivative appearing in the integrand of each $\I_k$ while retaining the structure observed in the first five invariants. To avoid writing too many integrals, we instead work at the level of differential forms and note that since $\d \Omega$ is a closed curve, cohomologous one-forms integrate to the same quantity. We call this procedure \textbf{differentiation by parts}.} \red{In what follows, we will refer to polynomials in the variables $\{\ka^{\pm \frac{1}{3}}, \ka_1, \cdots, \ka_m\}$ as \textbf{curvature polynomials} or ``polynomials in the curvature jet,'' with the understanding that their coefficients are constant multiples of powers (including negative powers) of $\ka^{\frac{1}{3}}$. We also extend this to one-forms $\q ds$ whose coefficient $\q$ is a curvature polynomial in the sense just described.}

\begin{prop}\label{IBPalgorithm}
	Let $\q = \kappa_0^{p_0} \kappa_1^{p_1} \cdots \kappa_m^{p_m}$ be a monomial in the \red{variables $\{\ka^{\pm \frac{1}{3}}, \ka_1, \cdots, \ka_m\}$, having} differential degree $D \geq 2$ (cf. Definition \ref{diffdeg}) and \red{at least one factor of $\ka_i, i \geq 1$.} Denote by $\q ds$ the associated one-form, where $ds \in \wedge^1(\d \Omega)$ is the arclength one-form. \red{Suppose that $m, p_m \geq 1$ and let}
	\begin{align*}
		m^* =
		\begin{cases}
			\max \{ i : i < m, p_i \neq 0\}, & \text{if}\{ i : i < m, p_i \neq 0\} \neq \emptyset,\\
			\red{0}, & \red{\text{if }\{ i : i < m, p_i \neq 0\} = \emptyset}
		\end{cases}
	\end{align*}
	\red{be the largest derivative index strictly less than $m$, with $m^* = 0$ if no such derivative occurs} and set
	\begin{align*}
		e =
		\begin{cases}
			m^* (p_{m^*} -1) + \sum_{i = 1}^{{m^* - 1}} i p_i, & m^* \geq 1,\\
			0 & m^* = 0.
		\end{cases} 
	\end{align*}
	 \red{Then,}
	 \begin{enumerate}
	 	\item \red{If $p_m = 1$ and $m - m^*$ is even, then} $\q ds$ is cohomologous to a one-form
	 	$$
	 	\red{r ds = \ka_{\frac{D- e}{2}}^2 u ds + v ds},
	 	$$
	 	\red{where $u, v$ are polynomials in the curvature jet of differential degree at most $e$ and $D$ respectively, where $u$ contains $\kappa$-derivatives of order at most $m^*$ and $v$ contains $\ka$-derivatives of order at most $\frac{D-e}{2} - 1$.} \red{If in addition $m^* \geq 1$ and $p_{m^*} = 1$, then $u$ can be chosen to depend on at most $m^* - 1$ derivatives of $\ka$.}
	 	
	 	\item \red{If $p_m = 1$ and $m - m^*$ is odd, then $\q ds$ is cohomologous to a one-form $v ds$, where $v$ has differential degree at most $D$ and contains $\kappa$-derivatives of order less than or equal to $\frac{D - e - 1}{2}$.}
	 	
	 	\item \red{If $p_m  = 2$ and $m^* = 0$, then $\q ds$ is already of the form $\ka_m^2 \ka_0^{p_0} ds$.}
	 	
	 	\item \red{If $p_m = 2$ and $m^* \geq 1$, then $\q ds$ already contains $\ka$-derivatives of order strictly less than $\frac{D - e}{2}$.}
	 		
	 	\item \red{If $p_m \geq 3$, then $\q ds$ already contains $\ka$-derivatives of order strictly less than $\frac{D - e}{2}$.}

	 \end{enumerate}
\end{prop}

\begin{rema}
	Note that $e + m^* + m p_m = D$, so $e$ represents the excess number of derivatives left after separating out the maximal derivatives and a single copy of the second highest derivative. Furthermore, if $p_m = 1$, then $\frac{D-e}{2} = \frac{m + m^*}{2}$ is the \textit{midpoint} between the maximal and subleading derivative.
\end{rema}

\begin{proof}
	\red{We will prove the lemma case by case.}
	\\
	\\
	\red{\textbf{Case 1.} $m - m^*$ is even and $p_m = 1$. We have
	\begin{align*}
		\q ds =& \ka_0^{p_0} \ka_1^{p_1} \cdots \ka_{m^*}^{p_{m^*}} \ka_m ds
		\\
		=& - \kappa_{m - 1} d \left(\kappa_0^{p_0} \kappa_1^{p_1} \cdots \kappa_{m^*}^{p_{m^*}} \right) + d  \left(\kappa_0^{p_0} \kappa_1^{p_1} \cdots \kappa_{m^*}^{p_{m^*}} \kappa_{m-1} \right)
		\\
		=& - p_{m^*} \kappa_{m-1} \ka_{m^* + 1} \ka_{m^*}^{p_{m^*} - 1} \left(\kappa_0^{p_0} \kappa_1^{p_1} \cdots \kappa_{m^* - 1}^{p_{m^* - 1}} \right)ds
		\\
		& - \ka_{m-1} \ka_{m^*}^{p_{m^*}} d\left(\kappa_0^{p_0} \kappa_1^{p_1} \cdots \kappa_{m^* - 1}^{p_{m^* - 1}}\right)
		\\
		& + d \left(\kappa_0^{p_0} \kappa_1^{p_1} \cdots \kappa_{m^*}^{p_{m^*}} \kappa_{m-1} \right),
	\end{align*}	
	which reduces the order of the highest derivative by one, modulo an exact remainder {(i.e., one that integrates to zero over the closed curve $\d \Omega$).} If $m^* = 0$, the product over indices $1, \cdots, m^*-1$ in the exact term is interpreted as one and the exterior derivative $d$ lands only on $\ka^{p_0}$, producing a factor of $\ka_1$. The second highest derivative has also increased by one, so the total differential degree remains constant while the gap between $m^*$ and $m$ is reduced by two; $\q ds$ is cohomologous to a term of the form
	\begin{align*}
		r_1 ds = \ka_{m-1} \ka_{m^*+1} u_1 ds +   \ka_{m-1} v_1 ds + d w_1 ,
	\end{align*}
	where $u_1$ has differential degree at most $e$, $v_1$ has differential degree at most $D- m + 1 = e + m^* + 1$, and $d w_1$ is an exact differential. Neither $u_1$ nor $v_1$ contain any derivatives of $\kappa$ having order greater than $m^*$ and if $p_{m^*} =1$, then it is clear that $u_1$ depends on at most $m^* - 1$ derivatives of $\ka$. In contrast to our notation $\ka_m = \ka^{(m)}$, the use of subscripts on $r,u,v$ and $w$ is \textit{not} intended to imply differentiation; it is a way of indexing the functions. Repeating the procedure above, we see that $r_1 ds$ (and hence $\q ds$) is cohomologous to a term of the form
	\begin{align*}
		r_2 ds = \ka_{m-2} \ka_{m^* + 2} u_2 ds + \ka_{m-2} v_2 ds + d w_2,
	\end{align*}
	where $u_2 = u_1$ has differential degree at most $e$, depending on derivatives of $\kappa$ having order less than or equal to $m^*$ (or $m^* - 1$ in the case $p_{m^*} = 1$), and $v_2$ has differential degree at most $D - m + 2$, depending on no more than $m^* + 1$ derivatives of $\ka$. The term $d w_2$ is another exact differential. Using the fact that $m - m^*$ is even, after finitely many (in fact, ${\frac{m-m^*}{2}}$) repetitions of this procedure, $\q ds$ is seen to be cohomologous to a term of the form
	\begin{align}\label{eq:D-e over 2 iterate}
		r_{\frac{m-m^*}{2}} ds = \ka_{\frac{m+m^*}{2}}^2 u_{\frac{m-m^*}{2}} ds + \ka_{\frac{m+m^*}{2}} v_{\frac{m-m^*}{2}} ds + d w_{\frac{m-m^*}{2}},
	\end{align}
	where $u = u_{\frac{m-m^*}{2}} = u_1$ has differential degree at most $e$, depending on derivatives of order less than or equal to $m^*$ (or $m^* - 1$ in the case $p_{m^*} = 1$ and $m^* \geq 1$), and $v_{\frac{m-m^*}{2}}$ has differential degree at most $D - m + {\frac{m-m^*}{2}}$, depending on $\ka$ derivatives of order less than or equal to $m^* + {\frac{m-m^*}{2}} - 1$. Recall that $e + m^* + p_m m = D$ and hence, if $p_m = 1$, then
	\begin{align*}
		\frac{m+m^*}{2} = \frac{D - e}{2}
	\end{align*}
	is the {midpoint} between the highest $(m)$ and second highest $(m^*)$ order of $\ka$-derivative appearing in $\q$. The first term $\ka_{\frac{D-e}{2}}^2 u$ appears in Part 1 of the lemma. $v_{\frac{m-m^*}{2}}$ depends on at most $\frac{D- e}{2} - 1$ derivatives of $\ka$ and contains two possible types of terms which we must consider separately.}
	\\
	\\
	\red{\textbf{Type A terms. If some monomial $f$ in $v_{\frac{m-m^*}{2}}$ contains a term of the form $\ka_{\frac{m+m^*}{2} - 1}^{p}$ for some exponent $p \geq 1$, then its appearance in the second term of \eqref{eq:D-e over 2 iterate} can be reduced to the form
	\begin{align}\label{eq:vtilde}
		\begin{split}
			\ka_{\frac{m+m^*}{2}} f ds =& \frac{1}{p+ 1} d 	\left(\ka_{\frac{m+m^*}{2} - 1}^{p + 1} \right) \wt f
			\\
			=& - \frac{1}{p + 1}  \ka_{\frac{m+m^*}{2} - 1}^{p+ 1} d \wt f
			+  d \left(\frac{1}{p+1} \ka_{\frac{m+m^*}{2} - 1}^{p + 1} \wt f \right),
		\end{split}
	\end{align}}
	where $\wt f$ consists of the remaining terms in $f$ (besides $\ka_{\frac{m+m^*}{2} - 1}^{p}$), having differential degree at most $\frac{D+e}{2} - p \left(\frac{m-m^*}{2} - 1 \right)$ and depending on $\ka$ derivatives of degree less than or equal to $\frac{m+m^*}{2} - 2$. The first term in the second line of \eqref{eq:vtilde} has differential degree at most $D$.	Let us write this as
	\begin{align}
	\label{eq:wh v and wh w}
	\ka_{\frac{m+m^*}{2}} f ds = {\wh f} ds + d \wh w,
	\end{align}
	where $\wh f$ has differential degree at most $D$, containing less than or equal to $\frac{D-e}{2} - 1$ derivatives of $\ka$, and $d \wh w$ is an exact differential.}
	\\
	\\
	\red{\textbf{Type B terms.}  If some monomial $g$ in $v_{\frac{m-m^*}{2}}$ does \textit{not} contain a factor of $\ka_{\frac{m+m^*}{2} - 1}$, then its appearance in the second term in \eqref{eq:D-e over 2 iterate} is of the form
	\begin{align}
		\label{eq:typeB}
		\begin{split}
			\ka_{\frac{m+m^*}{2}} g ds =& - \ka_{\frac{m+m^*}{2} - 1} dg + d \left(\ka_{\frac{m+m^*}{2} - 1} g \right)
			\\
			=& \wh g ds + d \wh {w'},
		\end{split}
	\end{align}
	where $\wh g$ has differential degree at most $D$, containing less than or equal to $\frac{D-e}{2} - 1$ derivatives of $\ka$, and $d \wh {w'}$ is another exact differential.
	\\
	\\
	Combining \eqref{eq:wh v and wh w} and \eqref{eq:typeB} with \eqref{eq:D-e over 2 iterate}, we see that both Type A and Type B terms satisfy the requirements of $v ds$ in Part (1) of the lemma.}
	\\
	\\
	\red{\textbf{Case 2: $m - m^*$ is odd and $p_m = 1$.} If $m - m^* = 1$, then $\q ds$ is of the form
		\begin{align}
			\label{eq:gap size one}
			\begin{split}
			\ka_m \ka_{m-1}^{p_{m-1}} \ka_{m-2}^{p_{m-2}} \cdots \ka_1^{p_1} \ka_0^{p_0} ds
			= &
			\frac{1}{p_{m-1} + 1} d \left(\ka_{m-1}^{p_{m-1} + 1}\right) \ka_{m-2}^{p_{m-2}} \cdots \ka_1^{p_1} \ka_0^{p_0}
			\\
			= & - \frac{1}{p_{m-1} + 1} \ka_{m-1}^{p_{m-1} + 1} d \left(\ka_{m-2}^{p_{m-2}} \cdots \ka_1^{p_1} \ka_0^{p_0}\right)
			\\
			& + d  \left(\frac{1}{p_{m-1} + 1} \ka_{m-1}^{p_{m-1} + 1} \ka_{m-2}^{p_{m-2}} \cdots \ka_1^{p_1} \ka_0^{p_0}\right).
		\end{split}
		\end{align}
		These are exactly the Type A terms handled in Case 1. Modulo the exact term
		$$
		d  \left(\frac{1}{p_{m-1} + 1} \ka_{m-1}^{p_{m-1} + 1} \ka_{m-2}^{p_{m-2}} \cdots \ka_1^{p_1} \ka_0^{p_0}\right),
		$$
		all derivatives of $\ka$ which appear in \eqref{eq:gap size one} are of order less than or equal to $\frac{D - e - 1}{2}$ and all monomials have differential degree at most $D$, hence satisfying the requirements in the lemma for $m - m^*$ odd.
		\\
		\\
		If $m - m^* \geq 3$ is odd, we repeat the procedure described in Case $1$, but now stopping when the difference between the maximal and second highest derivatives is one. In this case, $\q ds$ is cohomologous to a term of the form
	\begin{align*}
		\q ds =&
		\ka_{m-1} \ka_{m^*+1} \wh u_1 ds +   \wh v_1 ds + d \wh w_1
		\\
		=& \ka_{m-2} \ka_{m^*+2} \wh u_2 ds +   \wh v_2 ds + d \wh w_2
		\\
		=& \cdots
		\\
		=& \ka_{\frac{D - e + 1}{2}} \ka_{\frac{D - e - 1}{2}} \wh u_{\frac{D - e - 1}{2} - m^*} ds + \wh v_{\frac{D - e - 1}{2} - m^*} ds + d \wh w_{\frac{D - e - 1}{2} - m^*},
	\end{align*}
	where $\wh u_i, \wh v_i$ satisfy the same properties as $u_i, v_i$ in Case $1$ and $d \wh w_i$ are exact differentials. The same reasoning as in the $m - m^* = 1$ case applies, giving
	\begin{align*}
		\q ds  = & \frac{1}{2} d \left(\ka_{\frac{D - e -1}{2}}^2 \right) \wh u_{\frac{D - e - 1}{2} - m^*} + \wh v_{\frac{D - e - 1}{2} - m^*} ds + d \wh w_{\frac{D - e - 1}{2} - m^*} 
		\\
		= & - \frac{1}{2} \ka_{\frac{D - e -1}{2}}^2  d \left( \wh u_{\frac{D - e - 1}{2} - m^*}  \right) + \wh v_{\frac{D - e - 1}{2} - m^*} ds + d \wh w_{\frac{D- e - 1}{2} - m^*} 
		\\ & + d \left(\frac{1}{2} \kappa_{\frac{D - e - 1}{2}}^2 \wh u_{\frac{D - e - 1}{2} - m^*} \right).
	\end{align*}
	Each term above is either an exact differential or has differential degree at most $D$, containing only $\ka$-derivatives of order less than or equal to $\frac{D -e - 1}{2}$.}
	\\
	\\
	\red{\textbf{Case 3.} If $p_m  = 2$ and $m^* = 0$, then by the definition of $m^*$, $\q ds$ is of the form $\ka_m^2 \ka_0^{p_0} ds$. This is the same as the maximally reduced one-forms appearing terms in Case 1 if the excess $e$ is zero.}
	\\
	\\
	\red{\textbf{Case 4: $p_m = 2$ and $m^* \geq 1$}. The identity $e + m^* + p_m m= D$ implies that
	\begin{align*}
		m = \frac{D - e - m^*}{p_m} < \frac{D - e}{2},
	\end{align*}
	so $\q ds$ is already in reduced form.
	\\
	\\
	\textbf{Case 5: $p_m \geq 3$}. Again, the equality $e + m^* + p_m m = D$ again implies that
	\begin{align*}
		m = \frac{D - e - m^*}{p_m} \leq \frac{D - e}{3} < \frac{D - e}{2},
	\end{align*}
	and once again, $\q ds$ is already in reduced form.}
\end{proof}

\begin{def1}
	\label{def:excess differential gap reduced form}
	We will call $e$ the \textbf{excess} of a monomial and $m - m^*$ the \textbf{differential gap}. A monomial in the curvature jet is said to be in \textbf{\red{reduced form}} if its highest derivative appears at least quadratically \red{(i.e., raised to a power strictly greater than one)}.
\end{def1}

\red{
\begin{rema}
	If a monomial one-form in the curvature jet is in reduced form, further application of the differentiation by parts algorithm described in the proof of Proposition \ref{IBPalgorithm} will only increase the order of maximal derivatives.
\end{rema}
}

\begin{exam}\label{ex: monomial 1}
	\red{Consider the monomial $\kappa \kappa_3 \kappa_7$. It has differential degree $D = 10$, excess $e = 0$, and differential gap $m - m^* = 7-3 = 4$. Hence, it can be transformed as follows:
		\begin{align*}
			\ka_7 \ka_3 \ka ds
			= &
			d\left(\ka_6\right) \ka_3 \ka
			\\
			=&
			- \ka_6 d \left(\ka_3 \ka \right) + d \left(\ka_6 \ka_3 \ka \right)
			\\
			= &
			-\ka_6 \ka_4 \ka ds - \ka_6 \ka_3 \ka_1 ds + d \left(\ka_6 \ka_3 \ka \right)
			\\
			=&
			- d\left(\ka_5\right) \ka_4 \ka - d \left(\ka_5\right) \ka_3 \ka_1 + d \left(\ka_6 \ka_3 \ka \right)
			\\
			= &
			\ka_5 d \left(\ka_4 \ka \right) + \ka_5 d \left(\ka_3 \ka_1\right)
			+
			d \left(\ka_6 \ka_3 \ka - \ka_5 \ka_4 \ka - \ka_5 \ka_3 \ka_1 \right)
			\\
			=&
			\ka_5^2 \ka ds + \ka_5 \ka_4 \ka_1 ds + \ka_5 \ka_4 \ka_1 ds + \ka_5 \ka_3 \ka_2 ds
			\\
			& +
			d \left(\ka_6 \ka_3 \ka - \ka_5 \ka_4 \ka - \ka_5 \ka_3 \ka_1 \right)
			\\
			=&
			\ka_5^2 \ka ds +  d \left( \ka_4^2 \right) \ka_1 + d \left(\ka_4\right) \ka_3 \ka_2
			+
			d \left(\ka_6 \ka_3 \ka - \ka_5 \ka_4 \ka - \ka_5 \ka_3 \ka_1 \right)
			\\
			=&
			\ka_5^2 \ka ds - \ka_4^2 d\left( \ka_1\right) - \ka_4 d\left( \ka_3 \ka_2\right)
			\\
			& +
			d \left(\ka_6 \ka_3 \ka - \ka_5 \ka_4 \ka - \ka_5 \ka_3 \ka_1 + \ka_4^2 \ka_1 + \ka_4 \ka_3 \ka_2 \right)
			\\
			=&
			\ka_5^2 \ka ds - \ka_4^2 \ka_2 ds - \ka_4^2 \ka_2 ds - \ka_4 \ka_3^2 ds
			\\
			& +
			d \left(\ka_6 \ka_3 \ka - \ka_5 \ka_4 \ka - \ka_5 \ka_3 \ka_1 + \ka_4^2 \ka_1 + \ka_4 \ka_3 \ka_2 \right).
		\end{align*}
		Implementation of the algorithm in Proposition \ref{IBPalgorithm} involved stripping off only two derivatives from $\ka_7$ because that increased the second highest derivative $\ka_3$ to $\ka_5$ while decreasing that of $\ka_7$ to $\ka_5$. Moving any additional derivatives from $\ka_7$ to $\ka \ka_3$ would increase the order of the latter to six. As in the statement of Proposition \ref{IBPalgorithm}, the highest derivative has order equal to one half of the differential degree and it appears quadratically:
		\begin{align*}
			\int_0^\ell \ka_7 \ka_3 \ka ds = \int_0^\ell \left(\ka_5^2 \ka- 2 \ka_4^2 \ka_2 - \ka_4 \ka_3^2\right) ds.
		\end{align*}
		}
		
\end{exam}

Moving forward, we will obtain the invariants $\I_k$ as integrals of polynomials in the curvature jet with coefficients in \red{the ring} $\R[\ka^{\pm 1/3}]$. The form of these integrals is not unique, as one can integrate by parts arbitrarily many times, corresponding to cohomologous one-forms \red{(cf. Example \ref{rmk: nonuniqueness})}.

\begin{def1}
	\label{def:linearquadratic}
	We will call a curvature \red{monomial} \textbf{linear} \red{or \textbf{first order}} if \red{it is of the form $c \ka_0^{p_0} \ka_m$ for some $c \in \R, p_0 \in \frac{1}{3} \Z$ and $m \in \N_{> 0}$, i.e. it has a single term which is differentiated and it occurs with a power of one. Similarly, we will call a curvature monomial \textbf{second order}} if it is of the form $c \ka_0^{p_0} \ka_i \ka_j$ for some $i, j > 0$ and \red{\textbf{higher order} if there are \red{at least three} differentiated terms, i.e., for a monomial $\ka_0^{p_0} \ka_1^{p_1} \cdots \ka_m^{p_m}$,
	\begin{align*}
		\sum_{i = 1}^m p_i = & 2 \iff \text{second order,}
		\\
		\sum_{i = 1}^m p_i \geq& 3 \iff \text{higher order.}
	\end{align*}
}
	\red{The $i^{\text{th}}$ derivative is said to appear \textbf{quadratically} if the monomial is of the form $c \ka_0^{p_0} \ka_i^2$.} \red{Higher order terms will be denoted $\hot$.}
\end{def1}

	 {If the excess $e$ is positive, then Proposition \ref{IBPalgorithm} allows us to reduce the order of the maximal derivative to $\frac{D - e}{2}$, which is \textit{strictly} less than one half the differential degree.} The key observation is the following.
	 
\begin{coro}\label{threeterms}
	Amongst all curvature-jet \red{monomial} one-forms having differential degree $D$, those with at least three differentiated \red{factors} are cohomologous to another which has at most the same differential degree and contains only derivatives of order strictly less than $D/2$.
\end{coro}

\begin{proof}
\red{Suppose that $\q = \ka_0^{p_0} \ka_1^{p_1} \cdots \ka_{m}^{p_m}$ is a monomial of differential degree $D$. If there are at least three differentiated factors (including multiplicity greater than one), denote the highest by $m$, the second highest by $m^*$ and the third highest by $m^{**}$. If $p_m \geq 3$ or if $p_m = 2$ and $m^* \geq 1$, then Parts 4 and 5 of Proposition \ref{IBPalgorithm} tell us that $\q$ contains $\ka$-derivatives of order strictly less than $\frac{D - e}{2} \leq \frac{D}{2}$. If $p_m = 1$ and there are two other $\ka$-derivatives, then the excess $e$ is positive, in which case Proposition \ref{IBPalgorithm} again tells us that $\q ds$ is cohomologous to a one-form $r ds$ where $r$ has differential degree at most $D$ and contains $\ka$-derivatives of order at most $\frac{D- e}{2} < \frac{D}{2}$.}
\end{proof}

\red{This corollary is very useful for the following reason: when integrating the Marvizi-Melrose integral invariants $\I_k$ by parts to obtain the structure postulated in Theorem \ref{main}, we can relegate all terms with at least three differentiated terms and differential degree at most $2k - 2$ to the remainder $\mathcal{R}_k$. Let us see an example which illustrates this.}

\begin{exam}\label{ex: monomial 2}
	\red{Consider the monomial $\kappa_1^3 \kappa_7$. Just like the monomial in Example \ref{ex: monomial 1}, this has differential degree $10$, but there are now \textit{four} differentiated terms: $\ka_7$ and three copies of $\ka_1$. The excess is $e = 2$ and the differential gap is again $m - m^* = 7-1 = 6$. The monomial is transformed via ``differentiation by parts'' as follows:
	\begin{align*}
		\ka_1^3 \ka_7 ds
		=& 
		- d (\ka_1^3) \ka_6 + d \left(\ka_1^3 \ka_6\right)
		\\
		= & -3 \ka_1^2 \ka_2 \ka_6 ds + d \left(\ka_1^3 \ka_6\right)
		\\
		= & d\left(3 \ka_1^2 \ka_2\right) \ka_5 + d \left(\ka_1^3 \ka_6 - 3 \ka_1^2 \ka_2 \ka_5 \right)
		\\
		= & 6 \ka_1 \ka_2^2 \ka_5 ds + 3 \ka_1^2 \ka_3 \ka_5 ds + d \left(\ka_1^3 \ka_6 - 3 \ka_1^2 \ka_2 \ka_5 \right)
		\\
		= & - 6 d (\ka_1 \ka_2^2) \ka_4 - 3 d \left(\ka_1^2 \ka_3 \right) \ka_4
		\\
		& + d \left(\ka_1^3 \ka_6 - 3 \ka_1^2 \ka_2 \ka_5 + 6 \ka_1 \ka_2^2 \ka_4 + 3 \ka_1^2 \ka_3 \ka_4 \right)
		\\
		=& -6 \ka_2^3 \ka_4 ds - 12 \ka_1 \ka_2 \ka_3 \ka_4 ds - 6 \ka_1 \ka_2 \ka_3 \ka_4 ds - 3 \ka_1^2 \ka_4^2 ds
		\\
		& + d \left(\ka_1^3 \ka_6 - 3 \ka_1^2 \ka_2 \ka_5 + 6 \ka_1 \ka_2^2 \ka_4 + 3 \ka_1^2 \ka_3 \ka_4 \right)
		\\
		= & 6 d \left(\ka_2^3\right) \ka_3 - 9 \ka_1 \ka_2 d \left(\ka_3^2\right)- 3 \ka_1^2 \ka_4^2 ds
		\\
		& + d \left(- 6 \ka_2^3 \ka_3 + \ka_1^3 \ka_6 - 3 \ka_1^2 \ka_2 \ka_5 + 6 \ka_1 \ka_2^2 \ka_4 + 3 \ka_1^2 \ka_3 \ka_4 \right)
		\\
		= & 18 \ka_2^2 \ka_3^2 ds + 9 \ka_2^2 \ka_3^2 ds + 9 \ka_1 \ka_3^3 ds - 3 \ka_1^2 \ka_4^2 ds
		\\
		& + d \left(- 6 \ka_2^3 \ka_3 - 9 \ka_1 \ka_2 \ka_3^2+ \ka_1^3 \ka_6 - 3 \ka_1^2 \ka_2 \ka_5 + 6 \ka_1 \ka_2^2 \ka_4 + 3 \ka_1^2 \ka_3 \ka_4 \right).
	\end{align*}
	The positive excess allowed us to reduce the highest order derivative to \textit{less than half} of the differential degree: $\frac{D - e}{2} = 4 < \frac{D}{2} = 5$.
	}
\end{exam}

\begin{def1}
	\label{def:harmless}
	\red{A polynomial $\q$ in the curvature jet is said to be \textbf{$(D,i)$-harmless}, $D,i \in \N$, if the differential degree of $\q$ is at most $D$ and all of its constituent monomials which have differential degree \textit{equal} to $D$ contain at least $i$ factors of $\kappa$ derivatives having order at least one. We will write $F = G + \mathcal{H}_{D,i}$ to say that $F - G$ is $(D,i)$-harmless.}
\end{def1}

\red{
In view of Corollary \ref{threeterms}, $(D,i)$-harmless terms can often be absorbed into the remainder $\mathcal{R}_m$ in Theorem \ref{main}.

\begin{coro}
	\label{cor:harmlessremainders}
	If a curvature polynomial $\q$ is $(D,i)$-harmless with $i \geq 3$, then $\q ds$ is cohomologous to another polynomial one-form $u ds$, where $u$ has differential degree at most $D$ and all derivatives in $u$ have order strictly less than $D/2$.
\end{coro}

}

\red{The algorithm described in Proposition \ref{IBPalgorithm} is perhaps more important than the statement of the proposition itself. To recapitulate, the strategy involves peeling derivatives off the highest derivative of $\ka$ and applying them to the second highest derivative at the expense of an exact remainder. Assuming the maximal derivative appears with an exponent of one, this decreases the highest order derivative by one and increases the second highest derivative by one. After finitely many iterations, one arrives at the midpoint between the highest and second highest derivatives and that term is squared.}

	\section{Small $\lambda$ asymptotics of $\delta_{+}$}
	\label{Small lambda asymptotics}
	
	\red{Recall the definition of the interpolating Hamiltonian $\zeta$ in Theorem \ref{thm: interpolating hamiltonian}. It} is a smooth function of $\lambda = 1 - \sigma$ \red{(where $\sigma = \cos \phi$, cf. Sections \ref{Billiards} and \ref{subsec: Hamiltonian formulation})}, or equivalently of $\phi^2$. \red{It} has a formal Taylor (Borel) expansion near $\lambda = 0$, given by
	\begin{align}
	\label{eq:taylor exp of zeta}
		\zeta(s , \lambda) \sim \sum_{i = 1}^{\infty} \zeta_i(s) \lambda^i.
	\end{align}
	Recall \red{also} that $t = (3Q/2)^{2/3}$, \red{a function of} the Lazutkin parameter \red{and hence,} the asymptotics $\lambda \to 0, t \to 0, \phi \to 0, \sigma \to 1$ are all equivalent. Equation \eqref{sp} will be used to expand $s^+$ in terms of $\lambda$. Expanding powers of $\zeta^{1/2} X_\zeta$ in $\lambda$ yields an asymptotic expansion of the form
	\begin{align}\label{eq: splus in terms of s and lambda}
		s^+ \sim s + \sum_{\red{M} = 1}^\infty \red{A_M(s)} \lambda^{M/2},
	\end{align}
	where the coefficients $\red{A_M(s)}$ depend nonlinearly on \red{the jet of} $\zeta_i$ for $1 \leq i \leq \left \lfloor \frac{M + 1}{2} \right \rfloor$, \red{with derivatives of order at most $M-1$ appearing} (see Proposition \ref{ZA} below). \red{Throughout this section, we will use the variable $M$ to index powers of $\lambda^{1/2}$ in \eqref{eq: splus in terms of s and lambda} and when $M$ is odd, we write $M = 2m - 1$. We will then use $K$ to index powers of the Hamiltonian vector field $X_\zeta$ and when $K$ is odd, we will write $K = 2k+1$.}
	\\
	\\
	The first two terms \footnote{There is a small misprint in their paper. The coefficient of $\ka_1^2$ in $\zeta_2$ should instead be $-32$.} were computed in \cite{MM} \red{(pg. 485, equation (4.5)):}
	\begin{align}\label{smallzeta}
		\begin{split}
			\zeta_1 &= 2 \kappa^{-2/3},\\
			\zeta_2 &= \frac{1}{15} \kappa^{-2/3} - \frac{32}{135} \kappa^{-14/3} \kappa_1^2  + \frac{8}{45} \kappa^{- {11}/{3}} \kappa_2,
		\end{split}
	\end{align}
	and by a related computation \footnote{The coefficient preceding $\phi^3$ is incorrect. Nonetheless, the formulas for $\I_1$ and $\I_2$ are essentially correct, although $\I_1$ should be divided by four and $\I_2$ should be multiplied by $2$.} adapted to the coordinate system $(s, \lambda)$,
		\begin{align}\label{smallA}
		\begin{split}
			A_1 &= \zeta_1^{\frac{3}{2}},
			\\
			A_2 &= \frac{1}{2} \dot{\zeta}_1 \zeta_1^2,
			\\
			A_3 &= \frac{5}{2} \zeta_1^{\frac{1}{2}} \zeta_2 + \frac{1}{6} \zeta_1^{\frac{5}{2}}\dot{\zeta}_1^2 + \frac{1}{6} \zeta_1^{\frac{7}{2}} \ddot{\zeta}_1.
		\end{split}
	\end{align}
	\\
	\\
	\subsection{Outline of the algorithm leading to Theorem \ref{main}}\label{subsec: Outline of the algorithm}
	\red{We now pause to outline the remainder of the paper and indicate how it leads to the proof of Theorem \ref{main}.
	\begin{enumerate}
		\item \red{In Section \ref{sec: I in terms of zeta} below, we begin by showing that the action integral $\mathcal{I}(t)$ from Definition \ref{action}} can be written as an integral of rational functions of the Taylor coefficients of the interpolating Hamiltonian $\zeta$. \red{The $m^{\text{th}}$ Marvizi-Melrose invariant $\I_m = \frac{d^{m-1}}{dt^{m-1}} \I(t) |_{t = 0}$ will depend on $\zeta_1, \cdots, \zeta_m$ (see Proposition \ref{intinv}).}
		\\
		\item \red{In Section \ref{sec: Computing AM Geometrically}, we show that the coefficients} $A_{M}(s)$ from \eqref{eq: splus in terms of s and lambda} can be computed geometrically in terms of the curvature jet \red{(see Proposition \ref{AM}).}
		\\
		\item \red{In Section \ref{sec: computing AM Algebraically}, we show that the coefficients have a triangular structure:} for each $m$, we \red{can} compute $\zeta_m$ in terms of the coefficients \red{$A_{2m-1}$} together with $\zeta_1, \cdots, \zeta_{m-1}$ \red{(see Proposition \ref{ZA}).}
		\\
		\item \red{Combining Steps (2) and (3) above, we obtain a recursive procedure for computing each $\zeta_m$ in terms of the curvature and its derivatives. We then plug these coefficients into the formula for $\I_m$ in terms of $\zeta$ from Step (1). This is done in Sections \ref{sec: linear terms with maximal derivatives} and \ref{sec: quadratic terms of maximal differential degree}.}
		\\
		\item \red{In Section \ref{subsec: integration by parts}, we carry out the extremely labor-intensive task of integrating by parts the formulas from Step (4), using the algorithm for reduction of derivatives described in Section \ref{subsec: cohomological considerations and curvature polynomials}.}
		\\
		\item We finish the proof of Theorem \ref{main} in Section \ref{subsec: nonvanishing of the leading order coefficient} by showing that the leading order coefficient of $\ka_{m-1}^2$ in the integrand of $\I_m$ is nonvanishing.
	\end{enumerate}
	}

	\subsection{Computing $\I_k$ in terms of $\zeta$} \label{sec: I in terms of zeta}
	If we define \red{$\lambda(s,t)$}  implicitly so that $\zeta(s, \lambda(s,t)) = t$, then $\I(t)$ can be written in coordinates as
	\begin{align}\label{iit}
		\begin{split}
			\I(t) =& \int_0^\ell \frac{d  \zeta\red{(s,\lambda(s,t))} \circ J}{|d \zeta\red{(s,\lambda(s,t))}|^2} \left(1, \frac{\d \lambda}{\d s}(s,t)\right)^T ds
			\\
			=& -\int_0^\ell  \left(\frac{\d \zeta}{\d \lambda} \red{(s,\lambda(s,t))} \right)^{-1} ds,
		\end{split}
	\end{align}
	where the last equality follows from the relation
	\begin{align*}
		\frac{\d }{\d s} \zeta(s, \lambda(s,t)) = 0 \implies  \frac{\d \lambda}{\d s}(s,t) = - \frac{\d \zeta}{\d s}(s,\lambda(s,t)) \left(\frac{\d \zeta}{\d \lambda}(s,\lambda(s,t))\right)^{-1}.
	\end{align*}

	\red{It is easy to see that
	\begin{align*}
		\I_1 = - \int_0^\ell \zeta_1^{-1} (s) ds.
	\end{align*}
	To find the structure of higher order terms, it is convenient to introduce the following definition.
	
	\begin{def1}\label{def:zeta weight}
		Let
		\begin{align*}
			\mathcal{Q} = c \prod_{i = 1}^m \prod_{j =0}^{r} \left(\frac{d^j \zeta_i}{d s^j} \right)^{p_{ij}}
		\end{align*}
		be a monomial in the $r$-jet of $\zeta_1^{\pm 1/2}, \zeta_2, \zeta_3, \cdots, \zeta_m$ with $c \in \R$. We define its \textbf{$\zeta$-weight} to be the quantity
		\begin{align*}
			w_\zeta(\q ) = \sum_{i = 1}^m \sum_{j = 0}^r p_{i j} (2 i - 2).
		\end{align*}
		For a polynomial in the variables $\{\zeta_1^{\pm 1/2}, \zeta_2, \zeta_3, \cdots, \zeta_m\}$, we define its $\zeta$-weight to be the maximum $\zeta$-weight of all constituent monomials.
	\end{def1}
	\red{The notion of $\zeta$-weight will be important in Section \ref{Integral invariants}, especially in Corollary \ref{cor:zeta weight bounds ddeg} and its application to formulas \eqref{eq:Imzetakappa} and \eqref{eq:Imzetakappaupsilon}.} The following lemma characterizes highest $\zeta$-weight monomials in higher order Marvizi-Melrose invariants.}
	
	\begin{prop}\label{intinv}
		\red {For $m \geq 2$,} the integral invariants have the form
		\begin{align*}
			 \I_{m} = \frac{d^{m-1}}{dt^{m-1}} \I(t) \big|_{t = 0} = \int_0^\ell \Theta_{m}[\zeta](s) ds,
		\end{align*}
		where $\Theta_{m}$ is a polynomial in $\zeta_1^{\pm 1}, \zeta_2, \zeta_3, \cdots, \zeta_{m}$ of the form
		\begin{align*}
			\Theta_m =&  m! \zeta_1^{-m-1}  \zeta_{m} - (m-1)! \zeta_1^{-m-2} \sum_{j =1}^{m-2} (j + 1) (m - j)\zeta_{j+1} \zeta_{m - j}\\
				& - \sum_{\ell = 1}^{m-2}\sum_{i = 0}^{\ell - 1} \zeta_1^{-m-2}  \frac{(m - 1 - \ell +i )! \ell!}{i!} (m-\ell) (\ell+1) \zeta_{m - \ell} \zeta_{\ell + 1} + \mathcal{R}_m^\Theta[\zeta].
		\end{align*}
		Here, $\mathcal{R}_m^\Theta[\zeta]$ is a polynomial remainder term depending only on $\zeta_1^{\pm 1}, \cdots, \zeta_{m-1}$ and having the following properties:
		\begin{itemize}
			\item \red{The $\zeta$-weight of any $\mathcal{R}_m^\Theta$-monomial is at most} $2m - 2$. 
			
			\item \red{If an $\mathcal{R}_m^\Theta$-monomial has $\zeta$-weight} equal to $2m - 2$, then at least three of its $\zeta$-indices are greater than or equal to $2$.
		\end{itemize}
	\end{prop}

\begin{proof}
	The integrand of $\I(t)$ takes the form
	\begin{align*}
		\begin{split}
			\left(\frac{\d \zeta}{\d \lambda} \right)^{-1} &= \frac{1}{\sum_{i = 1}^\infty i \zeta_i (s) \lambda^{i -1}  } = \frac{{\zeta_1}^{-1}}{1 + \sum_{i = 2}^\infty i \left(\frac{\zeta_i}{\zeta_1} \right) \lambda^{i -1}}\\
			&=  \sum_{m = 0}^\infty \left(  \frac{1}{\zeta_1} \sum_{k = 0}^m (-1)^k \sum_{\substack{j_1 + \cdots + j_k = m,\\ {j_r \geq 1}}} \prod_{i = 1}^{k} (j_i +1) \wt{\zeta}_{j_i + 1} \right) \lambda^m\\
			&: = \sum_{m = 0}^\infty b_m \lambda^m,
		\end{split}
	\end{align*}
	where we have used the notation $\wt{\zeta}_j = \zeta_j / \zeta_1$. Let us denote the function above by
	\begin{align}\label{eq:bm}
		f(\lambda) \red{\sim} \sum_{m = 0}^\infty b_m \lambda^m \quad \red{\text{as } \lambda \to 0, \qquad b_m = \frac{1}{m!} \frac{d^m f}{d \lambda^m} \Big|_{\lambda = 0},}
	\end{align}
	with the understanding that $\lambda$ depends implicitly on \red{both $s$ and} $t$. From the identity $\zeta(s, \lambda(s,t)) = t$, it follows that
	\begin{align}
		\begin{split}
			1 &= \frac{\d}{\d t} \zeta(s, \lambda(s,t)) = \frac{\d \zeta}{\d \lambda} \frac{\d \lambda}{\d t},\\
			&\implies \frac{\d \lambda}{\d t} = \left(\frac{\d \zeta}{\d \lambda}\right)^{-1} = f(\lambda).
		\end{split}
	\end{align}
	We then have functions $f$ and $\lambda$ such that $\frac{\d}{\d t} f(\lambda(s,t)) = f'(\lambda(s,t)) f(\lambda(s,t))$. To compute the \red{Taylor} coefficients \red{of $f(\lambda(s,t))$ with respect to $t$}, observe the \red{pattern}:
	\begin{align*}
		\frac{\d}{\d t} f(\lambda) &= f'(\lambda) \frac{\d \lambda}{\d t} = f'(\lambda) f(\lambda),\\
		\frac{\d^2}{\d t^2} f(\lambda) &= f''(\lambda)f^2(\lambda) + f'(\lambda)^2 f(\lambda),\\
	\end{align*}
	and for higher $N$, we have
	\begin{align*}
		\frac{\d^N}{\d t^N} f(\lambda(s,t)) = \left(f (\lambda)\frac{d}{d \lambda} \right)^N f(\lambda)\big|_{\lambda = \lambda(s,t)}.
	\end{align*}
	Powers of differential operators have been extensively studied in the combinatorics literature. We use the following formula due to Comtet:
	\begin{lemm}[\cite{Comtet}, \red{pg. 166}]\label{comtet}
		Let $f: \R \to \R$ be a smooth function. We have
		\begin{align*}
			\left(f(\lambda) \frac{d}{d \lambda}\right)^{\red{N}} = \sum_{\ell = 1}^N A_{N, \ell}[f] \frac{d^\ell}{d \lambda^\ell},
		\end{align*}
		where the coefficients are given by
		\begin{align*}
			A_{N,\ell}[f] = \sum_{\textbf{k} \in P_{N,\ell}} \frac{f_0}{\ell!} \prod_{j = 1}^{N-1} \left(j + 1 - k_1 - \cdots - k_j\right) \frac{f_{k_j}}{k_j!}, \qquad f_{k_j} :=  \left(\frac{d}{d \lambda}\right)^{k_j} f(\lambda),
		\end{align*}
		and $P_{N,\ell}$ is the set
		\begin{align*}
			P_{N,\ell} = \left\{ \textbf{k} \in \Z_{\geq 0}^{N-1}: \sum_{j = 1}^{N-1} k_j = N - \ell, \quad \sum_{j = 1}^p k_j \leq p, \hspace{0.1in} \text{for all} \hspace{0.1in} 1 \leq p \leq N-1 \right\}.
		\end{align*}
	\end{lemm}
	\red{Note that $f_k$ are functions of $\lambda$ rather than constants; in particular, $f_0(\lambda) \big|_{\lambda = 0} = b_0$ in the notation \eqref{eq:bm}.} For each $\ell$, we have sums of products of the $b_m$; those in the coefficients $A_{N,\ell}$ have indices summing to $N - \ell$, while $\frac{\d^\ell}{\d \lambda^\ell} f(\lambda) |_{\lambda = 0}$ gives a multiple of $b_\ell$. Hence, all terms have indices summing to $N$. The terms having a single nonzero index $k_j$ for which $|\textbf{k}| = N -\ell$ are $\textbf{k} = (N-\ell) e_{N-\ell}, \cdots, (N-\ell) e_{N-1}$, where $e_i \in \Z^{N-1}$ is the standard basis vector. When $\ell < N$, their contribution to $A_{N,\ell}$ is then
	\begin{align*}
		\sum_{i = 0}^{\ell - 1} & \frac{f_0}{\ell!} \left(\prod_{j = 1}^{N - \ell + i - 1} (j+1 )\frac{f_0}{0!} \right) \left(\prod_{j = N - \ell + i}^{N - 1} (j + 1 - (N- \ell))  \frac{f_{k_j}}{k_j!}\right) \\
		=& \sum_{i = 0}^{\ell - 1} \frac{f_0^{N-1}}{\ell!} (N - \ell +i )! \frac{\ell!}{i!} \frac{f_{N- \ell}}{(N-\ell)!}\\
		=&  \sum_{i = 0}^{\ell - 1} \frac{(N - \ell +i )!}{ i ! (N- \ell)!} f_0^{N-1} f_{N- \ell}.
	\end{align*}

	In particular,	all terms in
	\begin{align*}
		\frac{d^N}{d t^N} \left(\frac{\d \zeta}{\d \lambda}\right)^{-1} &= \left(\sum_{m = 0}^\infty b_m \lambda^m \frac{d}{d \lambda}\right)^N \sum_{m = 0}^\infty b_m \lambda^m
		\\
		&= \sum_{\ell = 1}^N A_{N, \ell}\left[\sum_{m = 0}^\infty b_m \lambda^m\right] \frac{d^\ell}{d \lambda^\ell} \sum_{m = 0}^\infty b_m \lambda^m,
	\end{align*}
	which have no more than two \red{factors with} nonzero $b_j$-indices when evaluated at $\lambda = 0$, equivalently $t = 0$, are of the form
	\begin{align*}
		N! b_0^N b_N  + \sum_{\ell = 1}^{N-1} \sum_{i = 0}^{\ell - 1} b_0^{N-1}  \frac{(N - \ell +i )! \ell!}{i!} b_{N - \ell}b_\ell.
	\end{align*}
	It is clear that $b_0 = \zeta_1^{-1}$ and we take the terms with maximal $\zeta$-indices in $b_{N-\ell}, b_\ell$:
	\begin{align*}
		b_N &=  - \zeta_1^{-2} (N+1) \zeta_{N + 1} + \zeta_1^{-3} \sum_{j =1}^{N-1} (j + 1) (N - j + 1)\zeta_{j+1} \zeta_{N - j +1} + \hot,
		\\
		b_{N-\ell} b_\ell &=  \left(- \zeta_1^{-2} (N-\ell +1) \zeta_{N - \ell + 1}\right) \left(- \zeta_1^{-2} (\ell+1) \zeta_{\ell + 1}\right) + \hot,
	\end{align*}
	\red{where $\hot$ denotes higher order terms, which satisfy the conditions required for $\mathcal{R}_m^\Theta$ stated in the Lemma with $N = m-1$.\footnote{\red{It is not immediately clear that these terms are \textit{higher order} in the sense of Definition \ref{def:linearquadratic}, but after solving for $\zeta_m$ in terms of curvature in the next section, we will see that they are indeed $(2m-2,3)$-harmless in the sense of Definition \ref{def:harmless}; see Theorem \ref{Linear coeff}.}}} Combining, we have
	\begin{align*}
		\frac{d^N}{d t^N} \left(\frac{\d \zeta}{\d \lambda}\right)^{-1} \bigg|_{ t = 0}
		&
		=
		- (N+1)! \zeta_1^{-N-2}  \zeta_{N + 1}
		\\
		&
		+ N! \zeta_1^{-N-3} \sum_{j =1}^{N-1} (j + 1) (N - j + 1)\zeta_{j+1} \zeta_{N - j + 1}
		\\
		&
		+ \sum_{\ell = 1}^{N-1}\sum_{i = 0}^{\ell - 1} \zeta_1^{-N-3}  \frac{(N - \ell +i )! \ell!}{i!} (N-\ell +1) (\ell+1) \zeta_{N - \ell + 1} \zeta_{\ell + 1}
		\\
		&+ \hot
	\end{align*}
	
	Putting $N = m-1$ and recalling that
	\begin{align*}
		\I(t) = - \int_0^\ell \left(\frac{\d \zeta}{\d \lambda}(s,t) \right)^{-1}ds
	\end{align*}
	completes the proof of the lemma.
\end{proof}

	\subsection{Computing $A_M$ geometrically}\label{sec: Computing AM Geometrically}

	We already know the formulas for $\zeta_1, \zeta_2$ as well as $A_1, A_2, A_3$ and we will see below that the terms $A_{M}$ are always given by algebraic functions in the curvature jet. We can then use this structure recursively to find a general expression for $\zeta_m$, which has a similar form. In keeping track of maximal derivatives on the curvature in $\zeta_m$, it will also be important to keep track of the maximal derivatives in $A_{2m - 1}$.
	\\
	\\
	In Section \ref{curvecoord}, we fixed a gauge corresponding to tangency at the origin and chose the coordinate $\theta$ which is a primitive of the curvature. \red{We will use $\theta_p = \ka_{p-1}$ to denote the $p$-th derivative of $\theta$ with respect to $s$.} By rotation and translation invariance, it suffices to compute the local expansion \eqref{sp} at the origin $s = 0$. The goal is to equate \red{powers of $\lambda$ on both sides of the equation}
	\begin{align}\label{ypxp}
		\tan(0 + \phi) = \frac{\sqrt{\lambda(2 - \lambda)}}{1 - \lambda} = \frac{\int_0^{s^+} \sin \theta(t) dt}{\int_0^{s^+} \cos\theta(t) dt } \sim \sum_{p = 1}^\infty c_p[\theta] (s^+)^p,
	\end{align}
	with $s^+$ expressed in terms of $\lambda$ as in \eqref{sp}. The coefficients $c_p[\theta]$ are \red{polynomials in the derivatives of $\theta$ with respect to $s$, evaluated at} $s = 0$, where $\theta(0) = 0$. \red{Using $\theta_1(s) = \kappa(s)$,} these relations will allow us to recursively find $\zeta_i(s)$, which can then be plugged into \eqref{iit} and integrated by parts into the form appearing in Theorem \ref{main}.

	\begin{lemm}\label{cp}
		The \red{coefficients} $c_p[\theta]$ \red{depend only on $\theta_1, \cdots, \theta_{p}$ (equivalently, $\ka, \ka_1, \cdots, \ka_{p-1}$) at $s = 0$  and} have the form $c_p[\theta] = \wt{c}_p[\theta_1]$ where $\theta_1 = \kappa$ and
		\begin{align*}
			\red{c_p[\theta]} = \wt{c}_p[\kappa] = \frac{\kappa_{p-1} }{(p+1)!}+ \RR_{p}^c[\kappa].
		\end{align*}
		The remainder $\RR_{p}^c$ has differential degree at most $p - 3$ as a polynomial in the derivatives of $\kappa$.
	\end{lemm}
	
	\begin{proof}
		As $s= s^+ = 0$ corresponds to $\theta=0$, we can expand the quotient of integrals in the expression for $\tan \phi$, giving an asymptotic expansion in $s^+$ of the form
		\begin{align}\label{tanphi}
			\tan \phi \sim \sum_{p = 0}^\infty   \sum_{\substack{\ell + j = p + 1\\ \ell \geq 0,  j \geq 1}}  \sum_{i = 0}^\ell \sum_{\substack{ k_1 + \cdots + k_i = \ell,\\ {k_u \geq 1}}} (-1)^i S_j C_{k_1 + 1} \cdots C_{k_i + 1} (s^+)^p,
		\end{align}
		\red{for some coefficients $S_j, C_k \in \R$ depending on the jet of $\theta$ at $s = 0$.} This follows from writing
		\begin{align*}
			\int_0^{s^+} \sin \theta(t) dt = \sum_{j = 1}^\infty S_j (s^+)^j, \qquad \int_0^{s^+} \cos \theta(t) dt = \sum_{k = 1}^\infty C_k (s^+)^k, 
		\end{align*}
		and performing the usual trick
		\begin{align}
			\label{eq:sincos quotient}
			\frac{	\int_0^{s^+} \sin \theta(t) dt }{	\int_0^{s^+} \cos \theta(t) dt } \sim \frac{1}{C_1 (s^+)} \sum_{j = 1}^\infty S_j (s^+)^j \left(\frac{1}{1 + \left(\sum_{k = 1}^\infty \wt{C}_{k+1} (s^+)^{k}\right)} \right),
		\end{align}
		where \red{$\wt{C}_{k}  = C_1^{-1} C_k$.} The \red{summand in parentheses in the right-hand side of} \eqref{eq:sincos quotient} can be expanded in a geometric series. \red{It is clear that $C_1 = 1$, so $C_k = \wt{C}_k$ for all $k$.} To find the coefficients $S_j, C_k$, we first Taylor expand $\sin \theta$, $\cos \theta$ in terms of $\theta$ and then expand $\theta$ in terms of $s^+$ at $s^+ = 0$ with coefficients depending on the curvature jet. For the cosine term, we have
		\begin{align*}
			C_k &= \frac{1}{k!} \left(\frac{d}{d s^+}\right)^k \int_0^{s^+} \sum_{q = 0}^\infty \frac{(-1)^q}{(2q)!} \theta^{2q}(t) dt \bigg|_{s^+ = 0}\\
			&= \frac{1}{k!} \sum_{q = 0}^{\lfloor(k-1)/2\rfloor} \sum_{\substack{\ell_1 + \cdots + \ell_{2q} = k - 1 \\ \ell_u \geq 0} } \frac{(-1)^q}{(2q)!} \binom{k-1}{\ell_1, \cdots, \ell_{2q}} \prod_{i = 1}^{2q} \theta_{\ell_i},
		\end{align*}
		and for the sine term,
		\begin{align*}
			S_j = \frac{1}{j!} \sum_{r = 0}^{\lfloor j/2 - 1\rfloor} \sum_{\substack{\ell_1 + \cdots + \ell_{2r + 1} = j - 1 \\ \ell_u \geq 0} } \frac{(-1)^r}{(2r + 1)!} \binom{j-1}{\ell_1, \cdots, \ell_{2r + 1}} \prod_{i = 1}^{2r + 1} \theta_{\ell_i}.
		\end{align*}
		Notice that only the terms where each $\ell_i > 0$ contribute since $\theta(0) = 0$. Moreover, only terms with $0 \leq q \leq \lfloor \frac{k-1}{2}\rfloor $ and $0 \leq r \leq \lfloor \frac{j}{2} -1 \rfloor$ fulfill the criteria of the inner sums when $\ell_i \geq 1$. Hence, the maximal curvature derivative comes from the terms $q = 1$ (for $C_k)$ and $r = 0$ (for $S_j$), yielding
		\begin{align*}
			C_1 &= 1,\\
			C_2 &= 0,\\
			C_3 &= - \frac{ \ka^2}{6},\\
			C_k &=  \frac{(1 - k) }{k!}\theta_{k - 2} \theta_1 + \lot, \qquad k \geq 4\\
			S_1 & = 0,\\
			S_j &=  \frac{1}{j!} \theta_{j - 1} + \lot,  \qquad j \geq 2,
		\end{align*}
		\red{where by $\lot$, we mean lower order terms depending on fewer derivatives of $\theta$, or equivalently, of $\ka$.} To address the differential degree in $\kappa$, observe that since $\theta_1 = \kappa$, the terms involving undifferentiated factors of $\ka$ don't actually contribute to the differential degree of a polynomial in the curvature jet. In particular, for $q  \geq 2$ in the case of $C_k$ and $r \geq 1$ in the case of $S_j$, since only the terms $\theta_{\ell_i}$ with $\ell_i \geq 1$ are nonzero, the resulting polynomials have lower differential degree in the derivatives of $\kappa$ rather than of $\theta$. In either case, the differential degree of $\theta_{\ell_i}$ in $\kappa$ is $\ell_i - 1$, which implies that the $q^{\text{th}}$ term in $C_k$ has differential degree $k - 1 - 2q$ while that of the $r^{\text{th}}$ term in $S_j$ is $j - 1 - 2r - 1$. More concretely, we have
		\begin{align}\label{RMCS}
			\begin{split}
				C_k &= \frac{(1 -  k)}{k!} \kappa \kappa_{k - 3} + \RR_{k}^C,\\
				S_j &= \frac{\kappa_{j - 2}}{j!} + \RR_{j}^S,
			\end{split}
		\end{align}
		with $\RR_{k}^C$ a polynomial in $\kappa, \cdots, \kappa_{k-5}$ of differential degree at most $k - 5$ and $\RR_j^S$ a polynomial in $\kappa, \cdots, \kappa_{j-4}$ of differential degree at most $j - 4$. 
		\\
		\\
		To isolate the maximal derivatives appearing in $c_p[\theta]$, note that each term in the sum \eqref{tanphi} has differential degree $j - 2 + k_1 - 2 + \cdots + k_i -2 = j - 2 + \ell  - 2i = p - (2i + 1)$ as a polynomial in the jet of $\kappa$. Hence we should choose $i$ minimal to obtain terms with maximal differential degree. Putting \eqref{RMCS} into the expansion of $\tan \phi$, we have
		\begin{align*}
			c_p[\theta] & = \sum_{\substack{\ell + j = p + 1\\ \ell \geq 0,  j \geq 1}}  \sum_{i = 0}^\ell \sum_{\substack{ k_1 + \cdots + k_i = \ell,\\ {k_u \geq 1}}} (-1)^i S_j C_{k_1 + 1} \cdots C_{k_i + 1}\\
			&= \frac{\kappa_{p-1}}{(p+1)!} + \RR_p^c,
		\end{align*}
		where we have used only the term $i = 0, \ell = 0, j = p+1$. The remainder has differential degree $p-3$, completing the lemma.
	\end{proof}
	
	On the other hand, we can express $\tan \phi$ as
	\begin{align*}
		\frac{\sqrt{\lambda(2 - \lambda)}}{1 - \lambda} = \sum_{N = 1}^\infty d_N \lambda^{\frac{N}{2}},
	\end{align*}
	where the coefficients $d_N$ are purely combinatorial. Expanding \red{$s^+ \sim 0 + \sum_{M= 1}^\infty A_M(s) \lambda^{\frac{M}{2}}$ from \eqref{eq: splus in terms of s and lambda} in the right-hand side of} \eqref{ypxp} and matching \red{the coefficients of $\lambda^{\frac{N}{2}}$} yields the equation
	\begin{align}\label{eq:dMcpA}
		d_M = \sum_{p = 1}^M c_p[\theta] \sum_{\substack{j_1 + \cdots + j_p = M,\\ {j_k \geq 1}}} \prod_{k = 1}^p A_{j_k}(s).
	\end{align}
	Let us compute the first two terms explicitly to corroborate the formulas which have \red{already} been computed \red{in} \cite{MM} \red{and} \cite{Sorr15}. Using that $S_1 = C_2 = 0$,
	\begin{align*}
		S_1 &= 0, \qquad S_2 = \kappa / 2,\\
		S_3 &= \kappa_1 / 6, \qquad S_4 = \frac{\ka_2 - \ka^3}{24},\\
		C_1 &= 1, \qquad C_2 = 0,\\
		C_3 &= - \frac{\ka^2}{6},
	\end{align*}
	\red{and} we get
	\begin{align*}
		\tan \phi &= \frac{\sum_{j = 2}^\infty S_j (s^+)^j}{\sum_{k = 1}^\infty C_k(s^+)^k}\\
		&= \frac{S_2}{C_1^2} (s^+) + \frac{S_3 C_1 - S_2 C_2}{C_1^3} (s^+)^2 + \frac{- C_3 S_2 + S_4}{C_1^4} (s^+)^3 + O((s^+)^4)\\
		& = \frac{\kappa}{2} (s^+) + \frac{\kappa_1}{6} (s^+)^2 + \left(\frac{\ka^3}{12} + \frac{\ka_2 - \ka^3}{24} \right) (s^+)^3+ O((s^+)^4).
	\end{align*}
	Hence,
	\begin{align}\label{cps}
		c_1 = \frac{\kappa}{2}, \qquad c_2 = \frac{\kappa_1}{6}, \qquad
		c_3 = \frac{\ka_2 + \ka^3}{24},
	\end{align}
	which is in line with Proposition \ref{cp}. One also checks that $d_1 = \sqrt{2}, d_2 = 0, d_3 = \frac{3}{2\sqrt{2}}$, which gives
	\begin{align}\label{Aps}
		\begin{split}
		\sqrt{2} &= A_1 c_1 \implies A_1 = 2^\frac{3}{2} \ka^{-1}, \\
		0 &= A_2 c_1 + c_2 A_1^2 \implies A_2 = - \frac{8}{3} \kappa^{-3} \kappa_1,\\
		\frac{3}{2 \sqrt{2}} & = c_1 A_3 + c_2 (A_1 A_2 + A_2 A_1) + c_3 A_1^3,  \\
		& \implies A_3 = \frac{3}{\sqrt{2}} \ka^{-1}  + \frac{32 \sqrt{2}}{9} \ka^{-5} \ka_1^2 - \frac{4 \sqrt{2}}{3}\ka_2 \ka^{-4} - \frac{4\sqrt{2}}{3} \ka^{-1}.
	\end{split}
	\end{align}
	
	One can easily check that these formulas are in agreement with \eqref{smallzeta} and \eqref{smallA}. In general, we can recover $A_M$ in terms of curvature from the coefficients $c_p$. The following proposition characterizes the algebraic structure of $A_M$ and generalizes the computations above.

	\begin{prop}\label{AM}
		For $M \geq 3$, the coefficients $A_M$ are \red{polynomials in the curvature jet which have the form}
		\begin{align*}
			\red{A_M} =& \red{A_M^{\text{lin}} + A_M^{(2)} + \RR_M^A},
		\end{align*}
		where
		\begin{align*}
			A_M^{\text{lin}} =& - 2^\frac{3M + 2}{2} \frac{ \kappa_{M - 1}}{(M+1)!} \kappa^{-M-1},
			\\
			A_M^{(2)} = &  2^\frac{3M + 4}{2} \ka^{-M-2}  \sum_{p = 2}^{M-1} p \frac{ \kappa_{p - 1}\ka_{M - p} }{(p+1)! (M - p + 2)!},
		\end{align*}
		and $\RR_M^A$ is an \red{$(M-1,3)$-harmless remainder in the sense of Definition \ref{def:harmless}, i.e.,}
		\begin{itemize}
			\item The differential degree of $\RR_M^A$ \red{is at most} $M-1$.
			\item \red{All monomials in $\RR_M^A$ which have} differential degree equal to $M- 1$ contain at least three factors of $\kappa_j$ \red{with} $j \geq 1$.
		\end{itemize}
	\end{prop}

	\begin{proof}

		\red{From \eqref{Aps}, we see that the proposition is true when $M = 3$, so we proceed by induction. Assume that $A_N$ has the form stated in the proposition for $N = 3,4, \cdots M-1$. We begin by decomposing the right hand side of \eqref{eq:dMcpA}. Separating out the terms with maximal indices of $A_M$ and $c_M$, we have
			
			\begin{align*}
				d_M &= c_M A_1^M + c_1 A_M + \sum_{p = 2}^{M-1} c_p[\theta] \sum_{\substack{j_1 + \cdots + j_p = M,\\ {j_k \geq 1}}} \prod_{k = 1}^p A_{j_k},
			\end{align*}
			
			Solving for $A_M$ gives
			
			\begin{align}\label{eq:AMfirstfewcp}
				A_M = c_1^{-1} d_M - c_1^{-1} c_M A_1^M - c_1^{-1}\sum_{p=2}^{M - 1} c_p \sum_{\substack{j_1 + \cdots + j_p = M,\\ {j_k \geq 1}}} \prod_{k = 1}^p A_{j_k}.
			\end{align}
			
			Now, the same simplification can be applied to each $A_{j_k}$ in the product on the right hand side of \eqref{eq:AMfirstfewcp}; expanding $A_{j_k}$ in terms of \eqref{eq:dMcpA}, we separate out terms with exactly one index $j_k\geq 2$ from terms with at least two such indices. There are $p$ terms with a single index $j_k$ greater than or equal to two and the fact that all other $j_\ell = 1$ implies that $j_k = M - p + 1$. We have
			
			\begin{align}\label{eq:AMmainterms}
				A_M = - c_1^{-1} c_M A_1^M - c_1^{-1} \sum_{p=2}^{M-1} c_p p A_1^{p-1} A_{M - p + 1} + \RR_{M,1}^A,
			\end{align}
			where
			\begin{align*}
				\RR_{M, 1}^A = c_1^{-1} d_M -  c_1^{-1} \sum_{p=2}^{M - 1} c_p
				\sum_{
				\substack{j_{1} + \cdots + j_p = M,
				\\
				{\exists k_1 \neq k_2 : j_{k_1}, j_{k_2} \geq 2}}} \prod_{k = 1}^p A_{j_k},
			\end{align*}
			with the product containing only terms having at least two indices $j_{k_1}, j_{k_2} \geq 2$. Recall from Lemma \ref{cp} that $c_p = \frac{\kappa_{p-1}}{(p+1)!} + \RR_p^c$, where $\RR_p^c$ has differential degree at most $p-3$ as a polynomial in the jet of $\kappa$. In view of our induction hypothesis on $A_N$ for $3 \leq N \leq M-1$ together with the explicit formulas for $A_1$ and $A_2$ in \eqref{Aps}, we see that $\RR_{M,1}^A$ contains terms of differential degree at most $M - 1$ and all terms with differential degree equal to $M-1$ contain at least three $\ka$-derivatives having order at least one, coming from $A_{j_{k_1}}, A_{j_{k_2}}$ and $c_p$ with $2 \leq j_{k_1}, j_{k_2}, p \leq M-1$. We now plug our formulas for the leading order parts of $c_p$ (from Lemma \ref{cp}) and $A_N$ (from our induction hypothesis) into the main two terms in \eqref{eq:AMmainterms}.}
			\\
			\\
		\red{Using $c_1 = \ka/2$, $A_1 = 2^{3/2} \ka^{-1}$ and $c_M = \ka_{M-1}/(M+1)! + \RR_M^c$, we see that the linear term (in the sense of Definition \ref{def:linearquadratic}) is
			\begin{align*}
				-c_1^{-1} c_M A_1^M = -2^{(3M+2)/2} \frac{\ka_{M-1}}{(M+1)!} \ka^{-M - 1}.
			\end{align*}
			For the second term, we use the induction hypothesis, together with the explicit formula for $A_2$, to write
			\begin{align}\label{eq:inductiveAN}
				A_N^{\text{lin}} = -c_1^{-1} c_N A_1^N = -2^{(3N+2)/2} \frac{\ka_{N-1}}{(N+1)!} \ka^{-N - 1}, \qquad 2 \leq N \leq M-1.
			\end{align}
			It is clear that the second and higher order terms in \eqref{eq:inductiveAN} contribute only to a remainder with the properties stated in the proposition. Indeed, every such term contains at least two factors $\ka_j$ with $j \geq 1$. Multiplication by $c_p$, whose leading term contains $\ka_{p-1}$ therefore produces either a term of differential degree strictly less than $M-1$ or a term of differential degree equal to $M-1$ but containing at least three differentiated curvature factors. We then have
			\begin{align*}
				-c_1^{-1} \sum_{p= 2}^{M-1} p c_p A_1^{p-1} A_{M-p + 1}^{\text{lin}} = 2^{(3M+4)/2} \ka^{-M - 2} \sum_{p = 2}^{M-1} p \frac{\ka_{p-1} \ka_{M-p}}{(p+1)! (M-p+2)!},
			\end{align*}
			which provides the structure needed for $A_M^{(2)}$.
		}
	\end{proof}

	\red{The structure of the coefficients $A_M$ will be important when integrated in later sections. In particular, the remainder $\RR_M^A$ is amenable to the ``differentiation by parts'' algorithm described in Section \ref{subsec: cohomological considerations and curvature polynomials} (Corollary \ref{cor:harmlessremainders} in particular).}

\subsection{Computing $A_M$ algebraically}\label{sec: computing AM Algebraically}
Our next goal is to determine the algebraic relationship between the coefficients $A_M$ and $\zeta_i$. To do this, we will analyze the structure of terms appearing in \eqref{sp}. \red{For use in the second order analysis in Section \ref{Integral invariants}, we introduce the following decomposition of $L = X_\zeta^2$:}

\begin{def1}\label{L}
	\red{Let $\zeta$ be an interpolating Hamiltonian for the billiard map as in Theorem \ref{thm: interpolating hamiltonian}. We} define the second order differential operator
	\begin{align*}
		L : = X_\zeta^2 =  &
		\bigg(\frac{\d \zeta}{\d s} \frac{\d^2 \zeta}{\d s \d \lambda} \frac{\d}{\d \lambda}
		+ \left(\frac{\d \zeta}{\d s} \right)^2 \frac{\d^2}{\d \lambda^2}
		- \frac{\d \zeta}{\d s} \frac{\d^2 \zeta}{\d \lambda^2} \frac{\d}{\d s} 
		- \frac{\d \zeta}{\d s} \frac{\d \zeta}{ \d \lambda} \frac{\d^2}{\d s \d\lambda}\\*
		&
		- \frac{\d \zeta}{\d \lambda} \frac{\d^2 \zeta}{\d s^2} \frac{\d}{\d \lambda}
		- \frac{\d \zeta}{\d \lambda} \frac{\d \zeta}{\d s} \frac{\d^2}{\d s \d \lambda}
		+ \frac{\d \zeta}{\d \lambda} \frac{\d^2 \zeta}{\d s \d \lambda} \frac{\d}{\d s}
		+ \left(\frac{\d \zeta}{\d \lambda} \right)^2 \frac{\d^2 }{\d s^2}
		\bigg),
	\end{align*}
	\red{whose (non-constant) coefficients are polynomials in the jet of $\zeta(s,\lambda)$. We will} denote the individual terms above by $L_1, L_2, \cdots, L_8$. \red{In view of the expansion \eqref{eq:taylor exp of zeta}, each of the differential operators} $L_i$ \red{themselves admit expansions in powers of $\lambda$, whose} coefficients \red{are polynomials in the Taylor coefficients} \red{of $\zeta$}. We denote the corresponding operators by $L_{i,q,r}$.
\end{def1}

	For example, \red{we will write}
	\begin{align*}
		L_{1,q,r} = \dot{\zeta_q} \lambda^q r \dot{\zeta_r} \lambda^{r-1} \frac{\d}{\d \lambda},
	\end{align*}
	\red{where $\dot \zeta_i = \frac{d \zeta_i}{d s}$ for $i \geq 1$.}
	
\begin{def1}\label{Zk}
	\red{For $K \in \Z_{\geq 0}$,} denote by $\z_{K} := \frac{(-1)^K }{K!} X_\zeta^{K} s$, so that $(2k+1)! \z_{2k + 1} = L^k \z_1$ and $L \z_K = (K+1) (K+2) \z_{K + 2}$. Write $\z_{K,i}$ for the coefficient of $\lambda^{i/2}$ in $\z_K$:
	\begin{align*}
		\z_K = \sum_{i = 0}^\infty \z_{K,i} \lambda^{i/2}.
	\end{align*}
\end{def1}

\red{Since $X_\zeta$ has an expansion in only integer powers of $\lambda$, it is clear that for $K$ even, $\zeta^{K/2} X_\zeta^K s$ also has an expansion in integer powers of $\lambda$. For $K$ odd, $\zeta^{K/2} X_\zeta^K s$ has an expansion in only \textit{half} (but not whole) integer powers of $\lambda$. In both cases, we will expand $\z_K$ itself in half integer powers of $\lambda$ in order to reduce the notational burden in what follows.} From the definitions above and the fact that $\zeta^{K/2}$ commutes with $X_\zeta$, it follows that
\begin{align*}
	\sum_{M = 1}^\infty A_M \lambda^{M/2} &\sim \sum_{K = 1}^\infty \zeta^{K/2} \sum_{i = 0}^\infty \z_{K,i} \lambda^{i/2}.
\end{align*}

To isolate the effect of some $\z_K$ on a coefficient $A_M$, note that only when the parity of $M$ matches that of $K$ does $\z_K$ contribute, \red{since the exponent of $\zeta^{K/2}$ determines whether $\lambda$ appears with an integer power or an integer plus one half power}.
\\
\\
Each operator in $L^k$ corresponds to $8^k$ compositions of $k$ simpler operators in an obvious way. It will also be important to specify the order in which these are composed, \red{so we introduce the following notation.} Let $\bm{\sigma}: \Z_k \to \Z_{8}$ be \red{a $k$-letter word in the alphabet $\{1,2,3,4,5,6,7,8\}$}. We can associate to $\bm{\sigma} = (\sigma_1, \cdots, \sigma_k)$ the composite operator
	\begin{align}\label{eq:wordcompositions}
		L_{\bm{\sigma}} := L_{{\sigma}_k} \circ L_{{\sigma}_{k-1}} \cdots \circ L_{{\sigma}_1},
	\end{align}
	so that
		\begin{align*}
			L^k = \sum_{\bm{\sigma}: \Z_k \to \Z_8} L_{\bm{\sigma}}.
		\end{align*}
In order to keep track of dependence on powers of $\lambda$, we introduce the following \red{coefficient extraction operator}.

	\begin{def1}\label{notationdef}
		 \red{Let} $\mathcal{Y} \red{\sim \sum_{M = 0}^\infty \Y_M} \lambda^{\frac{M}{2}}$ be any \red{formal} asymptotic expansion in powers of $\lambda^{\frac{1}{2}}$ and define $\Lambda_M$ \red{to be the linear operator which} extracts the coefficient of $\lambda^{M/2}$:
		 \begin{align*}
		 	\Lambda_M \left[ \sum_{i = 0}^\infty \mathcal{Y}_i \lambda^{i/2} \right] = \mathcal{Y}_M.
		 \end{align*}
		 \red{We call $\Lambda_M$ the \textbf{$M$-extraction operator.}}
	\end{def1}
	
	 It follows that if $M = 2m - 1$ is odd, then setting $K = 2k + 1$ gives
	\begin{align*}
		A_M = \sum_{k = 0}^{\frac{M - 1}{2}} \Lambda_M \left[ \zeta^{\frac{2k+1}{2}} \z_{2k + 1} \right].
	\end{align*}

	We will later write $M = M_1 + M_2$ to compute the contributions of $\zeta^{K/2}$ and $\z_K$ to $A_M$ separately. Moving forward, we will almost exclusively deal with the case when $M$ is odd, for reasons to be made clear in Proposition \ref{ZA}. \red{Recall our convention that} $M = 2m - 1$ (\red{equivalently,} $\frac{M+1}{2} = m$).

	\begin{lemm}\label{maxind}
		\red{Let $K \in \N$ and consider the expansion of $\zeta^{\frac{K}{2}}$ in powers of $\lambda^{\frac{1}{2}}$. We have the following structure of $\Lambda_{M_1} [\zeta^{\frac{K}{2}}]$:}
		\begin{enumerate}[(a)]
			\item \red{If $K$ and $M_1$ are odd and} $1 \leq K < M_1$, \red{then} the $M_1$ coefficient of $\zeta^{K/2}$ is of the form
			\begin{align*}
				\Lambda_{M_1} \left[ \zeta^{\frac{K}{2}} \right] =& \frac{K}{2} \zeta_1^{\frac{K -2}{2}} \zeta_{\frac{M_1 - K +2}{2}} + \left(\frac{K^2 - 2K}{8}\right) \zeta_1^{\frac{K- 4}{2}}\\
				&\times \sum_{\substack{i_1 + i_2 = \frac{M_1 - K}{2}\\ i_\ell \geq 1}} \zeta_{i_1 + 1} \zeta_{i_2 + 1} + \upsilon_{M_1, K}[\zeta],
			\end{align*}
			where
			\begin{align*}
				\upsilon_{M_1, K}[\zeta] = \zeta_1^{\frac{K}{2}} \left( \sum_{j = 3}^{\frac{M_1 - K}{2}}  \binom{\frac{K}{2}}{j} \zeta_1^{-j} \sum_{\substack{i_1 + \cdots + i_j = \frac{M_1 - K}{2} \\ i_\ell \geq 1}} {\zeta}_{i_1 + 1} \cdots {\zeta}_{i_j + 1} \right)
			\end{align*}
			is a polynomial in $\zeta_1^{\red{\pm} \frac{1}{2}}, \zeta_2, \cdots, \zeta_{\frac{M_1-1}{2}}$ with the property that each constituent monomial contains at least three factors of $\zeta_p$, $p \geq 2$ \red{and has $\zeta$-weight at most $M_1 - K$ (cf. Definition \ref{def:zeta weight})}. In particular, the maximal index of $\zeta_i$ appearing in the $M_1$-coefficient of the $\zeta^{\Khalf}$ expansion is $i = \frac{M_1 + 1}{2}$ and appears only when $K = 1$.
			
			\item If \red{both $M_1$ and} $K = 2k$ are even with $k \geq \red{2}$, then $\Lambda_{M_1}\left[\zeta^{\frac{K}{2}}\right]$ is of the form
			\begin{align*}
				\sum_{\substack{i_1 + \cdots + i_k = \frac{M_1}{2} \\ i_\ell \geq 1}} \zeta_{i_1} \cdots \zeta_{i_k},
			\end{align*}
			\red{each term of which has $\zeta$-weight at most $M_1 - K$}. In particular, for $K \red{\geq 4}$ even, the maximal index $i$ of $\zeta_i$ appearing in the \red{$\lambda^{\frac{M_1}{2}}$-coefficient of} the expansion of $\zeta^{\frac{K}{2}}$ is at most $\frac{M_1}{2} - 1$. 
		\end{enumerate}
	\end{lemm}

	\begin{proof}
		We expand $\zeta^{\Khalf}$ via the generalized binomial theorem and see that
		\begin{align*}
			\zeta^{K/2}
			&\sim
			\zeta_1^{K/2} \lambda^{K/2}  \sum_{j =0}^\infty \binom{\frac{K}{2}}{j} \left(\sum_{i = 1}^\infty \wt{\zeta}_{i + 1} \lambda^i \right)^j\\ 
			&\sim 
			\zeta_1^{K/2} \lambda^{K/2} \left(  \sum_{i = 0}^\infty \sum_{j = 0}^i  \binom{\frac{K}{2}}{j} \sum_{\substack{i_1 + \cdots + i_j = i\\ i_\ell \geq 1}} \wt{\zeta}_{i_1 + 1} \cdots \wt{\zeta}_{i_j + 1} \lambda^i \right),
		\end{align*}
		with $\wt{\zeta_{i}} = \zeta_i / \zeta_1$ and the term $i = j = 0$ corresponding to one. The term containing the maximal $\zeta_i$ depends on the parity of $M_1$.
		\\
		\\
		The terms in the statement of the lemma for $M_1$ and $K$ odd come from the indices $i = \frac{M_1 - K}{2}$, $j = 1$, $j = 2$ and the corresponding binomial coefficients, while the remainder terms arise when $j \geq 3$. In either case, the left-hand side will have the \red{power} $\lambda^{\frac{K}{2}} \lambda^i = \lambda^\frac{M_1}{2}$. It is clear that all terms in the above sum have $\zeta$-weight at most $M_1 - K$.
		\\
		\\
		To maximize the index $i$, we take $K$ to be minimal and $j = 1$. When $M_1$ is odd, we choose $K = 1$, $i = \frac{M_1-1}{2}$, and $j = 1$,  in which case  $\zeta^{1/2}$ generates the term
		$$
		\zeta_1^{\half} \lambda^{\half} \binom{\frac{1}{2}}{1}  \wt{\zeta}_{\frac{M_1-1}{2} + 1} \lambda^{\frac{M_1-1}{2}} = \frac{1}{2} \zeta_1^{-\half} \zeta_{\frac{M_1+1}{2}} \lambda^{\Monehalf},
		$$
		together with other polynomial terms in $\zeta_1^{\half}, \zeta_2, \cdots, \zeta_{\frac{M_1 - 1}{2}}$.
		\\
		\\
		If $M_1$ is even, then \red{$\Lambda_{M_1} [\zeta^{K/2}]$ is nonzero only if} $K$ is even so that there are no fractional powers of $\lambda$. With $K = 2 k, k \in \Z_{\geq 2}$, we are just expanding an integer power of $\zeta$:
		\begin{align*}
			\zeta^{\Khalf} = \zeta^k = \sum_{i = k}^\infty \sum_{\substack{i_1 + \cdots + i_k = i\\ i_\ell \geq 1}} \zeta_{i_1} \cdots \zeta_{i_k} \lambda^i,
		\end{align*}
		and therefore, the coefficient of $\lambda^{\Monehalf}$ is just
		\begin{align*}
			\sum_{\substack{i_1 + \cdots + i_k = \frac{M_1}{2} \\ i_\ell \geq 1}} \zeta_{i_1} \cdots \zeta_{i_k}.
		\end{align*}
		\red{Since $k \geq 2$, the condition $i_\ell \geq 1$ forces the maximal index in the sum to be less than or equal to $\frac{M_1}{2} - 1$.}
	\end{proof}

\noindent We will now derive a similar structure for the terms $\zeta^{\Khalf} \z_K$ appearing in \eqref{sp}. \red{To do so, we will need a generalization of the notion of differential degree which allows for polynomials in the jet of several functions simultaneously.}

\red{
	\begin{def1}
		\label{def: multi-fold diff degree}
		Let $u = (u_1, \cdots, u_m)$ be a collection of smooth functions on $\R/\ell \Z$ and denote by
		\begin{align*}
			\mathcal{Q} = c \prod_{i = 1}^m \prod_{j =0}^{r} \left(\frac{d^j u_i}{d s^j} \right)^{p_{ij}}
		\end{align*}
		a monomial in the $r$-jet of $u_1, u_2, u_3, \cdots, u_m$ with $c \in \R$, $p_{i0} \in \Q$ and $p_{ij} \in \N$ when $j \geq 1$; contrary to our previous conventions in Definition \ref{diffdeg}, the subscript $i$ here does not indicate differentiation, but rather, a choice of one of the $m$ functions $u_i$. We define the \textbf{$\bm{u}$-differential degree} of $\q$ to be
		\begin{align*}
			\text{deg}_{\d, u} (\q) = \sum_{i = 1}^m \sum_{j = 1}^r j p_{i,j},
		\end{align*}
		which counts the total number of derivatives across all $m$ functions. For a polynomial in the jets of $\{u_1, \cdots, u_n\}$ with coefficients in $\R[u_1^{p_{i0}}, \cdots, u_m^{p_{m0}}]$, we define its $u$-differential degree to be the maximum $u$-differential degree of all constituent monomials.
	\end{def1}
	
	\begin{rema}
		When $m = 1$, this coincides with the notion of differential degree in Definition \ref{diffdeg}. We introduce this notion of multi-variable differential degree primarily to deal with polynomials in the jet of $\zeta$, in which case, by an abuse of notation, we will write $\deg_{\d, \zeta}(\q)$ to denote the $(\zeta_1, \cdots, \zeta_m)$-differential degree without necessarily specifying $m$.
	\end{rema}
}

	\begin{prop} \label{ZA}
		The data $\left\{A_1, \cdots, A_{2m - 1}\right\}$ are equivalent to $\left\{\zeta_1, \cdots, \zeta_m\right\}$ for all $m$, in the sense that both sets of coefficients are given in terms of algebraic functions of a finite number of derivatives of the other. Moreover, \red{for $m \geq 2$, $\zeta_m$ does not appear in the coefficients $A_1, \cdots, A_{2m-2}$; it} first appears in $A_{2m - 1}$ \red{and has} the form
		\begin{align*}
			A_{2m-1} = \frac{ 2m+1}{2} \zeta_1^{1/2} \zeta_m + \Upsilon_{2m - 1},
		\end{align*}
		where $\Upsilon_{2m - 1}$ is a polynomial in the $(2m - 2)$-jets of $\zeta_1^{\pm \frac{1}{2}}, \zeta_2, \cdots, \zeta_{m -1}$ \red{(i.e., containing $s$-derivatives of $\zeta_1^{\pm 1/2}, \cdots, \zeta_{m-1}$ having order \red{at most} $2m - 2$)}. \red{Furthermore, denoting by $w_\zeta(\nu)$ the $\zeta$-weight of a term $\nu$ in $\Upsilon_{2m - 1}$ and by $\deg_{\d, \zeta} (\nu)$ its multi-variable differential degree, we have
		\begin{align*}
			w_\zeta(\nu) + \deg_{\d, \zeta} (\nu) \leq 2m - 2.
		\end{align*}
		}
	\end{prop}

	\begin{proof}
		\red{Recall that $\zeta = O(\lambda)$, so the first few terms in the expansion \eqref{sp} read}
		\begin{align}
			\label{first few}
			\begin{split}
			\red{s^+  - s \sim}  &\red{\sum_{j= 1}^\infty (-1)^j \zeta^{j/2} \frac{X_\zeta^j}{j!} s}
			\\
			=
			&
			\red{- \zeta^{\frac{1}{2}} X_\zeta s + \frac{1}{2} \zeta X_\zeta^2 s - \frac{1}{3!} \zeta^{\frac{3}{2}} X_\zeta^3 s + O(\lambda^2)}
			\\
			=&\zeta^{1/2} \frac{\d \zeta}{\d \lambda} + \frac{1}{2} \left(\zeta \frac{\d \zeta}{\d \lambda} \frac{\d^2 \zeta}{\d \lambda \d s} - \zeta \frac{\d \zeta}{\d s} \frac{\d^2 \zeta}{\d \lambda^2}\right)
				\\
				&+ \frac{1}{3!}  \zeta^{3/2} \bigg( \frac{\d \zeta}{\d \lambda}  \left(\left(\frac{\d^2 \zeta}{\d\lambda \d s} \right)^2 + \frac{\d \zeta}{\d \lambda}  \frac{\d^3 \zeta}{\d \lambda \d s^2}  - \frac{\d^2 \zeta}{\d s^2} \frac{\d^2 \zeta}{\d \lambda^2} - \frac{\d \zeta}{\d s} \frac{\d^3 \zeta}{\d \lambda^2 \d s}        \right)
				\\
				&  \qquad \qquad - \frac{\d \zeta}{\d s} \left(\frac{\d \zeta}{\d \lambda} \frac{\d^3 \zeta}{\d \lambda^2 \d s} - \frac{\d \zeta}{\d s} \frac{\d^3 \zeta}{\d \lambda^3}\right)\bigg) + O(\lambda^\red{2}).
			\end{split}
		\end{align}
		\red{We now expand the full power series  in the first line of \eqref{sp} in powers of $\lambda^{1/2}$, with coefficients given in terms of $(s, \lambda)$-derivatives of $\zeta$ and $\zeta^{\half}$.} The general form of the expansion \eqref{sp} consists of sums of powers of $\lambda$, \red{which} come from products of terms of the form
		\begin{align}\label{pq}
			\frac{\d^{p+q} \zeta}{\d \lambda^p \d s^q} \sim \sum_{i = p}^\infty \frac{i!}{(i-p)!} \frac{d^q \zeta_i}{d s^q} \lambda^{i - p}.
		\end{align}
		\red{In each such term, we will look} for the maximal \red{index of} $\zeta_i$ in the coefficient of $\lambda^{M/2}$. We will assume the parity of $M$ \red{matches that} of $K$, \red{so that} \red{$\Lambda_M[\zeta^{K/2} X_\zeta^K s]$ is not trivially equal to zero.}
		\\
		\\
		\textbf{Case 1.} \red{Let $K = 1$ and suppose that $M$ is odd. We will find the $\lambda^{M/2}$-coefficient in $\zeta^{1/2} \z_1 = -\zeta^{1/2} X_\zeta s$ (which is the first term on the third} line of \eqref{first few}). We have
		\begin{align*}
			\zeta^{1/2} \frac{\d \zeta}{\d \lambda} &=  \left( \zeta_1^{1/2} \lambda^{1/2} + \frac{1}{2} \zeta_1^{-1/2} \zeta_2 \lambda^{3/2} + \cdots \right) \bigg(\zeta_1 + 2 \lambda \zeta_2 + \cdots \bigg)\\
			&=  \zeta_1^{1/2} \lambda^{1/2} \left(  \sum_{i = 0}^\infty \sum_{j = 0}^i  \binom{\frac{1}{2}}{j} \sum_{\substack{i_1 + \cdots + i_j = i\\ i_\ell \geq 1}} \wt{\zeta}_{i_1 + 1} \cdots \wt{\zeta}_{i_j + 1} \lambda^i \right) \left(\sum_{\ell = 1}^\infty \ell \zeta_\ell \lambda^{\ell-1}\right),
		\end{align*}
		which contributes to $\Lambda_M[- \zeta^{1/2} X_\zeta s]$ the maximal terms
		\begin{align*}
			\zeta_1^{1/2} \left(\frac{M + 1}{2}\right) \zeta_{\frac{M + 1}{2}} , \qquad i = j = 0, \ell = \frac{M+1}{2},
		\end{align*}
		and
		\begin{align*}
			\frac{1}{2} \zeta_1^{-1/2} \zeta_{\frac{M-1}{2} + 1} \zeta_1, \qquad i = \frac{M-1}{2}, j = \ell = 1.
		\end{align*}
		The maximal indices come from minimizing the power of $\lambda$ in one of the sums so that the other can be taken to have maximal power and hence maximal index. We used Lemma \ref{maxind} for the factor $\zeta^{1/2}$. These combine to give a contribution of
		\begin{align}\label{j0}
			\frac{M + 2}{2} \zeta_1^{1/2} \zeta_{\frac{M + 1}{2}},
		\end{align}
		which for $M = 3$, corroborates the principal term in \eqref{smallA}.
	\\
	\\
	\textbf{Case 2.} We now claim that for $K > 1$, the coefficient of $\lambda^{M/2}$ in the expansion of $\zeta^{K/2} X_\zeta^K s$ contains only terms depending on $\zeta_1, \cdots, $ $\zeta_{\frac{M - 1}{2}}$ or $\zeta_1, \cdots, \zeta_{\frac{M}{2}}$ depending on the parity of $M$, together with their $s$ derivatives. \red{We first expand $X_\zeta$ in powers of $\lambda$:
		\begin{align*}
			X_\zeta \sim \sum_{i = 1}^\infty \lambda^i \dot{\zeta_i} \frac{\d}{\d \lambda} - i \lambda^{i-1} \zeta_i \frac{\d}{\d s}.
		\end{align*}
		Define the $\lambda$-\textit{order} $\mu$ of a formal expansion $\Y \sim \sum_{j=0}^\infty \Y_j\lambda^{j/2}$ to be the minimal $j_0$ for which $\Y_{j_0} \neq 0$. Note that the action of $X_\zeta$ on such an expansion does not decrease its $\lambda$-order:
		\begin{align*}
			X_\zeta \sum_{j = 0}^\infty \Y_j \lambda^{j/2}
			=&
			\sum_{i = 1}^\infty \sum_{j = 0}^\infty \left( \frac{j}{2}\lambda^{i} \dot{\zeta_i} \Y_j \lambda^{j/2 -1} - i \lambda^{i - 1} \zeta_i \dot{\Y_j} \lambda^{j/2} \right)
			\\
			=&
			\sum_{i = 1}^\infty \lambda^{i-1} \left( \sum_{j = 0}^\infty \frac{j}{2} \dot{\zeta_i} \Y_j \lambda^{j/2} - i \zeta_i \dot{\Y_j} \lambda^{j/2} \right).
		\end{align*}
	}
	\red{Each $\zeta_i$ in $X_\zeta$ acts on a nonzero coefficient $\Y_j$ through the linear differential operator
		\begin{align*}
			\Xi_{\zeta, i} : = \lambda^{i} \dot{\zeta_i} \frac{\d}{\d \lambda} - i \lambda^{i-1} \zeta_i \frac{\d}{\d s}.
		\end{align*}
		In particular, since differentiation with respect to $\lambda$ lowers the $\lambda$-order by at most two, we see that for any $i \geq 1$,
		\begin{align*}
			\mu \left(\Xi_{\zeta, i} \Y \right) \geq \mu(\Y) + 2i - 2.
		\end{align*}
		Similarly, Lemma \ref{maxind} (with $M = M_1$) shows that for $K \geq 1$, $\zeta^{K/2}$ has an expansion of the form}	
	\red{
		\begin{align*}
			\zeta^{K/2} \sim \sum_{M = K}^\infty \WW_{K, M} \lambda^{M/2},
		\end{align*}
		where the term $\zeta_i$ can appear in $\WW_{K, M}$ only if
		\begin{align*}
			\frac{M - K + 2}{2} \geq i.
		\end{align*}
		This implies that multiplication by any term in the expansion of $\zeta^{K/2}$ which contains $\zeta_i$ also increases the $\lambda$-order $2i + K -  2$.
		\\
		\\
		Recall the notation $\z_K = (-1)^K \frac{X_\zeta^K}{K!}$ from Definition \ref{Zk}. If $\zeta_m$ appears in $\zeta^{K/2} \z_K$ through the factor $\zeta^{K/2}$, then it increases the $\lambda$-order of s by at least $2m + K - 2$. If instead (or in addition), $\zeta_m$ appears through one of the Hamiltonian vector fields $\Xi_{\zeta,m}$, then it raises the $\lambda$-order of $s$ by at least $2m - 2$. Combining the $2m-2$ increase from the differential operator with the $\lambda$-order $K$ increase from multiplication by $\zeta^{K/2}$, the total increase in the $\lambda$-order of $s$ is again no less than $2m + K  - 2$. Consequently, it follows that if $K \geq 2$, $\zeta_m$ appears only in coefficients $A_N$ where $N \geq 2m > 2m - 1$.
		\\
		\\
		In conjunction with the structure of $\zeta^{1/2} \z_1$ in Case 1 above, we see that
		\begin{align*}
			A_{2m - 1} = \frac{2m + 1}{2} \zeta_1^{1/2} \zeta_m + \Upsilon_{2m-1},
		\end{align*}
		with $\Upsilon_{2m - 1}$ depending only on $\zeta_1, \cdots, \zeta_{m-1}$ as stated in the proposition. It remains only to count the number of derivatives of $\zeta$ appearing in the remainder term $\Upsilon_{2m-1}$ together with their $\zeta$-weights (cf. Definition \ref{def:zeta weight}).
		\\
		\\
		Let $\q$ be any term in the remainder $\Upsilon_{2m-1}$. It arises from compositions of multiplication by $\zeta^{K/2}$ ($1 \leq K \leq 2m - 1$) and differentiation by a sequence of $K$ Hamiltonian vector fields $\Xi_{\zeta, i_j}$. Suppose that such a term arises from the composition
		\begin{align}
			\label{eq:Xicomposition}
			\Lambda_{M_1} [\zeta^{K/2}]  \Lambda_{M_2} [\Xi_{\zeta, i_K} \cdots \Xi_{\zeta, i_1} s],
		\end{align}
		where $M_1 + M_2 = 2m - 1$. Lemma \ref{maxind} tells us that the $\zeta$-weight of any term in $\Lambda_{M_1} [\zeta^{K/2}]$ is at most $M_1 - K$ and from the form of $\Xi_{\zeta, i_j}$, it is clear that they each contribute $\zeta$-weight $2 i_{j} - 2$. Since each $\Xi_{\zeta, i_j}$ raises the $\lambda$-order of $s$ by at least $2i_j - 2$, we have the bound
		\begin{align*}
			2 i_1 - 2 + \cdots + 2 i_K - 2 \leq M_2. 
		\end{align*}
		In total, the $\zeta$-weight of any term in the composition \eqref{eq:Xicomposition} is bounded by $M_1 + M_2 - K = 2m - 1 - K$. 
		\\
		\\
		We now find the $\zeta$-differential degree and the number of $s$-derivatives in a typical term of $\z_K$ which contributes to $\Upsilon_{2m-1}$. Note that $\z_1 = \frac{\d \zeta}{\d \lambda}$ does not contain any $s$-derivatives of the coefficients $\zeta_i$ and each application of $\Xi_{\zeta, i_j}$ increases the $\zeta$-differential degree by at most one. Hence, the total $\zeta$-differential degree of a term in $\z_K$ is at most $K-1$. In particular, since only the terms $\zeta^{1/2} \z_1, \cdots, \zeta^{(2m-1)/2} \z_{2m-1}$ contribute to $A_{2m-1}$, the $\lambda$-order of any $s$-derivative of $\zeta$-coefficients in $\Upsilon_{2m - 1}$ is no more than $2m - 2$, as claimed in the proposition. Multiplication by any term in the $\lambda$-expansion of $\zeta^{K/2}$ leaves the $\zeta$-differential degree unchanged. Hence, for any term in $\Upsilon_{2m-1}$, the sum of its $\zeta$-differential degree and its $\zeta$-weight bounded by $2m - 1 - K + K - 1 = 2m - 2$, which concludes the proof of the proposition.
	}		
\end{proof}

Proposition \ref{ZA} shows that the highest order coefficients are all generated by $-\zeta^{1/2} X_\zeta s$ and for each $m = \frac{M + 1}{2}$, we can read off $\zeta_m$ from the data $A_{2m - 1}, \zeta_1, \cdots, \zeta_{m -1}$. In \red{Proposition} \ref{AM}, we showed that $A_{2m-1}$ is a polynomial in the curvature jet which has a decomposition into maximal derivatives appearing linearly, \red{second order} submaximal derivatives \red{(in the sense of Definition \ref{def:linearquadratic})} having the same differential degree, and higher order terms which can be absorbed into the remainders in Theorem \ref{main}. Together with knowledge of $\zeta_1$, this allows us to recursively find subsequent $\zeta_m$ \red{in terms of curvature}. For example, we can read off $\zeta_1$ from $A_1$, while $A_2$ contains no \red{coefficients with $\zeta_i$ for $i \geq 2$}. We can then read off $\zeta_2$ from $A_3$ together with $\zeta_1$ and so on. \red{Since for each $m > 1$, $\zeta_m$ is absent in the coefficients $A_1, A_2, \cdots, A_{2m-2}$ and first appears in $A_{2m-1}$, we can restrict our focus to the case where $M = 2m-1$ is odd moving forward.} The map $\{\zeta_1, \cdots, \zeta_m\} \mapsto \{A_1, \cdots, A_{2m-1}\}$ is \red{invertible but} highly nonlinear; its inversion modulo lower order terms is one of the main goals of Section \ref{Integral invariants}.

	\begin{rema}		
		The computations in this section effectively deal with the structure of a one dimensional Hamiltonian Lie series and are in no way special to the convex billiards setting. They are valid whenever one has an interpolating Hamiltonian and, for example, apply equally well to symplectic, projective and outer (dual) billiards \cite{AlbersTabachnikov}, \cite{Tabachnikovprojectivebilliards}, \cite{TabachnikovDual}, \cite{BaraccoBernardiNardi}.
	\end{rema}

\section{Integral invariants}\label{Integral invariants}

Let us begin by comparing our results with those in \cite{MM}. From \eqref{Aps} and \eqref{first few}, we see that
\begin{align}
	\label{eq:zetaone}
	A_1 = \zeta_1^{3/2} = 2^{\red{3/2}} \ka^{-1} \implies \zeta_1 = 2 \ka^{-2/3}.
\end{align}
It follows that
\begin{align}
	\label{eq:Ione}
	\I_1 = - \int_0^\ell \zeta_1^{-1} ds = - \frac{1}{2} \int_0^\ell \ka^{2/3} ds.
\end{align}
For $\zeta_2$, computing algebraically gives
\begin{align*}
	A_3 =& 2 \zeta_1^{1/2} \zeta_2 + \frac{\zeta_2}{2 \zeta_1^{1/2}} \zeta_1 + \frac{1}{3!} \zeta_1^{3/2} \left( \zeta_1 \dot{\zeta_1}^2 + \zeta_1^2 \ddot{\zeta_1}\right)\\*
	=& \frac{5}{\sqrt{2}} \ka^{- 1/3} \zeta_2 + \frac{112\sqrt{2}}{27} \ka^{-5} \ka_1^2 - \frac{16 \sqrt{2}}{9} \ka^{-4} \ka_2.
\end{align*}
Equating this with formula \eqref{Aps} yields
\begin{align}
	\label{eq:zetatwo}
	\zeta_2 = \frac{1}{15} \ka^{-2/3} - \frac{32}{135} \ka^{-14/3} \ka_1^2 + \frac{8}{45} \ka^{-11/3} \ka_2,
\end{align}
which when integrated against $2 \zeta_1^{-3}$ gives
\begin{align}
	\label{eq:I2}
	\I_2 = \frac{1}{540} \int_0^\ell 9 \ka^{4/3} + 8 \ka^{-8/3} \ka_1^2 ds.
\end{align}
\red{These corroborate the formulas found in both \cite{MM} and \cite{Sorr15}, although the former has a small misprint ($-2$ instead of $-32$ in $\zeta_2$).}
\\
\\
\red{Let us now see how to combine the material in the preceding sections to find a formula for $\I_m$ in terms of the curvature. Proposition \ref{intinv} gives us
\begin{align*}
	\I_m
	= &
	\int_{\d \Omega} \Theta_m(s) ds
	\\
	=&
	\int_{\d \Omega} m! \zeta_1^{-m-1}  \zeta_{m} - (m-1)! \zeta_1^{-m-2} \sum_{j =1}^{m-2} (j + 1) (m - j)\zeta_{j+1} \zeta_{m - j} ds
	\\
	& -
	\int_{\d \Omega} \sum_{\ell = 1}^{m-2}\sum_{i = 0}^{\ell - 1} \zeta_1^{-m-2}  \frac{(m - 1 - \ell +i )! \ell!}{i!} (m-\ell) (\ell+1) \zeta_{m - \ell} \zeta_{\ell + 1} ds
	\\
	& + \hot,
\end{align*}
where $\hot$ denotes ``higher order terms'' coming from $\mathcal{R}_m^\Theta [ \zeta]$. We will see in Theorem \ref{Linear coeff} below that $\zeta_m = f_m \ka^{-2m + 1/3} \ka_{2m - 2} + \mathcal{R}_m^\zeta[\ka]$ for some constants $f_m$ and a remainder polynomial $\mathcal{R}_m^\zeta[\ka]$ \red{which is $(2m-2,2)$-harmless in the sense of Definition \ref{def:harmless}, from which it follows that the remainder terms above are also $(2m-2,3)$-harmless. Throughout this section, we will consistently use Corollary \ref{cor:harmlessremainders} to discard $(2m-2,3)$-harmless terms into the remainder of Theorem \ref{main}.}
\\
\\
Now, plugging in our formula for $\zeta_m$ in terms of $A_{2m-1}$ and $\Upsilon_{2m-1}$ from Proposition \ref{ZA} and also $\zeta_m$ in terms of $\ka_{2m -2}$ from Theorem \ref{Linear coeff}, we see that

\begin{align}\label{eq:Imzetakappa}
	\begin{split}
		\I_m =&
		\int_{\d \Omega} m! (2 \ka^{-2/3})^{-m-1} \times  \frac{2}{2m + 1} \zeta_1^{-\frac{1}{2}} \left(A_{2m- 1} - \Upsilon_{2m - 1} \right) ds
		\\
		&
		- \sum_{j =1}^{m-2} \int_{\d \Omega} (m-1)! (2 \ka^{-2/3})^{-m-2} (j + 1) (m - j) \times
		\\
		&
		\qquad \qquad  \left(\ka^{-2j - 2 + 1/3} f_{j+1}  \ka_{2j} + \RR_{j+1}^\zeta \right) \left(f_{m - j} \ka^{-2m +2j + 1/3} \ka_{2m - 2j - 2} + \RR_{m-j}^\zeta \right) ds
		\\
		& -
		\sum_{\ell = 1}^{m-2}\sum_{i = 0}^{\ell - 1} \int_{\d \Omega} (2 \ka^{-2/3})^{-m-2}  \frac{(m - 1 - \ell +i )! \ell!}{i!} (m-\ell) (\ell+1) \times
		\\
		&
		\qquad \qquad \left(f_{m - \ell} \ka^{-2m + 2 \ell + 1/3} \ka_{2m - 2 \ell - 2} + \RR_{m-\ell}^\zeta \right) \left(f_{\ell + 1} \ka^{-2\ell - 2 + 1/3} \ka_{2 \ell} +\RR_{\ell + 1}^\zeta \right)ds.
	\end{split}
\end{align}
Each monomial in $\RR_m^\zeta$ has differential degree at most $2m - 2$ and those terms with differential degree equal to $2m - 2$ are of second or higher order. Since both $\RR_{m-\ell}^\zeta$ and $\RR_{\ell + 1}^\zeta$ are multiplied by at least one additional derivative of $\ka$ in the formula above, they are $(2m-2,3)$-harmless and Corollary \ref{cor:harmlessremainders} shows that they contribute only to the remainder $\RR_m$ in Theorem \ref{main}.
\\
\\
Expanding $A_{2m-1}$ in terms of $\ka$ using Proposition \ref{AM}, we have
\begin{align}
	\label{eq:Imzetakappaupsilon}
	\begin{split}
		\I_m =&
		\int_{\d \Omega} m! (2 \ka^{-2/3})^{-m-1} \times  \frac{2}{2m + 1} (2 \ka^{-2/3})^{-\frac{1}{2}} \times
		\\
		&
		\bigg(- 2^\frac{6m -1}{2} \frac{ \kappa_{2m-2 }}{(2m)!} \kappa^{-2m}
		\\
		&
		+  2^\frac{6m + 1}{2} \ka^{-2m-1}  \sum_{p = 2}^{2m - 2} p \frac{ \kappa_{p - 1}\ka_{2m - p - 1} }{(p+1)! (2m - p + 1)!}  + \mathcal{R}_{2m-1}^A - \Upsilon_{2m - 1} \bigg)ds
		\\
		&
		- \sum_{j =1}^{m-2} \int_{\d \Omega} (m-1)! (2 \ka^{-2/3})^{-m-2} (j + 1) (m - j) \times
		\\
		&
		\qquad \qquad  \left(\ka^{-2j - 2 + 1/3} f_{j+1}  \ka_{2j} + \RR_{j+1}^\zeta \right) \times
		\\
		&
		\qquad \qquad \qquad \left(f_{m - j} \ka^{-2m +2j + 1/3} \ka_{2m - 2j - 2} + \RR_{m-j}^\zeta \right) ds\\
		& -
		\sum_{\ell = 1}^{m-2}\sum_{i = 0}^{\ell - 1} \int_{\d \Omega} (2 \ka^{-2/3})^{-m-2}  \frac{(m - 1 - \ell +i )! \ell!}{i!} (m-\ell) (\ell+1) \times
		\\
		&
		\qquad \qquad \left(f_{m - \ell} \ka^{-2m + 2 \ell + 1/3} \ka_{2m - 2 \ell - 2} + \RR_{m-\ell}^\zeta \right) \times
		\\
		&
		\qquad \qquad \qquad \left(f_{\ell + 1} \ka^{-2\ell - 2 + 1/3} \ka_{2 \ell} +\RR_{\ell + 1}^\zeta \right)ds.
	\end{split}
\end{align}
The only remaining term which has not been fully expanded is $\Upsilon_{2m - 1}$, which is a polynomial in $\zeta_1^{\pm \half}, \zeta_2 \cdots, \zeta_{m-1}$ together with their $s$-derivatives of order at most $2m-2$. As explained in the proof of Proposition \ref{ZA}, these terms arise from a Lie series in which the Hamiltonian vector field of $\zeta$ is iterated. $\Upsilon_{2m - 1}$ itself contains both linear and second order terms, which we now find in Sections \ref{sec: linear terms with maximal derivatives} and \ref{sec: quadratic terms of maximal differential degree} respectively.
}

\subsection{Linear terms with maximal derivatives}\label{sec: linear terms with maximal derivatives}

Recall the formula for $\Theta_m$ in Proposition \ref{intinv}, which together with
\begin{align*}
	\zeta_m = \frac{2}{2m + 1} \zeta_1^{-1/2} \left(A_{2m - 1} - \Upsilon_{2m - 1}\right),
\end{align*}
in Proposition \ref{ZA} and
\begin{align*}
	A_{2m - 1} =& 	- 2^\frac{3M + 2}{2} \frac{ \kappa_{M - 1}}{(M+1)!} \kappa^{-M-1}\\
	&+ 2^\frac{3M + 4}{2} \ka^{-M-2}  \sum_{p = 2}^{M-1} p \frac{ \kappa_{p - 1}\ka_{M - p} }{(p+1)! (M - p + 2)!} + \RR_{M}^A,
\end{align*}
in Proposition \ref{AM}, gives a recipe for computing $\Theta_m$ and hence $\I_m$ in terms of curvature. In Theorem \ref{Linear coeff} below, we will show that modulo monomials containing at least two differentiated factors of $\ka$, $\zeta_m = F_m[\ka] \ka_{2m -2}$ \red{modulo second and higher order terms}, where $F_m$ is an algebraic function of $\ka^{\pm1/3}$. Part of $F_m$ comes from $A_{2m-1}$ while the other part arises from a single term in $\Upsilon_{2m - 1}$. The maximal derivatives turn out to appear linearly in $\zeta_m$, \red{in the sense of Definition \ref{def:linearquadratic}}. This allows us to plug in the highest order derivatives in $\zeta_n$ ($n < m$) to find the \red{second order} part of $\Upsilon_{2m-1}$, which together with that of $A_{2m-1}$, gives $\mathcal{P}_m(\ka^{\pm \frac{1}{3}}, \ka_1, \cdots, \ka_{m-1})$ in Theorem \ref{main}.

\begin{theo}\label{Linear coeff}
	For $m \geq 2$, the coefficients $\zeta_m$ of the interpolating Hamiltonian are of the form
	\begin{align*}
		\zeta_m = f_m \kappa^{-2m + 1/3} \kappa_{ 2m - 2} + \mathcal{R}_m^\zeta,
	\end{align*}
	where
	\begin{align*}
		f_m =& - \frac{2^{3m + 1}}{(2m)!} B_{2m} = (-1)^{m} \frac{2^{3m + 2}}{(2\pi)^{2m}} \zr(2m).
	\end{align*}
	Here, $\zr$ is the Riemann $\zeta$-function and $B_{2n}$ are the even Bernoulli numbers. $\mathcal{R}_m^\zeta$ is a remainder term which \red{is $(2m-2,2)$-harmless}.
\end{theo}

\begin{rema}
	 \red{If $s > 0$, multiplication of any differential degree $2m - 2 - s$ curvature monomial in $\RR_m^\zeta$ by another curvature monomial having differential degree $s$ will then be $(2m-2,3)$-harmless in the sense of Definition \ref{def:harmless}. The structure of $\RR_m^\zeta$ will be important in applying Corollary \ref{threeterms} (equivalently, Corollary \ref{cor:harmlessremainders}) in Section \ref{subsec: integration by parts}. }
\end{rema}

\begin{proof}
	In Proposition \ref{ZA}, we saw that $\zeta_m$ can be determined from the data $A_{2m - 1}, \zeta_1, \cdots, \zeta_{m -1}$:
	\begin{align*}
		\zeta_m= \frac{2}{2m + 1} \zeta_1^{-\frac{1}{2}} \left(A_{2m- 1} - \Upsilon_{2m - 1} \right),
	\end{align*}
	and in formula \eqref{Aps}, together with Proposition \ref{AM}, we computed $A_{2m-1}$ geometrically modulo lower order terms. We now determine more carefully the structure of $\Upsilon_{2m - 1}$, separating out a sum of linear terms \red{(in the sense of Definition \ref{def:linearquadratic})} arising from $L_{8,1,1}^k \zeta_{m - k}$.
	\\
	\\
	From equations \eqref{smallzeta} and \eqref{smallA}, we see that the Theorem is satisfied for $m = 2$. We now proceed inductively \red{with respect to $m$, while still retaining the notation $M= 2m-1$ to avoid cumbersome notation}. Assume the proposition is true for $2 \leq n \leq m - 1$ and write
	\begin{align}
		\label{eq:zetakappa}
		\begin{split}
			\zeta_{m} &= \frac{2}{2m + 1} \zeta_1^{-\frac{1}{2}} \left(A_{2m - 1} - \Upsilon_{2m - 1} \right)
			\\
			=&
			\frac{2}{2m + 1} \frac{\ka^{1/3}}{\sqrt{2}}  \left(- 2^{\frac{6m - 1}{2}}  \frac{\ka_{2m - 2}}{(2m)!} \ka^{-2m} - \Upsilon_{2m - 1}  \right) + R_{2m - 1, 0}^\zeta
			\\
			=& - \frac{8^m}{(2m + 1)!}  \ka_{2m - 2} \ka^{-2m + 1/3} - \frac{\sqrt{2} \ka^{1/3}}{2m+1} \Upsilon_{2m -1} + R_{2m -1, 0}^\zeta,
		\end{split}
	\end{align}
where
\begin{align*}
	R_{2m - 1,0}^\zeta = \frac{\sqrt{2} \ka^{1/3}}{2m + 1} \left(A_M^{(2)} + \RR_M^A\right)
\end{align*}

\red{is $(2m-2,2)$-harmless as in the statement of the theorem; $A_{2m-1}^{(2)}$ is the second order part of $A_{2m-1}$ in Proposition \ref{AM} and $\RR_{2m - 1}^A$ is the remainder, which contains higher order terms in the sense of Definition \ref{def:linearquadratic}}.
	\\
	\\
	Set $K = 2k + 1$ and consider a term in the sum \eqref{sp} which contributes to the coefficient $A_M$ of $\lambda^{M/2}$. Since $\zeta^{K/2} = O(\lambda^{K/2})$ for all  $K$, only terms from $\zeta^{K/2} X_\zeta^K s$ with  $1 \leq K \leq M$ contribute to $A_M$. In the notation of Definition \ref{notationdef}, we have
	\begin{align*}
		A_M &= \sum_{k = 0}^{\frac{M-1}{2}} \Lambda_M \left[\zeta^{k + 1/2} \z_{2k + 1} \right] = \sum_{k = 0}^{\frac{M-1}{2}} \frac{1}{(2k + 1)!} \Lambda_M \left[\zeta^{k + 1/2} L^k \z_{1} \right]\\*
		&=  \sum_{k = 0}^{\frac{M-1}{2}} \frac{1}{(2k + 1)!} \sum_{\substack{M_1 + M_2 = M\\ M_1 \geq 1\\ M_2 \geq 0 }} 		\Lambda_{M_1}\left[\zeta^{k + 1/2}\right]  \Lambda_{M_2} \left[L^k  \z_{1} \right]\\*
		&=  \sum_{k = 0}^{\frac{M-1}{2}} \frac{1}{(2k + 1)!} \sum_{\substack{M_1 + M_2 = M\\ M_1 \geq 1\\ M_2 \geq 0 }} 		\Lambda_{M_1}\left[\zeta^{k + 1/2}\right]  \sum_{{\bm{\sigma}}:\Z_{k} \to \Z_{8}}  \Lambda_{M_2} \left[ L_{\bm{\sigma}} \z_{1} \right].
	\end{align*}
	The last sum is over all maps $\bm{\sigma}: \Z_k \to \Z_8$ and contains terms of the form 
	\begin{align*}
		\sum_{\bm{\sigma}:\Z_{k} \to \Z_{8}} \sum_{0 \leq j_1 \leq j_2 \leq \cdots \leq j_k \leq M_2} \Lambda_{M_2} \left[ L_{{\sigma}_{k}} \Lambda_{j_{k}} \left[ L_{{\sigma}_{{k-1}}} \Lambda_{j_{k-1}} \left[\cdots \right] \right]\right].
	\end{align*}
	\vspace{0.1in}

	For $k = 0$ we have by the proof of Proposition \ref{ZA} above, that the term $- \zeta^{1/2} X_\zeta s = \zeta^{1/2} \z_1$ in the first line of \eqref{first few} is given by
	\begin{align*}
		\zeta^{1/2} \frac{\d \zeta}{\d \lambda} &= \zeta_1^{1/2} \lambda^{1/2} \left(  \sum_{i = 0}^\infty \sum_{j = 0}^i  \binom{\frac{1}{2}}{j} \sum_{\substack{i_1 + \cdots + i_j = i\\ i_p \geq 1}} \wt{\zeta}_{i_1 + 1} \cdots \wt{\zeta}_{i_j + 1} \lambda^i \right) \left(\sum_{\ell = 1}^\infty \ell \zeta_\ell \lambda^{\ell-1}\right).
	\end{align*}
	Recall from formula \eqref{j0} in Case $1$ of the proof of Proposition \ref{ZA} that the maximal terms from which we found $\zeta_m$ come from the endpoints $i = \frac{M-1}{2}, j = 1, \ell = 1$ and $i = j = 0, \ell = \frac{M+ 1}{2}$. Note that all of the terms in $\Lambda_M \left[\zeta^{1/2} \frac{\d \zeta}{\d \lambda}\right]$ have $M = 1 + 2i + 2 \ell - 2$ and indices satisfying
	\begin{align*}
		& 2 (i_1 + 1) - 2 + 2(i_2 + 1) - 2 + \cdots + 2(i_j + 1) - 2 + 2 \ell - 2\\
		&= 2i + 2 \ell - 2 = M - 1 = 2m - 2.
	\end{align*}
	Separating out the maximal terms, we see by the inductive hypothesis that the remaining terms have $\zeta$-indices {less than or equal to} $m-1$ and each monomial has differential degree {at most} $2m - 2$. In each case, we have $i + \ell = \frac{M+1}{2}$. If $j > 1$, there are at least two terms with $\zeta$-indices {greater than or equal to} two. When $j = 1$, if $i \neq 0, \frac{M-1}{2}$, then $\ell \geq 2$, in which case there are again at least two terms with $\zeta$-indices greater than or equal to two. Hence, there are no terms other than $\frac{2m+1}{2} \zeta_1^{1/2}\zeta_m $ in $\Lambda_M[\zeta^{1/2} \z_1]$ which have $\ka$-derivatives of order {greater than} $2m - 3$; in particular, all remaining terms in $\Lambda_M[\zeta^{1/2} \z_1]$ which have differential degree equal to $2m-2$ are $(2m-2,2)$-harmless.
	\\
	\\
	We now claim that for $k \geq 1$, all terms in $\zeta^{k + 1/2} L^k \z_1$, except for one, generate data with submaximal differential degree and/or contain \red{multiple derivatives} distributed across at least two factors, as in the statement of the theorem. The exceptional term will be $\zeta^{k + 1/2} L_{8,1,1}^k \z_1$ for reasons to be made clear shortly. \red{The term $\zeta^{(2k+1)/2}$ already contributes a factor of $\lambda^{(2k + 1)/2}$, so in order to appear in the coefficient of $\lambda^{(2m - 1)/2}$, any term arising from $L^k \z_1$ should have order at most $2m - 2k - 2$.}
	\\
	\\
	\red{\textbf{Claim 1:}} \red{Under the induction hypothesis, for} each $0 \leq k \leq m - 1$ and $0 \leq j \leq 2m - 2k - 2$ even, $\Lambda_j\left[\z_{2k + 1} \right]$ is a polynomial in the curvature jet with coefficients in $\R[\ka^{\pm \frac{1}{3}}]$ and differential degree at most $j + 2k$.
	\\
	\\
	\red{To see why this is true, we begin} with $\z_1$. By the primary induction hypothesis on $\zeta_i$, for $k = 0$ and any $0 \leq j \leq 2m - 2k - 2$ even, we have that
	\begin{align}\label{Zdiff}
		\begin{split}
		\Lambda_{j}\left[\z_1\right] =& \Lambda_{j} \left[\sum_{i = 1}^\infty i \zeta_i \lambda^{i - 1}\right]
		\\
		=&
		{\left(\frac{j}{2} + 1 \right) \zeta_{\frac{j}{2} + 1} }
		\\
		=&
		\left(\frac{j}{2} + 1\right) f_{\frac{j}{2} + 1} \kappa^{- j - \frac{5}{3}} \kappa_{j} 
		+\left(\frac{j}{2} + 1\right)  \mathcal{R}_{\frac{j}{2} + 1}^\zeta.
	\end{split}
	\end{align}
	In particular, $\Lambda_{j}\left[\z_1\right]$ has differential degree at most $j$. The term $\ka_{j}$ has differential degree $j + 2 \cdot 0$, but arises from $L_{8,1,1}^k \z_1$ with $k = 0$.	 \red{It is easy to see from the formula in Definition \ref{L} that each subsequent application of an operator $L_{\sigma_i, q_i, r_i}$ increases the half power of $\lambda$ by $2 (q_i + r_i -2)$ and raises the differential degree by at most $2q_i + 2r_i -2$. Hence, each application raises the differential degree at most two more than it raises the half power and after $k$ applications of $L$, the differential degree will have increased by no more than $2k$.
	}
	\\
	\\
	\textbf{\red{Claim 2.}} Each monomial of differential degree equal to $j + 2k$ in $\Lambda_j[\z_{2k + 1}]$ which \red{does not arise from} $L_{8,1,1}^k \z_1$ \red{is of second or higher order, i.e.,} containing at least two factors of $\kappa$ derivatives having order \red{at least} one.
	\\
	\\
	To prove Claim 2, let $\Y$ be any monomial in $\z_{\ell}$ with $1 \leq \ell \leq 2k + 1$. Except for $L_5$ and $L_{8,1,1}$, $L_{i,q,r} \mathcal{Y}$ contains at least two separate factors being differentiated in $\zeta_q, \zeta_r$ or $\mathcal{Y}$. Since each of these is assumed to be a polynomial in the curvature jet, there are at least two separated derivatives. For the term $L_{5}  \Y \lambda^{\frac{j}{2}}$, the only case in which there is a single factor being differentiated is when $q = 1$ and $\Y = \zeta_1$. But $\zeta_1$ only appears undifferentiated in $\z_1 = \frac{\d \zeta}{\d \lambda}$ as $\zeta_1 \lambda^0$, in which case the operator $\frac{\d}{\d \lambda}$ in $L_5$ annihilates it. This leaves only terms arising from $L_{8,1,1}^k$, which finishes the proof of Claim 2.
	\\
	\\
	To complete the proof of the theorem, we now estimate the structure and differential degree of the terms $\zeta^{k + 1/2} \z_{2k + 1}$ for $k > 0$. In order to keep only the highest order derivatives, $\frac{\d^2}{\d s^2}$ in $L_{8,1,1}$ should be applied repeatedly to the coefficients $\zeta_{\frac{2m - 1 -(2k+1) + 2}{2}}$. By the inductive hypothesis, \red{the contribution of $\Upsilon_{2m - 1}$ to the highest order derivatives of $\kappa$ in $\zeta_m$ is of the form}
	\begin{align*}
		\sum_{k = 1}^{m-1}& \frac{1}{(2k + 1)!} \zeta^{(2k + 1)/2} X_\zeta^{2k} \z_1\\
		=& \sum_{k = 1}^{m-1} \frac{1}{(2k+1)!} \zeta_1^{(2k + 1)/2} L_{8,1,1}^k \z_{1, 2m  - 2k - 2} + \mathcal{H}_{2m-2,2}\\
		=& \sum_{k = 1}^{m-1} \frac{(m-k)}{(2k + 1)!} \zeta_1^{(2k + 1)/2} \zeta_1^{2k} \frac{d^{2k}}{d s^{2k}} \zeta_{m - k} + \mathcal{H}_{2m-2,2}\\
		= & \sum_{k = 1}^{m-1} \frac{(m-k)}{(2k + 1)!} \zeta_1^{3k + 1/2} F_{m-k} [\ka] \ka_{2m - 2} + \mathcal{H}_{2m-2,2},
	\end{align*}
	where $\mathcal{H}_{2m-2,2}$ denotes \red{$(2m-2,2)$ harmless remainders} and $F_{m-k}[\ka]$ is an algebraic function of $\ka^{\pm\frac{1}{3}}$ with coefficients depending only on $m - k$. For $m - k = 1$, we already know that
	\begin{align*}
		\zeta_1 = 2 \ka^{-2/3}.
	\end{align*}
	We set $F_1 = (-4/3) \ka^{-5/3}$, so that
	\begin{align*}
		\frac{d^{2k} \zeta_1}{d s^{2k}} = F_1[\ka] \ka_{2k} + \mathcal{H}_{2k, 2}.
	\end{align*}
	Taking $M = 2m - 1$ and $1 < k < m$, the $F_k$ are determined inductively. To find $F_m$, note that for $m > 1$ we have
	\begin{align*}
		\zeta_m &= \frac{2}{2m + 1} \zeta_1^{-1/2} \left( A_{2m - 1} - \Upsilon_{2m - 1} \right) = F_m[\ka] \ka_{2m - 2} + \mathcal{H}_{2m-2,2}.
	\end{align*}
	Denote by
\begin{align*}
	a_{2m - 1} = - \frac{8^{m}}{ \sqrt{2} (2m)!},
\end{align*}
the coefficient of $\kappa_{2m-2}\ka^{-2m}$ in $A_{2m-1}^{\text{lin}}$ given in Proposition \ref{AM}. Then, \red{we have}
	\begin{align}\label{Recurrence}
		F_m & = \frac{2}{2m + 1} \zeta_1^{-1/2} a_{2m -1} \red{\ka^{-2m}}  - \frac{2}{2m +1} \sum_{k = 1}^{m-1} \frac{m-k}{(2k + 1)!}  \zeta_1^{3k} F_{m - k}.
	\end{align}
	Using the induction hypothesis, we immediately see that $F_m$ is a multiple of $\ka^{-2m +1/3}$ \red{by a universal constant \red{$f_m$} of combinatorial nature,} as in the statement of the theorem. We can find the coefficient $f_m$ explicitly as follows. The infinite order recurrence relation \eqref{Recurrence} can be written as
	\begin{align}\label{recursionsequence}
		f_m = - \frac{8^m}{(2m+1)!} - \frac{2}{2m + 1} \sum_{k = 1}^{m-1} \frac{8^k}{(2k + 1)!} (m-k) f_{m-k}, \quad f_1 = -4/3.
	\end{align}
	We assemble the coefficients into the following generating functions:\footnote{\red{It is not obvious as of yet that the generating function for $f_m$ is convergent, so for now, we write $f$ as a formal power series.}}
\begin{align*}
	f(z) =& \sum_{m = 1}^\infty f_m z^m,\\
	a(z) =& \red{-} \sum_{m = 1}^\infty \frac{2}{2m + 1} \frac{a_{2m - 1}}{\red{\sqrt{2}}} z^m\\
	=& \sum_{m = 1}^\infty \frac{8^m}{(2m + 1)!} z^m = \frac{\sinh( \sqrt{8 z})}{\sqrt{8z}} - 1.
\end{align*}
\red{Now, for any formal power series $g = \sum_{m = 1}^\infty g_m z^m$, we have the following identity:} 
$$
\sum_{m = 1}^{\infty} (2m + 1) g_m z^m = 2(z g)' - g.
$$
\red{Setting} $g(z) = f(z) + a(z)$, we can rewrite the recurrence relation \eqref{recursionsequence} as
\begin{align*}
	2 \frac{d}{d z} \bigg(z f(z) + z a(z)\bigg) - f(z) - a(z) = - 2 z a(z) \frac{d f}{d z},
\end{align*}
or equivalently,
\begin{align*}
		f' + \frac{1}{2z(1+ a)} f = - \frac{2z a' + a}{2z (1+ a)},
\end{align*}
which is a simple ODE. The Taylor coefficients of $f$ will then give the sequence $f_m$. Simplifying using double and half angle formulas for the hyperbolic sine function, we have
\begin{align*}
	 \frac{1}{2z (1 + a(z))} =&  \sqrt{\frac{2}{z}} \frac{1}{\sinh \sqrt{8z}},\\
	- \frac{2z a'(z) + a(z)}{2 z (1 + a(z))} =&- \sqrt{\frac{2}{z}} \tanh \sqrt{2 z}.
\end{align*}

If we define the integrating factors
\begin{align*}
	\mu_\pm (z) = \exp \pm \int^z \sqrt{\frac{2}{w}} \frac{1}{\sinh\sqrt{8 w}} dw =  \left(\tanh \sqrt{2 z}\right)^{\pm 1},
\end{align*}
the equation then simplifies to
\begin{align*}
	f'(z) + \sqrt{\frac{2}{z}} \frac{1}{\sinh \sqrt{8z}} f(z) = \mu_-(z) \frac{d}{d z}\big(f (z) \mu_+(z) \big) = - \sqrt{\frac{2}{z}} \mu_+(z),
\end{align*}
which has the solution
	\begin{align}
		\label{eq:fofz}
		\begin{split}
			f(z) = & - \red{\coth \sqrt{2z} }\int^z \sqrt{\frac{2}{w}} \tanh^2 \sqrt{2 w} dw + C \coth \sqrt{2z}\\
			= & 2 - \sqrt{8 z} \coth \sqrt{2z} + C \coth \sqrt{2z}.
		\end{split}
	\end{align}
One can easily check that this indeed solves the equation. The initial condition $f(0) = 0$ gives $C = 0$ and one checks that $\lim_{z \to 0} 2 - \sqrt{8z} \coth \sqrt{2z} = 0$ automatically. Taylor expanding at $z = 0$, we have
\begin{align}
	f(z) = 2 - \sqrt{8 z} \sum_{n = 0}^{\infty} \frac{2^{2n} B_{2n} (2z)^{\frac{2n - 1}{2}}}{(2n)!},
\end{align}
with $B_{2n}$ being the even Bernoulli numbers \red{(see \url{https://dlmf.nist.gov/4.19})}. The theorem then follows from the relation between Bernoulli numbers and the Riemann zeta function at the even integers.
\end{proof}

\begin{rema}
	\red{Expanding \eqref{eq:fofz}, we see that
	\begin{align*}
		f(z) = & 2 - \sqrt{8 z} \coth \sqrt{2z} \sim \sum_{m =1}^\infty f_m z^m\\
		=& -\frac{4}{3} z + \frac{8}{45} z^2 + O(z^3).
	\end{align*}
	In particular, $f_2 = 8/45$ corroborates the coefficient of $\ka_2$ in $\zeta_2$ from equation (4.5) of \cite{MM} (equivalently, formula \eqref{eq:zetatwo} above).}
\end{rema}

\red{
 In light of Theorem \ref{Linear coeff}, we can now explain the significance of the term $\zeta$-weight from Definition \ref{def:zeta weight}.
 
\begin{coro}
	\label{cor:zeta weight bounds ddeg}
	For any polynomial $\q$ in the jet of $\{\zeta_1, \cdots, \zeta_m\}$ with coefficients in $\R[\zeta_1^{\pm 1/2}]$,
	\begin{align*}
		\deg_{\d, \kappa} (\q) \leq \deg_{\d, \zeta}(\q) + w_\zeta(\q),
	\end{align*}
	where $\deg_{\d, \kappa}(\q)$ and $\deg_{\d, \zeta}(\q)$ are the multi-variable $\kappa$ and $\zeta$ differential degrees of $\q$ and $w_\zeta(\q)$ is its $\zeta$-weight (cf. Definitions \ref{def:zeta weight} and \ref{def: multi-fold diff degree}).
\end{coro}
}

\red{Throughout the remainder of the paper, we will use
\begin{align}
	\label{eq:msfzk}
	\msfz_k = \zr(2k) = \sum_{q = 1}^\infty \frac{1}{q^{2k}}
\end{align}
to denote the Riemann $\zeta$-function at even integers. This notation is not to be confused with $\z_K$ from Definition \ref{Zk}, which was the $K$th term of a Lie series.}

\subsection{\red{Second order terms of maximal differential degree}}
\label{sec: quadratic terms of maximal differential degree}

\red{We now list all second order terms in $\Upsilon_{2m - 1}$.} \red{Using} the integration by parts algorithm \red{in} Proposition \ref{IBPalgorithm} and Corollary \ref{threeterms}, we only need to find terms in $\mathcal{P}_{m}[\ka]$ \red{(the integrand of $\I_m$ in Theorem \ref{main})} with \red{maximal differential degree and} at most two differentiated factors of $\ka$. By Theorem \ref{Linear coeff}, these are precisely the monomials in the Taylor coefficients of $\zeta$ \red{with maximal differential degree and} at most two nontrivial \red{(greater than or equal to two)} $\zeta$-indices, \red{which contribute to the integrand $\Theta_m[\zeta]$ of $\I_m$ (cf. Proposition \ref{intinv})}.
\\
\\
\red{Recall the formula
\begin{align*}
	A_{2m-1} = \Lambda_{2m - 1} \left[\sum_{k = 0}^\infty \frac{\zeta^{(2k+1)/2}}{(2k+  1)!} L^k \z_1 \right]
	= \frac{2m+1}{2} \zeta_1^{1/2} \zeta_m + \Upsilon_{2m - 1}
\end{align*}
from Proposition \ref{ZA}. The term $\Upsilon_{2m - 1}$ consists of $\lambda^{(2m-1)/2}$-coefficients which do not come from ``endpoint'' indices, i.e., those which generate $\zeta_m$ as a factor. The \red{second order} terms of maximal differential degree in $\Upsilon_{2m - 1}$ arise in a few different ways. We organize them as follows.}
\\
\\
\red{Consider a $k$-letter word $\bm{\sigma} = (\sigma_1, \cdots, \sigma_k)$ in the alphabet $\{1,2,3,4,5,6,7,8\}$. To each such word corresponds a sequence of compositions as in \eqref{eq:wordcompositions}. We will say that a letter $\sigma_b$ is \textbf{operator-left} of a letter $\sigma_a$ if $b > a$:
\begin{align*}
	L_{\bm{\sigma}} = L_{\sigma_k} \cdots \circ L_{\sigma_b} \circ \cdots \circ \cdots L_{\sigma_a} \cdots L_{\sigma_1}.
\end{align*}
The dictionary of all such words can be organized into the following mutually exclusive types:}

\red{
\begin{itemize}
	\item \textbf{Type-1. } If $k = 0$, there is only one word, the empty word.
	
	\item \textbf{Type-2.} $k \geq 1$ and the word contains at least one of the letters in $\{1,2,3\}$.
	
	\item \textbf{Type-3. } $k \geq 1$, the word contains no letters in $\{1,2,3\}$, \textit{exactly one} of the letters $4,5,6,$ or $7$, and no incidence of $8$ to the operator-left of that letter.
	
	\item \textbf{Type-4.} $k \geq 1$, the word contains no letters in $\{1,2,3\}$, exactly one $4,5,6$ or $7$, and at least one incidence of $8$ to the operator left of that letter.
	
	\item \textbf{Type-5.} $k \geq 1$, the word has no letters in $\{1,2,3\}$ and contains at least two of the letters $4,5,6$ or $7$, counted with multiplicity.
	
	\item \textbf{Type-6.} $k \geq 1$ and the word contains no letters other than $8$.
\end{itemize}

\begin{exam}
	Consider the $4$-letter word $\bm{\sigma} = (4,6,7,8)$. It corresponds to the operator
	\begin{align*}
		L_{\bm{\sigma}} = L_8 \circ L_7 \circ L_6 \circ L_4,
	\end{align*}
	which is type-$5$. Here are some other prototypical $4$-letter words of each type
	\begin{align*}
		(1,3,4,6), & \qquad \text{type-$2$},
		\\
		(8, 8, 8, 4), & \qquad \text{type-$3$},
		\\
		(8, 8, 4, 8), & \qquad \text{type-$4$},
		\\
		(4,8,6,8), & \qquad \text{type-$5$},
		\\
		(8,8,8,8), & \qquad \text{type-$6$}.
	\end{align*}
\end{exam}
These word types classify the \textit{differential operator} part of a term in $\zeta^{(2k+1)/2} \frac{1}{(2k+1)!} L^k \z_1$ (i.e., that which comes from $L^k$ rather than multiplication by $\zeta^{(2k+1)/2}$). Below, we present several cases which incorporate both the differential operator part and the multiplicative part from $\zeta^{(2k+1)/2}$, which introduces nontrivial $\zeta$-indices. This allows us to separate out second order terms from $(2m - 2,3)$-harmless remainders.}
\\
\\
\textbf{Type 1 words.} $k = 0$ (equivalently, $K = 2k+1 = 1$). \red{The empty word (type-$1$)} generates the product $\zeta^{1/2} \frac{\d \zeta}{\d \lambda}$. Recall the structure of $\Lambda_{M_1}\left[\zeta^{K/2}\right]$ from Lemma \ref{maxind}:
\red{
\begin{align*}
	\Lambda_{M_1} \left[ \zeta^{\frac{K}{2}} \right] =& \frac{K}{2} \zeta_1^{\frac{K -2}{2}} \zeta_{\frac{M_1 - K +2}{2}} + \left(\frac{K^2 - 2K}{8}\right) \zeta_1^{\frac{K- 4}{2}}\\
	&\times \sum_{\substack{i_1 + i_2 = \frac{M_1 - K}{2}\\ i_\ell \geq 1}} \zeta_{i_1 + 1} \zeta_{i_2 + 1} + \upsilon_{M_1, K}[\zeta].
\end{align*}
Choosing $M_2$ so that $M_1 + M_2 = 2m - 1$, we can write the contribution of $k = 0$ terms to $A_{2m-1}$ as
\begin{align*}
	\sum_{M_1 + M_2 = 2m - 1} \Lambda_{M_1} \left[\zeta^{1/2}\right] \Lambda_{M_2} \left[\frac{\d \zeta}{\d \lambda}\right].
\end{align*}
}
For the factor
\begin{align*}
	\frac{1}{2} \zeta_1^{-\frac{1}{2}} \zeta_{\frac{M_1 + 1}{2}} \red{\left(\frac{M_2}{2} + 1\right) \zeta_{\frac{M_2}{2} + 1}},
\end{align*}
the endpoints $M_1 = 1, M_2 = 2m -2$ and $M_1 = 2m-1, M_2 = 0$ are not included in $\Upsilon_{2m - 1}$ as they involve $\zeta_m$. Hence, we take $M_1 + M_2 = 2m - 1$ and $3 \leq M_1 \leq 2m - 3$ to be odd. \red{Reindexing with $2j + 1 = M_1$, we get
\begin{align*}
	\frac{1}{2} \zeta_1^{-1/2} \sum_{j = 1}^{m-2} \zeta_{j+1} (m - j) \zeta_{m - j}.
\end{align*}
}
In the \red{term}
\begin{align*}
	\left(\frac{K^2 - 2K}{8}\right) \zeta_1^{\frac{K- 4}{2}} \sum_{\substack{i_1 + i_2 = \frac{M_1 - K}{2}\\ i_\ell \geq 1}} \zeta_{i_1 + 1} \zeta_{i_2 + 1} + \upsilon_{M_1, K}[\zeta],
\end{align*}
the condition $K = 1$ forces $M_1 > 1$, so we take $M_1 = 2m - 1$ and $M_2 = 0$; if $M_2 \geq 1$, there are at least three nontrivial $\zeta$-indices, \red{so its contribution to $\Theta_m$ is $(2m-2,3)$-harmless}. \red{Reindexing with $2 i + 1 = M_1$, its contribution to $\Upsilon_{2m - 1}$ is
\begin{align*}
	- \frac{1}{8} \zeta_1^{-3/2} \sum_{i = 1}^{m - 2} \zeta_{i + 1} \zeta_{m - i} \zeta_1.
\end{align*}
}
The relevant \red{second order} terms \red{which contribute to $\Upsilon_{2m - 1}$ when $k = 0$} are then
\begin{align*}
	\frac{1}{2} \zeta_1^{-1/2} &\sum_{j = 1}^{m - 2} \zeta_{j+1} (m-j) \zeta_{m - j}
	 - \frac{1}{8} \zeta_1^{-3/2} \sum_{i = 1}^{m - 2} \zeta_{i + 1} \zeta_{m - i} \zeta_1
	 \\
	=&
	\zeta_1^{-1/2} \sum_{i = 1}^{m - 2} \left(\frac{(m-i)}{2} - \frac{1}{8}\right) \zeta_{i + 1} \zeta_{m - i}.
\end{align*}
\vspace{0.1in}
\\
\textbf{Type 2 words} \red{Suppose that $k \geq 1$ and consider the action of a type-$2$ word, i.e., at least one of the operators $L_1, L_2$ or $L_3$ is applied. Proposition \ref{ZA} shows that the action of the differential operators $L_1, L_2$ and $L_3$ on any term $\nu \lambda^p$ in $L^{j} \z_1$ (with $0 \leq p \leq 2m - 2$ and $1 \leq j \leq k-1$) satisfies
\begin{align}
	\label{eq:L123}\begin{split}
	w_\zeta (\Lambda_q [L_i \nu \lambda^p ]) + \deg_{\d, \zeta} ( \Lambda_q [L_ i \nu \lambda^p]) \leq 2m - 2,
	\end{split}
\end{align}
for $ i = 1,2,3,$ and $0 \leq q \leq 2m - 1$; those terms for which \eqref{eq:L123} is an equality contain at least three differentiated factors of $\ka$. Indeed,
\begin{align*}
	L_1 = \frac{\d \zeta}{\d s} \frac{\d^2 \zeta}{\d s \d \lambda} \frac{\d}{\d \lambda}
\end{align*}
contains two $s$ derivatives of $\zeta$ in the coefficients and $\frac{\d}{\d \lambda}$ either (1) annihilates {$\zeta_1$ while preserving $\zeta_2 \lambda$ in $\z_1$  or (2) is precomposed with another differential operator $L_{i'}$ (to the right) which produces at least one $s$-derivatives of $\zeta$.} The same order counting principle clearly applies to
\begin{align*}
	L_2 = \left(\frac{\d \zeta}{\d s} \right)^2 \frac{\d^2}{\d \lambda^2}
	\quad \text{and} \quad
	L_3 = - \frac{\d \zeta}{\d s} \frac{\d^2 \zeta}{\d \lambda^2} \frac{\d}{\d s}.
\end{align*}
In other words, their contributions to $\Upsilon_{2m - 1}$ are $(2m-2, 3)$-harmless.}
\\
\\
\textbf{Type 3 words.} \red{Again take $k \geq 1$.} If any one of $L_4, \cdots, L_7$ is applied to \red{some term $\nu \lambda^p$ in $L^j \z_1$ with $0 \leq p \leq 2m - 1$ and $1 \leq j \leq k-1$, Proposition \ref{ZA} again tells us that \eqref{eq:L123} applies for $1 \leq i \leq 7$; the same order counting also shows that if \eqref{eq:L123} is an equality, then any term in $\Lambda_q[L_i \nu]$ with $4 \leq i \leq 7$ has at least \textit{two} differentiated factors of $\ka$, i.e., it is $(2m-2,2)$-harmless. Therefore, we have the following effects of type-$3$ words, which contribute} terms of the form $\zeta^{k + 1/2} L_j L_8^{k-1}$, with $4 \leq j \leq 7$. As $\zeta = O(\lambda)$ and each $L_i$ with $4 \leq i \leq 7$ has at least two nontrivial factors (\red{i.e., having} $\zeta$-index \red{at least} two and/or an $s$-derivative on $\zeta_1$), the term in the expansion of $\zeta^{k+1/2}$ must be $\zeta_1^{k + 1/2}$; \red{otherwise its contribution would be $(2m-2,3)$-harmless}. Recall the formula for $L = X_\zeta^2$ in Definition \ref{L}. If $k \geq 1$, then we look for terms in $\Lambda_{M_2}[L_i L_8^{k-1} \z_1]$ with $M_2 = 2m - 2k - 2$ so that $2k + 1 + M_2 = 2m-1$:
\begin{align*}
	\Lambda_{M_2}\left[L_4 L_8^{k-1} \z_1\right] =& - \sum_{p + q = m - k + 1} \dot{\zeta}_p \zeta_1 \zeta_1^{2(k-1)} q(q-1) \frac{d^{2k - 1} \zeta_q}{d s^{2k - 1}} + \mathcal{H}_{2m-2, 3}\\
	\Lambda_{M_2}\left[L_5 L_8^{k-1} \z_1\right]  =& - \sum_{p + q = m - k + 1} \zeta_1 \ddot{\zeta}_p \zeta_1^{2(k-1)} q(q-1) \frac{d^{2k - 2} \zeta_q}{d s^{2k - 2}} + \mathcal{H}_{2m-2, 3}\\
	\Lambda_{M_2}\left[L_6 L_8^{k-1} \z_1\right] =& \Lambda_{M_2} [L_4 L_8^{k-1} \z_1]\\
	=& - \sum_{p + q = m - k + 1} \dot{\zeta}_p \zeta_1 \zeta_1^{2(k-1)} q(q-1) \frac{d^{2k - 1} \zeta_q}{d s^{2k - 1}} + \mathcal{H}_{2m-2, 3}\\
	\Lambda_{M_2}\left[L_7 L_8^{k-1} \z_1 \right] =& \sum_{p + q = m - k + 1} \zeta_1  p \dot{\zeta}_p \zeta_1^{2(k-1)} q  \frac{d^{2k - 1} \zeta_q}{d s^{2k - 1}} + \mathcal{H}_{2m-2, 3}.
\end{align*}
\red{Setting $p = r + 1$, $q = m - k - r$ and reindexing,} the last two terms then combine to give
\begin{align*}
	L_6 L_8^{k-1} \z_1& + L_7 L_8^{k-1} \z_1 \\
	=&  - \sum_{p + q = m - k + 1}  q(q-1)\dot{\zeta}_p \zeta_1^{2k-1} \frac{d^{2k - 1} \zeta_q}{d s^{2k - 1}}
	\\
	&+ \sum_{p + q = m - k + 1}  p q \zeta_1^{2k - 1} \dot{\zeta}_p  \frac{d^{2k - 1} \zeta_q}{d s^{2k - 1}} + \mathcal{H}_{2m - 2, 3}\\
	=& \sum_{r = 0}^{m - k - 1} (m - k - r) (r + 1 - (m-k - r - 1)) \zeta_1^{2k - 1} \dot{\zeta}_{r + 1} \frac{\d^{2k - 1} \zeta_{m - k - r}}{d s^{2k - 1}}\\
	&+ \mathcal{H}_{2m - 2, 3}.
\end{align*}
\\
\textbf{Type 4 words.} \red{Type-$4$ words generate terms of the form
\begin{align*}
	\Y_3 L_8 \Y_2 L_j \Y_1,
\end{align*}
for some differential polynomials $\Y_1, \Y_2, \Y_3$ in the jet of $\zeta$.} Any incidence of $L_8$ to the \textit{operator left} of one of $L_4, \cdots, L_7$ will generate a term which is cohomologous to a \red{$(2m-2,3)$-harmless} remainder. More precisely, denote by $[\mathcal{Y}ds]_{\text{dR}}$ the de Rham cohomology class of a one-form $\mathcal{Y}ds$, where $\mathcal{Y}$ is a monomial in the jet of $\zeta^{1/2}$ and $ds$ is the arclength measure on $\d \Omega$. If $4 \leq j \leq 7$, then for any such \red{$\mathcal{Y}_1, \Y_2, \Y_3$, }we have
\begin{align*}
	\left[\red{\Lambda_{2m - 2}} \left[\Y_3 L_8 \Y_2 L_j \Y_1\right] \right]_{\text{dR}}
	=&
	\left[ \Lambda_{2m - 2}\left[ \red{\Y_3} \sum_{p, q} pq \zeta_p \lambda^{p - 1} \zeta_q \lambda^{q - 1} \left(\frac{d}{d s}\right)^2 \Y_2 L_j \Y_1 ds \right] \right]_{\text{dR}} 
	\\
	=& 
	- \left[ \Lambda_{2m - 2} \left[\left(\frac{d}{ds} \red{\Y_3} \sum_{p, q} \red{\Y_2} p q \zeta_p \lambda^{p - 1} \zeta_q \lambda^{q - 1} \right) \frac{d}{ds} L_j \Y_1 ds\right] \right]_{\text{dR}},
\end{align*}
where the latter contains at least one of $\dot{\Y}_2, \dot{\zeta}_p$ or $\dot{\zeta}_q$ in the leftmost factor and the rightmost is again $(2m-2, 2)$-harmless. By Theorem \ref{Linear coeff} \red{and the fact that differential degree is invariant under ``differentiation by parts'' (cf. Section \ref{subsec: cohomological considerations and curvature polynomials})}, this amounts to \red{a $(2m-2,3)$-harmless term} which can be absorbed into the remainder \red{in Theorem \ref{main}.}
\\
\\
\textbf{Type 5 words.} Any composition with at least two factors of $L_4, L_5, L_6$ or $L_7$ \red{has at least two differentiated factors of $\ka$ in any monomial of maximal differential degree and is hence, clearly $(2m-2,4)$-harmless. These terms} can again be discarded into the remainder.
\\
\\
\textbf{Type 6 words.} \red{We now consider type-$6$ words, which generate} terms of the form $\zeta^{(2k+1)/2} L_8^k \z_1$. There are two subtypes: (1) either all $\zeta$-indices in the expansion of $\zeta^{(2k + 1)/2}$ are equal to one or (2) there are some indices strictly greater than one. In the first case, \red{its contribution to $\Upsilon_{2m - 1}$ is of the form}
\begin{align*}
	\sum_{k = 1}^{m-1} \frac{1}{(2k + 1)!}\zeta_1^{k + 1/2} \Lambda_{2m - 2k - 2} \left[\left(\frac{\d \zeta}{\d \lambda}\right)^2 \frac{\d^2}{\d s^2}  L_8^{k-1} \z_1\right].
\end{align*}
In the second, all derivatives in $L_{8,1,1}$ must land on $\z_1$:
\begin{align*}
	\sum_{k = 1}^{m-2} \sum_{q = 1}^{m - k -1} \frac{1}{(2k + 1)!} \frac{2k + 1}{2} (m- q - k) \zeta_1^{3k - 1/2} \zeta_{q + 1}   \zeta_{m - q - k}^{(2k)}.
\end{align*}
\red{Both types will be strategically integrated by parts in Section \ref{subsec: integration by parts} below.} Inserting the formulas from Cases $1, 3$, and $6$ into $\zeta_m$ in Proposition \ref{intinv} gives
\begin{align}
	\label{eq:AMUpsilonM}
	\begin{split}
		m!& \zeta_1^{-m-1}  \zeta_{m}
		=  \frac{2 \, m!}{2m + 1} \zeta_1^{-1/2} \zeta_1^{-m - 1} \left(A_{2m- 1} - \Upsilon_{2m - 1}\right)
		\\
		&= \frac{2 \, m!}{2m + 1} \zeta_1^{-m - 3/2} \bigg(A_{2m- 1} - \underbrace{\zeta_1^{-1/2} \sum_{i = 1}^{m - 2} \left(\frac{(m-i)}{2} - \frac{1}{8}\right) \zeta_{i + 1} \zeta_{m - i}}_{k = 0}
		\\
		& - \underbrace{\sum_{k = 1}^{m-2} \sum_{q = 1}^{m - k -1} \frac{1}{(2k + 1)!} \frac{2k + 1}{2}  \zeta_1^{3k - 1/2} \zeta_{q + 1}  (m- q - k) \zeta_{m - q - k}^{(2k)} }_{\text{nontrivial index in}\,\, \zeta^{(2k+1)/2} \,\, \text{preceding}\,\, L_8^k \z_1}
		\\
		& -\underbrace{ \sum_{k = 1}^{m-1} \frac{1}{(2k + 1)!}\zeta_1^{k + 1/2} \Lambda_{2m - 2k - 2} \left[\left(\frac{\d \zeta}{\d \lambda}\right)^2 \frac{\d^2}{\d s^2}  L_8^{k-1} \z_1\right]}_{\zeta_1^{(2k+1)/2}L_8^k \z_1}
		\\
		& + \underbrace{\sum_{k = 1}^{m - 1}\sum_{p =1}^{m- k} \frac{1}{(2k+ 1)!}\dot{\zeta}_p \zeta_1^{3k - 1/2}  (m + 1 - k - p)(m - k - p) \frac{d^{2k - 1} \zeta_{m-k - p + 1}}{d s^{2k - 1}}}_{L_4 L_8^{k-1} \z_1}
		\\
		&+ \underbrace{\sum_{k = 1}^{m - 1} \sum_{p =1}^{m- k}  \frac{1}{(2k+ 1)!}\zeta_1^{3k - 1/2} \ddot{\zeta}_p (m + 1 - k - p)(m - k - p) \frac{d^{2k - 2} \zeta_{m-k - p + 1}}{d s^{2k - 2}}}_{L_5 L_8^{k-1} \z_1}\\
		&- \underbrace{\sum_{k = 1}^{m-1} \sum_{r = 0}^{m - k - 1} \frac{1}{(2k + 1)!} (m - k - r) (2r + k - m + 2) \zeta_1^{3k - 1/2}\dot{\zeta}_{r + 1} \frac{\d^{2k - 1} \zeta_{m - k - r}}{d s^{2k - 1}}}_{\zeta^{(2k+1)/2}(L_6 + L_7) L_8^{k-1} \z_1}
		\bigg)
		\\ &+  \red{\Hm},
	\end{split}
\end{align}
\red{where $\Hm$ denotes $(2m-2,3)$-harmless terms} in the sense of Definition \ref{def:harmless}. The two additional terms in \red{$\Theta_m$ which come from} Proposition \ref{intinv} are
\begin{align*}
	(m-1)! \zeta_1^{-m-2} \sum_{j =1}^{m-2} (j + 1) (m - j)\zeta_{j+1} \zeta_{m - j}
\end{align*}
and
\begin{align*}
	 \sum_{\ell = 1}^{m-2}\sum_{i = 0}^{\ell - 1} \zeta_1^{-m-2}  \frac{(m - 1 - \ell +i )! \ell!}{i!} (m-\ell) (\ell+1) \zeta_{m - \ell} \zeta_{\ell + 1}.
\end{align*}

\red{By Corollary \ref{cor:harmlessremainders}, after substituting the curvature jet expressions for $\zeta_i$ from Theorem \ref{Linear coeff}, every monomial with $\deg_{\d, \zeta} + w_\zeta \leq 2m - 2$ has $\ka$-differential degree at most $2m-2$; if equality holds and the monomial contains at least three nontrivial $\zeta$-indices, then it is $(2m-2,3)$-harmless. Since $\zeta_1$ contains only an undifferentiated power of $\kappa$, multiplication by a power of $\zeta_1$ affects neither the differential degree nor the harmlessness of any terms.} We are now prepared to \red{begin} proving Theorem \ref{main}.

\subsection{Integration by parts}
\label{subsec: integration by parts}
Let us integrate each term above separately, \red{following the algorithm described in Section \ref{subsec: cohomological considerations and curvature polynomials}, to reduce each of the linear terms in Section \ref{sec: linear terms with maximal derivatives} and second order terms in Section \ref{sec: quadratic terms of maximal differential degree} to an integral in which $\ka_{m-1}$ appears \textit{quadratically}, as in the statement of Theorem \ref{main}}. Define the integrals
\begin{align*}
	J_1 = &    \int_0^\ell \zeta_1^{-m -3/2} A_{2m - 1}^\text{linear} ds = - \frac{1}{(2m)!} \int_0^\ell  2^{2m - 2} \ka^{-4m/3 + 1} \ka_{2m - 2} ds,\\
	J_2 = &  \int_0^\ell \zeta_1^{-m -3/2} A_{2m - 1}^{\red{(2)}} ds
	= 2^{2m - 1} \sum_{p = 2}^{2m - 2} \int_0^\ell \ka^{-4m/3} p \frac{\ka_{p-1} \ka_{2m - p - 1} }{(p+1)! (2m - p + 1)!} ds,\\
	J_3 = & \sum_{i = 1}^{m-2} \int_0^\ell \zeta_1^{-m -2}  \left(\frac{(m-i)}{2} - \frac{1}{8}\right) \zeta_{i+1}\zeta_{m-i} ds,\\
	 J_4 = &   \sum_{k = 1}^{m-2} \sum_{q = 1}^{m - k -1} \frac{1}{(2k + 1)!} \frac{2k + 1}{2} \int_0^\ell \zeta_1^{3k - m - 2} \zeta_{q + 1}  (m- q - k) \zeta_{m - q - k}^{(2k)} ds,\\
	 J_5 = &    \sum_{k = 1}^{m-1} \int_0^\ell \frac{1}{(2k + 1)!} \Lambda_{2m - 2k - 2}\left[\zeta_1^{k - m - 1}  \left(\frac{\d \zeta}{\d \lambda}\right)^2 \frac{\d^2}{\d s^2}  L_8^{k-1} \z_1 \right] ds,\\
 J_6 = &  \sum_{k = 1}^{m - 1}\sum_{p =1}^{m- k} \frac{(m + 1 - k - p)(m - k - p)}{(2k+ 1)!} \int_0^\ell \dot{\zeta}_p  \zeta_1^{3k - m - 2}  \frac{d^{2k - 1} \zeta_{m- k - p + 1}}{d s^{2k - 1}} ds,\\
J_7 =& \sum_{k = 1}^{m - 1} \sum_{p =1}^{m- k}  \frac{(m + 1 - k - p)(m - k - p)}{(2k+ 1)!} \int_0^\ell \ddot{\zeta}_p \zeta_1^{3k - m - 2} \frac{d^{2k - 2} \zeta_{m-k - p + 1}}{d s^{2k - 2}} ds,\\
	 J_8 = & \sum_{k = 1}^{m-1} \sum_{r = 0}^{m - k - 1} \frac{(m - k - r) (2r - m + k + 2)}{(2k + 1)!}\\
	 & \times \int_0^\ell  \zeta_1^{3k - m - 2}\dot{\zeta}_{r + 1} \frac{\d^{2k - 1} \zeta_{m - k - r}}{d s^{2k - 1}}ds,\\
	 J_9 =&  (m-1)! \sum_{r =1}^{m-2} (r + 1) (m - r) \int_0^\ell \zeta_1^{-m-2} \zeta_{r+1} \zeta_{m - r}ds,\\
	 J_{10} = & \sum_{r = 1}^{m-2}\sum_{i = 0}^{r - 1} \frac{(m - 1 - r +i )! r!}{i!} (m-r) (r+1) \int_0^\ell \zeta_1^{-m-2}  \zeta_{m - r} \zeta_{r + 1}ds.
\end{align*}
\red{The first two integrals $J_1$ and $J_2$ come from the term $A_{2m-1}$ in Proposition \ref{AM} while the subsequent six integrals $J_3, \cdots, J_8$ come from \eqref{eq:AMUpsilonM}. The final two, $J_9$ and $J_{10}$, come from} Proposition \ref{intinv}, which then reads
\begin{align*}
	\I_m =& \int_0^\ell \Theta_m[\ka] ds\\
	=& \frac{2 \,  m!}{2m + 1}(J_1 + J_2 - J_3 - J_4 - J_5 + J_6 + J_7 - J_8) -J_9 - J_{10} 
	\\
	&+ \int_0^\ell \mathcal{H}_{2m- 2, 3} ds,
\end{align*}
\red{where $\int_0^\ell \mathcal{H}_{2m - 2, 3}$ now denotes \textit{integrals} of $(2m-2,3)$-harmless terms in the sense of Definition \ref{def:harmless}; by Corollary \ref{cor:harmlessremainders}, they can be absorbed into the remainder $\mathcal{R}_m$ in Theorem \ref{main}.} In each integral, we can integrate by parts keeping only the top order terms. Let $a \pm b$ be even, as will be the case for the integrals $J_i$. Then,
\begin{align*}
	\int_0^\ell \zeta_1^p \frac{d^a \zeta_i}{d s^a} \frac{d^b \zeta_j}{d s^b} ds = (-1)^{i - j + \frac{a - b}{2}} \int_0^\ell 2^p \ka^{-2p/3} F_i F_j \ka_{i + j + \frac{a + b}{2} - 2}^2 ds + \int_0^\ell \mathcal{H}_{(2i+2j + a + b - 4,3)} ds.
\end{align*}
\red{Recall our convention \eqref{eq:msfzk}, where we denote by $\msfz_m = \zr(2m)$ the Riemann $\zeta$-function at even integers $2m$. The integrals $J_i$ can be simplified as follows:}
\\
\\
\textbf{Integral $J_1$.}
\begin{align*}
	J_1 = (-1)^m \frac{ \left(- \frac{4m}{3} + 1\right) 2^{2m - 2}}{(2m)!} \int_0^\ell \ka^{-4m/3} \ka_{m-1}^2 ds + \int_0^\ell \mathcal{H}_{2m-2,3} ds.
\end{align*}
\\
\textbf{Integral $J_2$.}
\begin{align*}
	J_2 =& 2^{2m - 1} \sum_{p = 2}^{2m - 2} \int_0^\ell \ka^{-4m/3} p \frac{\ka_{p-1} \ka_{2m - p - 1} }{(p+1)! (2m - p + 1)!} ds\\
	=& 2^{2m - 1} \sum_{p = 2}^{2m - 2} (-1)^{m-p}   \frac{p}{(p+1)! (2m - p + 1)!} \int_0^\ell \ka^{-4m/3} \ka_{m-1}^2 ds + \int_0^\ell \Hm ds.
\end{align*}
\\
\textbf{Integral $J_3$.} By Theorem \ref{Linear coeff}, the integrand of $J_3$ contains $F_{i+1} F_{m - i} \ka_{2i} \ka_{2m - 2i - 2}$. At the expense of lower order terms, we can integrate by parts $|(2m - 2i - 2 - 2i)| / 2 = |m -2i - 1|$ times to obtain
\begin{align*}
	J_3 =& \sum_{r = 1}^{m-2} \frac{4m - 4r - 1}{8} \frac{2^{2m + 5} }{(2 \pi)^{2m + 2}} \msfz_{r+1} \msfz_{m - r} \int_0^\ell \ka^{-4m / 3} \ka_{m-1}^2 ds + \int_0^\ell \Hm ds.
\end{align*}
\\
\textbf{Integrals $J_6$ and $J_7$.}
Notice that \red{modulo $(2m-2, 3)$ harmless terms,} the integrands of $J_6$ and $J_7$ \red{cancel with one another} as cohomology classes:
\red{
	\begin{align*}
		J_6 = &  \sum_{k = 1}^{m - 1}\sum_{p =1}^{m- k} \frac{(m + 1 - k - p)(m - k - p)}{(2k+ 1)!} \int_0^\ell \dot{\zeta}_p  \zeta_1^{3k - m - 2}  \frac{d^{2k - 1} \zeta_{m-k - p+1}}{d s^{2k - 1}} ds
		\\
		= &
		- \sum_{k = 1}^{m - 1}\sum_{p =1}^{m- k} \frac{(m + 1 - k - p)(m - k - p)}{(2k+ 1)!} \int_0^\ell \frac{d}{d s} \left(\dot{\zeta}_p  \zeta_1^{3k - m - 2} \right) \frac{d^{2k - 2} \zeta_{m-k - p+1}}{d s^{2k - 2}} ds
		\\
		=&
		- \sum_{k = 1}^{m - 1} \sum_{p =1}^{m- k}  \frac{(m + 1 - k - p)(m - k - p)}{(2k+ 1)!} \int_0^\ell \ddot{\zeta}_p \zeta_1^{3k - m - 2} \frac{d^{2k - 2} \zeta_{m-k - p+1}}{d s^{2k - 2}} ds + \int_0^\ell  \Hm ds
		\\
		= &
		- J_7 + \int_0^\ell \Hm ds.
	\end{align*}
}
\red{Here, we have integrated by parts,} moving one derivative off of ${\zeta}_q^{2k - 1}$ in $L_4 L_8^{k-1} \z_1$ and putting it onto $\dot{\zeta}_p$. \red{If the derivative were to land on $\zeta_1^{3k - m - 2}$ instead, the differential degree would remain unchanged and the resulting term would be $(2m - 2,3)$-harmless.} Thus,
$$
J_6 + J_7 = \int_0^\ell \Hm ds,
$$
and these terms can be absorbed into the remainder in Theorem \ref{main}
\\
\\
\textbf{Integrals $J_4, J_5$ and $J_8$.}
$J_4$ can be integrated by parts to be put in the form of $J_8$. \red{Adding and subtracting the term $q = 0$ and changing $q$ to $r$ to match the indices of summation with those in $J_8$, we obtain}
\begin{align*}
	J_4 =& - \sum_{k = 1}^{m-1} \sum_{r = 0}^{m - k - 1} \frac{1}{(2k + 1)!} \left(\frac{2k + 1}{2}\right) (m - k - r) \int_0^\ell \zeta_1^{3k - m - 2} \dot{\zeta}_{r+1} \frac{d^{2k - 1} \zeta_{m - k - r}}{d s^{2k - 1}} ds\\
	&+   \sum_{k = 1}^{m-1} \frac{1}{(2k + 1)!} \left(\frac{2k + 1}{2}\right) (m - k) \int_0^\ell \zeta_1^{3k - m - 2} \dot{\zeta}_{1} \frac{d^{2k - 1} \zeta_{m - k}}{d s^{2k - 1}} ds + \int_0^\ell  \Hm ds,
\end{align*}
\red{where the terms in which $\frac{d}{d s}$ lands on $\zeta_1^{3k-m-2}$ when integrating by parts are absorbed in the remainder.} Integrating by parts once in $J_5$, we obtain
\begin{align*}
	J_5 =& - \sum_{k = 1}^{m-1} \frac{(k - m - 1)(m-k)}{(2k+1)!} \int_0^\ell \zeta_1^{3k - m - 2}  \dot{\zeta_1} \frac{d^{2k - 1} \zeta_{m-k}}{d s^{2k- 1}} ds\\
	&- \sum_{k = 1}^{m-1} \sum_{r = 0}^{m - k  - 1} \frac{2 (r+1) (m- k - r)}{(2k + 1)!} \int_0^\ell \zeta_1^{3k - m - 2} \dot{\zeta}_{r+1} \frac{d^{2k - 1} \zeta_{m - k - r}}{d s^{2k - 1}} ds + \int_0^\ell  \Hm ds.
\end{align*}
Combining with $J_8$ and simplifying, we get
\begin{align*}
	J _4 +J_5 + J_8 =&   \sum_{k = 1}^{m-1} \frac{(m+3/2)(m-k)}{(2k+1)!} \int_0^\ell \zeta_1^{3k - m - 2}  \dot{\zeta_1} \frac{d^{2k - 1} \zeta_{m-k}}{d s^{2k- 1}} ds\\
	&- \sum_{k= 1}^{m- 1} \sum_{r = 0}^{m- k - 1} \frac{(m-k - r) (m + 1/2)}{(2k+1)!} \int_0^\ell \zeta_1^{3k  - m - 2} \dot{\zeta}_{r+1} \frac{d^{2k - 1} \zeta_{m - k - r}}{d s^{2k -1}} ds
	\\
	& + \int_0^\ell  \Hm ds.
\end{align*}
In terms of curvature, we have
\begin{align*}
	J_4 + J_5 + J_8 =&  \sum_{k = 1}^{m-1} (-1)^{k+1} \frac{(m+3/2)(m-k)}{(2k+1)!}  \frac{2^{2m + 2}}{3(2\pi)^{2m - 2k}}  \msfz_{m-k} \int_0^\ell \ka^{-4m/3} \ka_{m-1}^2 ds\\
	& - \sum_{k= 1}^{m- 1} \sum_{r = 0}^{m- k - 1} (-1)^{k + 1} \frac{(m-k - r) (m + 1/2)}{(2k+1)!} \frac{2^{2m + 5}}{(2\pi)^{2(m-k + 1)}}\\
	& \qquad \times \msfz_{r+1} \msfz_{m - k - r} \int_0^\ell \ka^{-4m/3} \ka_{m-1}^2 ds + \int_0^\ell  \Hm ds.
\end{align*}
\textbf{Integrals $J_9$ and $J_{10}$.} Integrating by parts modulo remainders as above, we find that
\begin{align*}
	J_9 =& (m-1)! \sum_{r=1}^{m-2} (r+1) (m-r) \frac{2^{2m+5}}{(2\pi)^{2m + 2}}
	\msfz_{m-r} \msfz_{r+1} \int_0^\ell \ka^{-4m/3} \ka_{m-1}^2 ds + \int_0^\ell  \Hm ds
	\\
	J_{10} =& \sum_{r = 1}^{m-2} \sum_{i = 0}^{r-1} \frac{(m - 1 - r + i) ! r!}{i!} (m-r)(r+1) \frac{2^{2m+5}}{(2\pi)^{2m + 2}}
	\msfz_{m-r} \msfz_{r+1}  \int_0^\ell \ka^{-4m/3} \ka_{m-1}^2 ds
	\\
	& + \int_0^\ell  \Hm ds.
\end{align*}

\subsection{The leading order coefficient}\label{subsec: nonvanishing of the leading order coefficient}

\noindent Note that in each $J_i$ above, the integrand is simplified to $\ka^{-4m/3} \ka_{m-1}^2$, which matches perfectly with the structure of $\I_1, \I_2, \I_3$ and $\I_4$ observed in \cite{Sorr15}. In each of the integrals above, we can now factor out the term
\begin{align*}
	J_0 = \int_0^\ell \ka^{-4m/3} \ka_{m-1}^2 ds.
\end{align*}
\red{We aim to show that the leading order coefficient $c_m$ in Theorem \ref{main} is nonvanishing. To do so, we will} separate the coefficients of $J_i$ ($1 \leq i \leq 10$) into single and double sums, \red{which we denote by $S(m)$ and $D(m)$ respectively}. Those coming from $J_1, J_2, -J_3,$ the first sum in $-J_4 - J_5 - J_8$ and $-J_9$ contribute
\begin{align*}
	S(m) =& \frac{2 \, m!}{2m+1} \Bigg( (-1)^m \frac{ \left(- \frac{4m}{3} + 1\right) 2^{2m - 2}}{(2m)!}
	\\
	& + 
	\sum_{p = 2}^{2m - 2} (-1)^{m-p}  2^{2m - 1} \frac{p}{(p+1)! (2m - p + 1)!}
	\\
	&-
	\sum_{r = 1}^{m-2} \frac{4m - 4r - 1}{8} \frac{2^{2m + 5} }{(2 \pi)^{2m + 2}} \msfz_{r+1} \msfz_{m-r}
	\\
	&-
	\sum_{k = 1}^{m-1} (-1)^{k+1} \frac{(m+3/2)(m-k)}{(2k+1)!}  \frac{2^{2m + 2}}{3(2\pi)^{2m - 2k}} \msfz_{m-k} \Bigg)
	\\
	&-
	(m-1)! \sum_{r=1}^{m-2} (r+1) (m-r) \frac{2^{2m+5}}{(2\pi)^{2m + 2}} \msfz_{m-r} \msfz_{r+1}.
\end{align*} 
Those coming from the second \red{sum in} $-J_4 - J_5 - J_8$ and $-J_{10}$ contribute
\begin{align*}
	D(m) =&  \frac{2 \, m!}{2m+1} \sum_{k= 1}^{m- 1} \sum_{r = 0}^{m- k - 1} (-1)^{k+1} \frac{(m-k - r) (m + 1/2)}{(2k+1)!} \frac{2^{2m + 5}}{(2\pi)^{2(m-k + 1)}} \msfz_{r+1} \msfz_{m - k - r}
	\\*
	&- \sum_{r = 1}^{m-2} \sum_{i = 0}^{r-1} \frac{(m - 1 - r + i) ! r!}{i!} (m-r)(r+1) \frac{2^{2m+5}}{(2\pi)^{2m + 2}} \msfz_{m-r} \msfz_{r+1} .
\end{align*}
One can easily check that $S(2) + D(2) = \frac{2}{135} = \frac{8}{540}$, which corroborates the formula for $\I_2$ in \cite{MM}; see \eqref{eq:I2} above. \red{In fact, a comparison of the coefficients in Proposition \ref{prop: Kovachev Popov coordinates and caustic length lazutkin expansion} with those on page 22 of \cite{Sorr15} shows that the coefficients $\iota_k$ in \eqref{SIK} are given by

\begin{align*}
	\iota_1 =& - \frac{1}{2},
	\qquad
	\iota_2 = \frac{1}{540},
	\qquad
	\iota_3 = \frac{1}{4200},
	\qquad
	\iota_4 = \frac{16}{135}.
\end{align*}

Taking the coefficient of the quadratic highest derivative in each $\I_k$ in \eqref{SIK}, we see that the leading order coefficients $c_m$ (in the notation of Theorem \ref{main}) are given by
\begin{align*}
	c_2 = 8 \iota_2 = \frac{8}{540},
	\qquad
	c_3 = 24 \iota_3 = \frac{24}{4200},
	\qquad
	c_4 = \frac{1}{42} \iota_4 = \frac{8}{2835}.
\end{align*}
\red{A direct substitution also shows that these coefficients agree with the formula
\begin{align} \label{eq:iotavalues}
	c_m = \frac{(2m +3)}{\pi^{2m + 2}} m! \msfz_{m+1}
\end{align}
in Theorem \ref{main}, which we now derive.}}
\\
\\
\red{Let us first simplify a few of the individual terms in $S(m) + D(m)$. First note that the inner sum over $i$ in the second double sum of $D(m)$ can be evaluated explicitly via the hockey-stick identity. Set $n = m - 1 - r$. We then have
\begin{align*}
	\sum_{i = 0}^{r - 1} \frac{(m - 1 - r + i)! r! }{i!} =& \sum_{i = 0}^{r - 1} \frac{(n + i)! r!}{i!}
	\\
	=&
	n! r! \sum_{i = 0}^{r - 1} \binom{n+i}{i}
	\\
	=&
	n! r! \binom{n + r}{r - 1}
	\\
	=&
	(m - r - 1)! r!  \binom{m  - 1}{ r- 1}
	\\
	=& \frac{r}{m - r} (m - 1)!
\end{align*}
Its contribution to the second sum in $D(m)$ is then
\begin{align*}
	- (m - 1)!  \frac{2^{2m +5}}{(2\pi)^{2m + 2}} \sum_{r = 1}^{m-2} r ( r+1) \msfz_{m-r} \msfz_{r+1}.
\end{align*}
This combines with the fifth term in $S(m)$ to give
\begin{align*}
	- m !  \frac{2^{2m +5}}{(2\pi)^{2m + 2}} \sum_{r = 1}^{m-2} (r + 1) \msfz_{m-r} \msfz_{r+1}.
\end{align*}
}

\red{
Let us now pull out a common factor from each of the terms above:
\begin{align}\label{eq:SplusD}
	S(m) + D(m) = &
	\frac{2^{2m +6} \, m!}{(2\pi)^{2m + 2} (2m+1)} \left(D_1(m) + \sum_{i = 1}^5 S_i(m) \right),
\end{align}
where
\begin{align*}
	S_1(m)
	=&
	(-1)^m (2\pi)^{2m + 2} \frac{ \left(- \frac{4m}{3} + 1\right)}{2^7 (2m)!}
	\\
	S_2(m) =& (2\pi)^{2m + 2} \sum_{p = 2}^{2m - 2} (-1)^{m-p}  \frac{p}{2^6 (p+1)! (2m - p + 1)!}
	\\
	S_3(m) = &- \sum_{r = 1}^{m-2} \frac{4m - 4r - 1}{8} \msfz_{r+1} \msfz_{m-r}
	\\
	S_4(m) =& \frac{(2\pi)^2}{24} \left(m + \frac{3}{2}\right) \sum_{k = 1}^{m-1} (-1)^{k} \frac{(m-k)}{(2k+1)!}  (2\pi)^{2k} \msfz_{m-k}
	\\
	S_5(m) = & - \frac{2m + 1}{2} \sum_{r = 1}^{m-2} (r + 1) \msfz_{m-r} \msfz_{r+1}
	\\
	D_1 = & (m + 1/2) \sum_{k= 1}^{m- 1} \sum_{r = 0}^{m- k - 1} (-1)^{k+1} \frac{(m-k - r) }{(2k+1)!} (2\pi)^{2k} \msfz_{r+1} \msfz_{m-k - r}.
\end{align*}
}
Our goal is to find a formula for the leading order coefficients $c_m$, $m \in \N$ in Theorem \ref{main}, which arise from the sum $S(m) + D(m)$. To evaluate certain terms in \eqref{eq:SplusD} explicitly, we will make use of convolution formulas from analytic number theory. \red{We present a convenient} formula for weighted sums of Bernoulli numbers \red{(equivalently, the Riemann $\zeta$-function)}.
\begin{lemm}\label{zetaidentity}
	For any $n \geq 2$.
	\begin{align*}
		\sum_{r = 1}^{n-1} r \msfz_{n-r} \msfz_r  = \left(\frac{n}{2}\right)\left(n+ \frac{1}{2} \right) \msfz_n.
	\end{align*}
\end{lemm}

\begin{proof}
	The following well known formula, originally proved for products of binomial coefficients and pairs of Bernoulli numbers, is due to Euler:
	\begin{align}\label{Euler}
		\sum_{r = 1}^{n-1} \msfz_{n-r} \msfz_r  = \left(n + \frac{1}{2}\right) \msfz_n,
	\end{align}
	\red{(see \cite{WilliamsZeta}, Theorem 1; also \cite{ApostolZeta}).} We follow the algorithm described in \cite{DavisSitaramachandra} for computing the first moment of such sequences via generalized convolutions. Let $f,g: 2 \Z \to \C$ be sequences on the even integers and define the convolution operator $*_{n}$ by
	\begin{align*}
		f \ast_{n} g(2n) = \sum_{r = 1}^{n-1} f(2n - 2r) g(2r).
	\end{align*}
	For each sequence $f: 2 \Z \to \C$, let $\overline{f}(2n) = (2n - 1) f(2n)$. It follows that
	\begin{align}
		\overline{f \ast_{n} g}(2n) = \overline{f} \ast_n g (2n) + f \ast_n \overline{g}(2n) + f \ast_n g(2n).
	\end{align}
	Formula \eqref{Euler} can then be written as $\zr \ast_n \zr (2n) = (n + \frac{1}{2}) \zr(2n)$. Choosing $g = \overline{f}$ gives $f \ast_n \overline{f}(2n) = (n - 1) f \ast_n f(2n)$ and setting $f(2r) = \zr(2r)$, we have
	\begin{align*}
		\sum_{r =  1}^{n-1} r \msfz_r \msfz_{n-r}  =& \frac{1}{2} \left(\zr \ast_n \overline{\zr} + \zr \ast_n \zr \right)\\
		=& \frac{1}{2} \left((n-1) \zr \ast_n \zr + \zr \ast_n \zr\right)\\
		=& \left(\frac{n}{2}\right) \left(n+ \frac{1}{2}\right) \msfz_n,
	\end{align*}
	which proves the lemma.
\end{proof}

\red{Applying Lemma \ref{zetaidentity} to the individual terms in $S(m)$ and $D(m)$ and adding and subtracting boundary terms as necessary, we obtain the following simplified expressions for $S_3(m), S_5(m)$ and $D_1(m)$:}

\begin{coro}\label{SThree}
	For $m \geq 2$, we have
	\begin{align*}
		S_3(m)
		= &
		- \left(\frac{4m^2 + 8 m + 3}{16}\right) \msfz_{m + 1} + \left(\frac{2m + 1}{4}\right) \msfz_1 \msfz_m.
	\end{align*}
	and
	\begin{align*}
		S_5(m) =& \red{- \frac{(2m + 1)(m+1) (2m+3)}{8} \msfz_{m + 1}}
		+
		\red{\frac{(2m + 1) (m+1)}{2} \msfz_1 \msfz_m.}
	\end{align*}
	In addition, for $m \geq 2$, the inner sum in $D_1$ simplifies to give
	\begin{align*}
		D_1(m) =& (m+1/2) \sum_{k = 1}^{m-1} \frac{(-1)^{k+1}}{2}\frac{(2\pi)^{2k}}{(2k + 1)!}\left(m - k + 1\right)
		\\
		& \times \left(m - k + 3/2\right) \msfz_{m-k + 1}.
	\end{align*}
\end{coro}

\begin{lemm}
	\red{For $m \geq 2$, $S_2(m)$ is given by
	\begin{align*}
		S_2(m) = (-1)^m\frac{(2\pi)^{2m+2}}{2^{6}} \frac{4m^3+2m^2}{(2m+2)!}.
		\end{align*}
	}
\end{lemm}

\begin{proof}
	We assemble the coefficients into a generating function. $S_2$ becomes the $2m + 2$-coefficient of
	\begin{align*}
		S_{2}(x; m) :=& (-1)^m \frac{(2\pi)^{2m + 2}}{2^6} \left(1 + x\right)
		\\
		=&
		(-1)^m \frac{(2\pi)^{2m + 2}}{2^6}  \left(x e^{-x} + e^{-x}\right) e^x
		\\
		=&
		(-1)^m \frac{(2\pi)^{2m + 2}}{2^6} \left(\sum_{q = 0}^{\infty} (-1)^{q+1} \frac{ q-1}{q!} x^q \right)\left(\sum_{r = 0}^{\infty} \frac{x^r}{r!} \right)
		\\
		=&
		\red{(-1)^m \frac{(2\pi)^{2m + 2}}{2^6} \sum_{n = 0}^\infty \left(\sum_{p = -1}^{n-1} (-1)^{p} \frac{p}{(p+1)!(n - p - 1)!} \right) x^n}.
	\end{align*}
	
	\red{On the one hand, since $S_2(x;m)$ is affine as a function of $x$, the coefficient of $x^{2m + 2}$ in $S_2(x;m)$ is $0$ for $m \geq 0$. On the other hand, we see that it can be written in terms of $S_2(m)$:
	\begin{align*}
		0 =& (-1)^m \frac{(2\pi)^{2m + 2}}{2^6} \sum_{p = -1}^{2m+1} (-1)^{p} \frac{p}{(p+1)! (2m - p +1 )!}
		\\
		=& S_2(m) + (-1)^m \frac{(2\pi)^{2m + 2}}{2^6} \left(\sum_{p = -1}^1 + \sum_{p = 2m - 1}^{2m+1} \right) (-1)^{p} \frac{p}{(p+1)! (2m - p +1 )!} 
		\\
		=&S_2(m) + (-1)^m \frac{(2\pi)^{2m + 2}}{2^6} \left( \frac{1}{(2m+2)!} - \frac{1}{2! (2m)!} - \frac{2m - 1}{(2m)! 2!} + \frac{2m}{(2m+1)! } - \frac{2m+1}{(2m + 2)!}\right)
		\\
		=& S_2(m) + (-1)^m \frac{(2\pi)^{2m + 2}}{2^6} \frac{1}{(2m + 2)!} \left(-2m - \frac{2m ( 2m + 1) (2m+2)}{2} + 2m (2m+2)\right)
		\\
		=& S_2(m) + (-1)^m \frac{(2\pi)^{2m + 2}}{2^6} \frac{(-4m^3 - 2m^2)}{(2m+2)!},
	\end{align*}
	from which the lemma follows.
	}
\end{proof}

We now move on to evaluate the terms $S_1(m), S_4(m)$ and $D_1(m)$. To do so, we need another combinatorial formula involving even Riemann $\zeta$-coefficients. As in the proof of Theorem \ref{Linear coeff}, we will use generating functions. Let us introduce the notation

\red{
\begin{lemm}\label{lem:zetafactorial}
	Set
	\begin{align*}
		a_k = (-1)^k \frac{(2\pi)^{2k}}{(2k + 1)!} \qquad \text{and} \qquad e_k = (-1)^{k+1} \frac{(2\pi)^{2 k}}{8\, (2k)!}.
	\end{align*}
	For $n \geq 2$,
	\begin{align*}
		\sum_{k = 0}^{n-1} (n-k) a_k  \msfz_{n -k} = - \frac{\msfz_n}{2} + e_n
	\end{align*}
	and
	\begin{align*}
		\sum_{k = 0}^{n-1} (n - k) \left(n - k + \half \right)  a_k \msfz_{n-k} = -\left(n+\half\right) \msfz_n + e_n.
	\end{align*}
\end{lemm}

\begin{proof}
	Let us once again assemble various coefficients into generating functions:
	\begin{align*}
		A(x) :=& \sum_{k = 0}^\infty a_k x^k = \frac{\sin(2\pi \sqrt{x})}{2 \pi \sqrt{x}},
		\\
		Z(x) :=& \sum_{k = 1}^\infty \msfz_k x^k = \frac{1 - \pi \sqrt{x} \cot(\pi \sqrt{x})}{2},
	\end{align*}
	where the last identity follows from the same formulas used in the proof of Theorem \ref{Linear coeff} \url{https://dlmf.nist.gov/4.19}. Define the differential operator $N = x \frac{d}{dx}$; for any polynomial $P$, the operator $P(N)$ acts on a generating function $\sum_k b_k x^k$ by multiplying its coefficients $b_k$ by $P(k)$ (see \cite{WilfGF}, Section 2.2). A direct computation
	\begin{align*}
		A(x) N Z(x) = -\half Z(x) + \frac{1}{4} \sin^2(\pi \sqrt{x})
	\end{align*}
	and
	\begin{align*}
		A(x) \left(N^2 + \half N \right) Z(x) = - Z(x)^2 + \frac{1}{4} \sin^2(\pi \sqrt{x}).
	\end{align*}
	
	The coefficient of $x^n$ in $\frac{1}{4} \sin^2(\pi \sqrt{x})$ is precisely $e_n$, from which the first formula in the lemma follows. For the second, we again extract the same coefficient from $\frac{1}{4} \sin^2(\pi \sqrt{x})$ and then apply the Euler convolution formula \eqref{Euler} for the $x^n$ coefficient of $Z^2(x)$.

\end{proof}
}

\red{We now apply Lemma \ref{lem:zetafactorial} to $S_4$ and $D_1$. Adding and subtracting the $k = 0$ term in $S_4$, we have
\begin{align*}
	S_4(m) =& \frac{(2\pi)^2}{24} \left(m + \frac{3}{2}\right) \sum_{k = 1}^{m-1} (m - k) a_k \msfz_{m-k}
	\\
	= & \frac{(2\pi)^2}{24} \left(m + \frac{3}{2} \right) \left(- \frac{\msfz_m}{2} + e_m - m a_0 \msfz_m \right).
\end{align*}

Similarly, using the simplification of the inner sum of $D_1$ in Corollary \ref{SThree} and adding and subtracting the boundary terms $k = 0$, $k = m$, we have
\begin{align*}
	D_1(m) =& - \frac{\left(m + \half \right)}{2} \sum_{k = 1}^{m-1} a_k \left(m - k + 1\right) \left(m - k + \frac{3}{2} \right) \msfz_{m - k + 1}
	\\
	=& - \frac{\left(m + \half \right)}{2} \sum_{k = 0}^{m} a_k \left( (m+1) - k \right) \left((m + 1) - k + \frac{1}{2} \right) \msfz_{(m+1) -k}
	\\
	&+
	\frac{\left(m + \half \right)}{2} \left(m + 1\right) \left(m+ \frac{3}{2}\right) a_0 \msfz_{m+1}
	+ \frac{3}{4} \left(m+ \half \right)a_m \msfz_1
	\\
	=& \frac{\left(m + \half \right)}{2} \left( \left(m + \frac{3}{2} \right) (m + 2) \msfz_{m+1} - e_{m + 1} +  \frac{3}{2} a_m \msfz_1 \right).
\end{align*}
Adding these together, we get
\begin{align*}
	S_4 (m ) + D_1(m)
	= &
	\frac{(2m + 1) (2m + 3) (m+2)}{8} \msfz_{m+1} - \frac{(2m+3) (2m+1)}{4} \msfz_1 \msfz_m
	\\
	&+ \left(\frac{2m +3}{2} \msfz_1 e_m - \frac{2m+1}{4} e_{m+1} + \frac{3}{8} (2m+1) a_m \msfz_1 \right)
	\\
	=&
	\frac{(2m + 1) (2m + 3) (m+2)}{8} \msfz_{m+1} - \frac{(2m+3) (2m+1)}{4} \msfz_1 \msfz_m
	\\
	&+
	(-1)^{m+1} \frac{(2\pi)^{2m+2}}{(2m+2)!} \left( \frac{(2m+1) (2m^2 - m + 3)}{192}\right),
\end{align*}
where the second equality follows from simplifying using the formula $\msfz_1 = \pi^2/6$.}
\\
\\
\red{We now compare this to the contribution of $S_1 (m) + S_2(m)$. We compute
\begin{align*}
	S_1(m) +S_2(m) =& (-1)^m \frac{(2\pi)^{2m +2}}{(2m + 2)!} \left( \frac{(-4m +3)}{3 \times 2^{7} }(2m + 1) (2m+2) + \frac{4m^3 +2 m^2}{2^6}\right)
	\\
	=&
	(-1)^m
	\frac{(2\pi)^{2m + 2}}{(2m+2)!} \left(\frac{4m^3 + 5m + 3}{192}\right),
\end{align*}
which cancels exactly with the third term of $S_4(m) + D_1(m)$.}
\\
\\
\red{Similarly, if we combine $S_3$ and $S_5$, we obtain yet another partial cancellation:
\begin{align*}
	S_3(m) + S_5(m) = &  \left(- \frac{m}{4} - \frac{m}{2} - \frac{3}{16} \right) \msfz_{m+1} + \left(\frac{m}{2} + \frac{1}{4}\right) \msfz_1 \msfz_m
	\\
	&-\left(\frac{2m + 1}{2}\right) \left(\frac{m+1}{2}\right) \left(m + \frac{3}{2} \right) \msfz_{m+1} + \left(\frac{2m+1}{2}\right) \left(m + 1\right) \msfz_1 \msfz_m
	\\
	= & - \frac{(2m + 1) (2m +3)^2}{16} \msfz_{m+1} + \frac{(2m + 1) (2m +3)}{4} \msfz_1 \msfz_m
\end{align*}
(the second term in $S_3(m) + S_5(m)$ cancels with the second term in $S_4(m) + D_1(m)$). Combining all of the miraculous cancellations above, we obtain the following simplification:
	\begin{align*}
		S_1(m) + S_2(m) + S_3(m) +  S_4(m) + S_5(m) + D_1(m) = \frac{(2m + 1) (2m +3)}{16} \msfz_{m + 1},
	\end{align*}
which, together with \eqref{eq:SplusD}, concludes the proof of Theorem \ref{main}.}

\section{Open questions and future work}

Several natural extensions remain. One may find other marked length spectral invariants by studying the Taylor expansion of $\I(t)$ or $\beta(\omega)$ near the Lazutkin parameter (resp. rotation number) of a caustic other than the boundary, as was pointed out to the author by Alfonso Sorrentino. It would also be interesting to extremize higher Marvizi-Melrose invariants, extending the extremal results for $\I_1$ and $\I_2$ in \cite{MM}. Finally, we note that the methods developed here, in particular those in Section \ref{Small lambda asymptotics}, are amenable to much more general settings. For example, they also apply to symplectic, projective and outer (dual) billiards. It would be interesting to study the mean minimal action coefficients and compactness of isospectral sets in those settings as well.

\section{Acknowledgements} The author {would like to thank} the Erwin Schr\"odinger Institute in Vienna for hosting him during the summer of 2023, {during which parts of this work were carried out. The author is grateful to Alfonso Sorrentino for helpful discussions about Mather's $\beta$-function and to} {Ilya Kachkovskiy for suggesting the use of generating functions in Section \ref{Integral invariants}. The author also thanks Corentin Fierobe, Hamid Hezari, Vadim Kaloshin, Illya Koval, Rafael Ramirez-Ros, Alfonso Sorrentino and Amie Wilkinson for helpful conversations and the anonymous referee for helpful comments.}

\appendix

\section{Index of notation}

\begin{longtable}{p{0.18\textwidth} p{0.46\textwidth} p{0.28\textwidth}}
	\textbf{Symbol} & \textbf{Description} & \textbf{Location}
	\\
	\hline
	\endfirsthead
	\vspace{0.05in} & \vspace{0.05in} & \vspace{0.05in}
	\\
	 $\Omega$ & Birkhoff billiard table & Sections \ref{sec: Main Results} and \ref{Billiards}
	 \\
	 $\ka, \ka_i$ & Curvature and its $i^{\text{th}}$ derivative with respect to arclength & Theorem \ref{main}
	 \\
	 $\mathcal{B}$ & Space of all Birkhoff billiard tables & Section \ref{subsec:Length spectra}
	 \\
	 $B^* \d \Omega$ & Coball bundle of the boundary & Section \ref{Billiards}
	 \\
	 $S_{\d \Omega}^* \R^2$ & Cosphere bundle of $\R^2$ with footpoints on $\d \Omega$ & Section \ref{Billiards}
	 \\
	 $(x, \xi)$ & Point and covector on $\d \Omega$ & Section \ref{Billiards}
	 \\
	 $\widehat{\xi}_\pm$& Covectors in $S_{\d \Omega}^* \R^2$ & Section \ref{Billiards}
	 \\
	 $\phi$ & Angle of incidence of a billiard ray at the boundary & Section \ref{Billiards}
	 \\
	 $(s,\sigma)$ & Symplectic coordinates on $B^* \d \Omega$, $\sigma = \cos \phi$ & Section \ref{Billiards}
	 \\
	 $\lambda$ & Glancing coordinate $1 - \sigma$ & Section \ref{subsec: Hamiltonian formulation}
	 \\
	 $\dt_\pm$ & Forwards and backwards billiard maps & Definition \ref{def:dtpm billiard maps}
	 \\
	 $(s^{\pm}, \sigma^\pm)$ & Image of $(s,\sigma)$ under $\dt_\pm$ & Section \ref{Billiards}
	 \\
	 $\omega$ & Rotation number of a billiard orbit & Equation \ref{eq: general rotation number}
	 \\
	 type-$(p,q)$ & Winding number and period of a periodic billiard orbit & Section \ref{Billiards}
	 \\
	 $\beta$ & Mather's $\beta$-function & Definition \ref{mbf}
	 \\
	 $\alpha$ & Mather's $\alpha$-function, convex conjugate of $\beta$ & Theorem \ref{thm: caustic length lazutkin expansion}
	 \\
	 $\lsp, \lsp_{p,q}$ & Length spectrum & Definition \ref{def:LSP}, Section \ref{Connection with Laplace spectrum}
	 \\
	 $\mls_\Omega$ & Marked length spectrum of $\Omega$ & Definition \ref{def:LSP}
	 \\
	 $\mathcal{M}(\Omega)$ & Space of all Birkhoff billiard tables which are marked length isospectral to $\Omega$ & Section \ref{subsec:Length spectra}
	 \\
	 $t_{p,q}, T_{p,q}$ & Minimal and maximal lengths of type-$(p,q)$ orbits & Equation \eqref{eq:tpqTpq}
	 \\
	 $\Gamma, \Gamma_\omega, \Gamma_Q$ & A caustic (of rotation number $\omega$ or Lazutkin parameter $Q$) & Section \ref{subsec: Caustics}
	 \\
	 $Q$ & Lazutkin parameter of a caustic & Definition \ref{laz}
	 \\
	 $|\Gamma|$ & Length of a caustic & Section \ref{subsec: Caustics}
	 \\
	 $\iota_k$ & Nonzero constants in the normalization of $\I_m$ & Equations \eqref{SIK} and \eqref{eq:iotavalues}
	 \\
	 $\theta$ & Curvature coordinate, tangent angle & Section \ref{curvecoord}, Figure \ref{Billiard Table}
	 \\
	 $\zeta$ & Interpolating Hamiltonian for the billiard map & Definition \ref{def:interpolating Hamiltonian}
	 \\
	 $X_\zeta$ & Hamiltonian vector field of $\zeta$ & Equation \eqref{eq:Hamiltonian vector field}
	 \\
	 $\I(t), \I_m$ & Action integral and Marvizi-Melrose invariants & Definition \ref{action}
	 \\
	 $\RR_m$ & Remainder in $\I_m$ & Theorem \ref{main}
	 \\
	 $\mathcal{H}_{D,i}$ & $(D,i)$-harmless remainders & Definition \ref{def:harmless}
	 \\
	 $\zeta_i$ & $i^{\text{th}}$ term in $\lambda$ expansion of $\zeta$ & Equation \eqref{eq:taylor exp of zeta}
	 \\
	 $A_M(s)$ & Taylor coefficient of $s^+$ with respect to $\lambda$ & Equation \eqref{eq: splus in terms of s and lambda}
	 \\
	 $M$ & Index, usually odd; $M = 2m -1$ & Section \ref{Small lambda asymptotics}
	 \\
	 $\Theta_{m}[\zeta]$ & Integrand of $\I_m$ in terms of the jet of $\zeta$ & Proposition \ref{intinv}
	 \\
	 $\wt \zeta_i$ & Normalized coefficient $\zeta_1^{-1} \zeta_i$ & Proof of Proposition \ref{intinv}, Equation \eqref{eq:bm}
	 \\
	 $L, L_i$ & $L = X_\zeta^2$, $L_i$ are individual components & Definition \ref{L}
	 \\
	 $\z_K$ & $\z_K = \frac{(-1)^K}{K!} X_\zeta^K s$, a term in the Lie series expansion of $s^+$ with respect to $\lambda$ & Definition \ref{Zk}
	 \\
	 $K$ & Index, usually odd; $K = 2k+1$ & Definition \ref{Zk}, Section \ref{sec: computing AM Algebraically}
	 \\
	 $\bm{\sigma}$ & A $k$-letter word on the letters $\{1,2,3,4,5,6,7,8\}$ & Section \ref{sec: computing AM Algebraically}, Section \ref{sec: quadratic terms of maximal differential degree}
	 \\
	 $L_{\bm{\sigma}}$ & Composition of operators corresponding to $\bm{\sigma}$ & Equation \eqref{eq:wordcompositions}, Section \ref{sec: quadratic terms of maximal differential degree}
	 \\
	 $\Lambda_M$ & Extraction operator for $\lambda^{M/2}$ coefficient & Definition \ref{notationdef}
	 \\
	 $\RR_M^A, \mathcal{R}_M^\zeta, \mathcal{R}_M^c$ & Remainders & Section \ref{Small lambda asymptotics}, Section \ref{Integral invariants}
	 \\
	 $w_\zeta$ & $\zeta$-weight & Definition \ref{def:zeta weight}
	 \\
	 $\deg_{\d, \kappa}, \deg_{\d, \zeta}$ & Differential degree with respect to $\ka$ or $\zeta$ & Definition \ref{diffdeg}, Definition \ref{def: multi-fold diff degree}
	 \\
	 $\upsilon_M$ & Remainder & Proposition \ref{ZA}
	 \\ 
	 $\Upsilon_M$ & Remainder & Lemma \ref{maxind}
	 \\
	 $\zr$ & Riemann $\zeta$-function & Theorem \ref{Linear coeff}
	 \\
	 $\msfz_k$ & $\zr(2k)$ & Equation \eqref{eq:msfzk}
	 \\
	 $S, D, S_i, D_1$ & Single and double sums in the formula for $c_m$ & Section \ref{subsec: nonvanishing of the leading order coefficient}
\end{longtable}

\bibliographystyle{alpha}
\bibliography{BetaSource}

\end{document}